\documentclass[letterpaper,11pt]{amsart}
\usepackage{amsmath,amssymb,amsthm,pinlabel,tikz,hyperref,mathrsfs,color, thmtools}
\usepackage{verbatim}
\usepackage{float}
\usepackage{caption}
\usepackage{subcaption}
\usepackage{enumitem}
\newcommand{\nc}{\newcommand}
\nc{\dmo}{\DeclareMathOperator}
\dmo{\ra}{\rightarrow}
\dmo{\Prob}{\mathbb{P}}
\dmo{\E}{\mathbb{E}}
\dmo{\N}{\mathbb{N}}
\dmo{\Z}{\mathbb{Z}}
\dmo{\inte}{int}
\dmo{\Q}{\mathbb{Q}}
\dmo{\R}{\mathbb{R}}
\dmo{\C}{\mathcal{C}}
\dmo{\X}{\mathcal{X}}
\dmo{\U}{\mathcal{U}}
\dmo{\T}{\mathcal{T}}
\dmo{\F}{\mathcal{F}}
\dmo{\AC}{\mathcal{AC}}
\dmo{\w}{\omega}
\dmo{\MIN}{\mathcal{MIN}}
\dmo{\Mod}{Mod}
\dmo{\PMod}{PMod}
\dmo{\PMF}{\mathcal{PMF}}
\dmo{\Mat}{Mat}
\dmo{\supp}{supp}
\dmo{\UE}{\mathcal{UE}}
\dmo{\vol}{vol}
\dmo{\B}{B}
\dmo{\PB}{PB}
\dmo{\PR}{PSL(2,\mathbb{R})}
\dmo{\GL}{GL(k, \mathbb{C})}
\dmo{\SL}{SL(2, \mathbb{Z})}
\dmo{\Isom}{Isom}
\dmo{\Aut}{Aut}
\dmo{\RP}{\mathbb{R} \mathrm{P}}
\dmo{\I}{\mathcal{I}}
\dmo{\el}{\ell_{\C}}
\dmo{\NN}{\mathcal{N}}
\dmo{\rk}{rank}
\dmo{\tr}{tr}
\dmo{\llangle}{\langle\langle}
\dmo{\rrangle}{\rangle\rangle}
\dmo{\Unif}{Unif}
\dmo{\Out}{Out}
\dmo{\Homeo}{Homeo}
\dmo{\C*}{\mathcal{C}^{\dagger}}
\dmo{\sumRho}{\mathcal{N}}
\dmo{\stopping}{\vartheta}
\dmo{\diam}{\operatorname{diam}}

\usetikzlibrary{decorations.markings, patterns}
\tikzset{->-/.style={decoration={
  markings,
  mark=at position #1 with {\arrow{>}}},postaction={decorate}}}

\nc{\nt}{\newtheorem}

\nt{theorem}{Theorem}

\newtheorem{thm}{{\bf Theorem}}[section]

\newtheorem{lem}[thm]{{\bf Lemma}}
\newtheorem{cor}[thm]{{\bf Corollary}}
\newtheorem{prop}[thm]{{\bf Proposition}}
\newtheorem{fact}[thm]{Fact}

\newtheorem{claim}[thm]{Claim}

\newtheorem{prob}[thm]{Problem}

\newtheorem{question}[thm]{Question}
\newtheorem{conj}[thm]{Conjecture}
\newtheorem{dfn}[thm]{Definition}
\numberwithin{equation}{section}
\newtheorem{obs}[thm]{Observation}

\title{Percolation in acylindrically hyperbolic groups}

\date{\today}
\author{Inhyeok Choi}

\address{%
		School of Mathematics, KIAS\\
		85 Hoegi-ro, Dongdaemun-gu, Seoul 02455, South Korea
}
\email{
        inhyeokchoi48@gmail.com
        }

\author{Donggyun Seo}

\address{%
		Research Institute of Mathematics, Seoul National University, Seoul, Korea}
\email{
        seodonggyun@snu.ac.kr, seodonggyun7@gmail.com
        }
\begin{document}
\begin{abstract}
Let $G$ be an acylindrically hyperbolic group. We prove that Bernoulli bond percolation on every Cayley graph of $G$ has a nonuniqueness phase, in which there are infinitely many infinite clusters. This generalizes Hutchcroft's result for Gromov hyperbolic graphs to relatively hyperbolic groups, mapping class groups and rank-1 CAT(0) groups for example.


\noindent{\bf MSC classes:} 20F67, 57K20, 60K35
\end{abstract}

\maketitle

%
%

\section{Introduction}

In geometric group theory, groups are often studied as a geometric object. Some groups resemble $\Z^{d}$ while some others resemble free groups and surface groups. In between them there is a varying degree of hyperbolicity. Here, the notion of hyperbolicity can either be phrased internally using the Dehn function, small cancellation or the prevalence of Morse elements, or by using the group action on hyperbolic spaces, e.g., word hyperbolicity, relative hyperbolicity, hierarchical hyperbolicity and acylindrical hyperbolicity. For an overview in this aspect, we refer to M. Bestvina's survey \cite{bestvina0groups}. There have been efforts to study these hyperbolicity of groups  by means of stochastic processes such as random walks and Markov chains (\cite{sisto2018contracting}, \cite{maher2018random}, \cite{mathieu2020deviation}, \cite{goldsborough2021markov}).

On the other hand, groups naturally arise in probability theory as sources of many homogeneous graphs with vertex-transitive automorphism group. In this paper, we study percolation in groups. It was classically studied for Euclidean lattices $\Z^{d}$ in relation to physical situations where liquid passes through a porous medium. I. Benjamini and O. Schramm considered its generalization to Cayley graphs of groups and sketched the general landscape of the expected phenomena \cite{benjamini1996percolation}. See also papers by I. Benjamini, R. Lyons, Y. Peres and O. Schramm (\cite{benjamini1999critical}, \cite{benjamini1999group-invariant}) regarding the critical percolation in general groups.

Given a connected, locally finite (simplicial) graph $\Gamma$, \emph{Bernoulli bond percolation} on $\Gamma$ is defined by endowing independent Bernoulli random variables with expectation $p$ to the edges. Edges whose Bernoulli RV takes value 0 are deleted, and those with Bernoulli RV taking value 1 are retained. We can then ask how the connected components, i.e., clusters, of the resulting random subgraph $\Gamma[p]$ are shaped. To this end, we define two parameters, called the \emph{critical parmeter}
\[
p_{c} = p_{c}(\Gamma):= \inf \left\{ p \in [0, 1] : \textrm{$\Gamma[p]$ contains an infinite cluster almost surely}\right\}.
\]
and the \emph{uniqueness threshold}
\[
p_{u} = p_{u}(\Gamma):= \inf \left\{ p \in [0, 1] : \textrm{$\Gamma[p]$ contains a unique infinite cluster almost surely}\right\}.
\]
See Subsection \ref{subsection:perc} for further basics of the percolation theory.

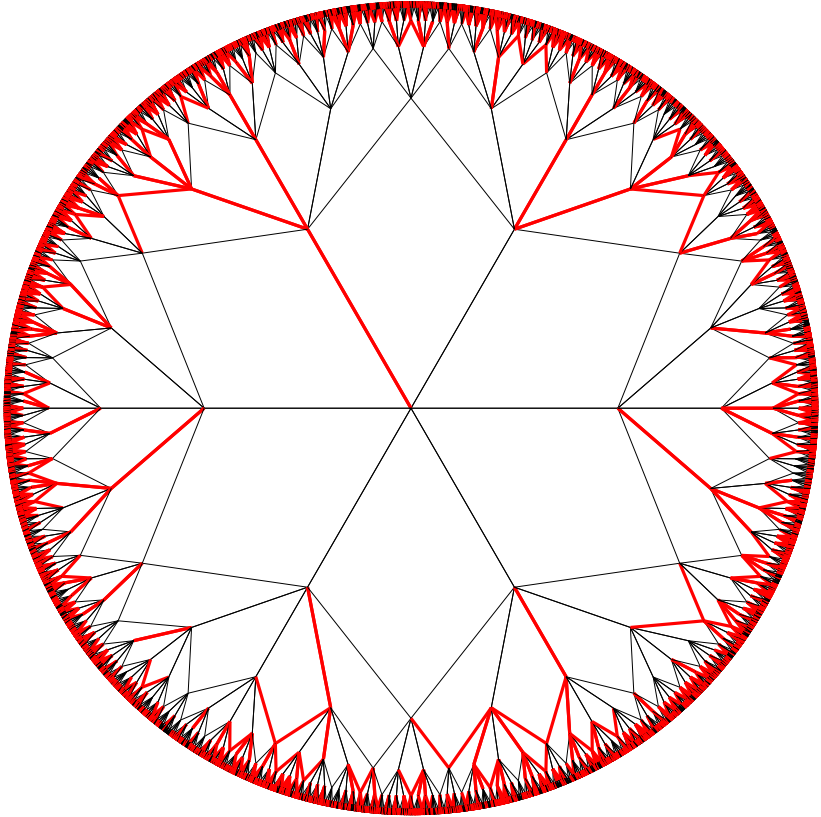
\begin{figure}
\begin{tikzpicture}[very thin, scale=0.55]
	\pgfmathsetmacro{\r}{10} 
	\pgfmathsetmacro{\depth}{6} 
	\pgfmathsetmacro{\prob}{0.387} 
  
  
	\foreach \d in {1, ..., \depth}
	{
		\pgfmathsetmacro{\l}{int(max(round(sqrt(3)*((2+sqrt(3))^(\d-1)-(2-sqrt(3))^(\d-1))), 1))}
		\pgfmathsetmacro{\k}{int(round(sqrt(3)*((2+sqrt(3))^\d-(2-sqrt(3))^\d)))}
		\foreach \x in {1, ..., \k}
		{
			\pgfmathsetmacro{\p}{random()}
			\ifdim \p pt > \prob pt
				\draw ({(round((\x/\k)*\l+0.13684)/\l)*360}:{(1-(1/2)^(\d-1))*\r}) -- ({(\x/\k)*360}:{(1-(1/2)^\d)*\r});
			\else
				\draw [red, very thick] ({(round((\x/\k)*\l+0.13684)/\l)*360}:{(1-(1/2)^(\d-1))*\r}) -- ({(\x/\k)*360}:{(1-(1/2)^\d)*\r});
			\fi
			
			\ifdim \p pt > \prob pt
				\draw ({(round((\x/\k)*\l-0.1309)/\l)*360)}:{(1-(1/2)^(\d-1))*\r}) -- ({(\x/\k)*360}:{(1-(1/2)^\d)*\r});
			\else
				\draw [red, very thick] ({(round((\x/\k)*\l-0.1309)/\l)*360)}:{(1-(1/2)^(\d-1))*\r}) -- ({(\x/\k)*360}:{(1-(1/2)^\d)*\r});
			\fi
		}
		}
		
\end{tikzpicture}
\caption{Percolation in the Cayley graph of a surface group.}
\end{figure}

Benjamini and Schramm posed several conjectures regarding percolation in the Cayley graphs of groups beyond $\mathbb{Z}^{d}$. Among them is the following:

\begin{conj}[{\cite[Conjecture 6]{benjamini1996percolation}}]\label{conj:nonamen}
A connected, locally finite, quasi-transitive graph  $\Gamma$  is non-amenable if and only if $p_{c}(\Gamma) < p_{u}(\Gamma)$.
\end{conj}

See \cite{haggstrom2006uniqueness} for an overview of this conjecture.

Let us list some facts for the context. Given a connected graph $\Gamma$, we have $0 \le p_{c} \le p_{u} \le 1$ by definition. It is a fact that for each $p > p_{c}$, the random graph $\Gamma[p]$ almost surely has an infinite cluster. When $\Gamma$ is quasi-transitive in addition, for each $p > p_{u}$ the random graph $\Gamma[p]$ almost surely has a unique infinite cluster. This is due to O. H{\"a}ggstr{\"o}m and Y. Peres \cite{haggstrom1999monotonicity} for unimodular cases, and due to R. H. Schonmann \cite{schonmann1999stability} in general (see also \cite{haggstrom1999percolation}).  Lastly, C. M. Newman and L. S. Schulman proved in \cite{newman1981infinite} for quasi-transitive $\Gamma$ that, for each $p \in (0, 1)$ there exists $N_{\infty}(p) \in \{0, 1, +\infty\}$ such that the number of infinite clusters in $\Gamma[p]$ is almost surely $N_{\infty}(p)$. 

Hence for quasi-transitive graphs, $N_{\infty}(p) = 0$ almost surely for $p < p_{c}$, $N_{\infty}(p) = +\infty$ almost surely for $p_{c} < p < p_{u}$ and $N_{\infty}(p) = 1$ almost surely for $p > p_{u}$. In particular, if $p_{c} < p_{u}$ then there exists (uncountably many) $p$ such that $\Gamma[p]$ has infinitely many infinite clusters almost surely.

Now, for \emph{non-amenable} quasi-transitive graphs, it is known that $N_{\infty}(p_{c}) = 0$ almost surely. This is due to I. Benjmaini, R. Lyons, Y. Peres and O. Schramm \cite[Theorem 1.1]{benjamini1999critical} and is generalized by T. Hutchcroft to graphs with exponential growth \cite[Theorem 1.2]{hutchcroft2016critical}. Hence, for a non-amenable quasi-transitive graph, $p_{c} < p_{u}$ if and only if $N_{\infty}(p) = +\infty$ for some (countably many) $p$'s.

Let us go back to the conjecture. The equality $p_{c}(\Z^{d}) = p_{u}(\Z^{d})$ was first achieved in a deep theorem by M. Aizenman, H. Kesten and C. M. Newman \cite{aizenman1987uniqueness}, and a simpler method was later obtained by R. M. Burton and M. Keane  \cite{burton1989density}. A. Gandolfi, M. S. Keane and C. M. Newman observed in \cite{gandolfi1992uniqueness} that Burton and Keane's method generalizes to amenable graphs. Hence, the only nontrivial direction is the ``only if" direction. A significant breakthrough was made by T. Hutchcroft, who showed the conjecture for non-amenable quasi-transitive graphs that admit an action by a non-unimodular group \cite{hutchcroft2020nonuniqueness}. We note that the first example of a quasi-transitive graph $\Gamma$ for which all the three cases—$N_{\infty} = 0$, $N_{\infty} = \infty$ and $N_{\infty} = 1$—take place was the direct product of trees and $\Z$, which admits a non-unimodular automorphism group \cite{grimmett1990percolation}.

Hence, the remaining case is non-amenable graphs with unimodular automorphism groups. The most natural examples in this category are Cayley graphs of non-amenable groups.

We thus focus on Cayley graphs of non-amenable groups. Let $G$ be a group with a finite generating set $S$. The Cayley graph $\Gamma(G, S)$ consists of the vertex set $G$ and the edge set $\{ \overline{vw} : \exists s \in S \cup S^{-1} : v = ws\}$. The group $G$ naturally acts as graph automorphisms and the action is vertex transitive, i.e., for each $v, w \in V(\Gamma) = G$ there exists $g \in G$ such that $gv = w$. The action is also vertex-faithful, i.e., each vertex has trivial stabilizer.

A strong evidence for Conjecture \ref{conj:nonamen} is given by I. Pak and T. Smirnova-Nagnibeda \cite{pak2000on-non-uniqueness}, who proved that every non-amenable group has a Cayley graph for which $p_{c} < p_{u}$. Furthermore, A. Nachmias and Y. Peres proved that Cayley graphs of non-amenable groups with high girth relative to the spectral radius satisfy Conjecture \ref{conj:nonamen} \cite{nachmias2012non-amenable}. One can then ask if there are groups all of whose \emph{every} Cayley graph satisfy Conjecture \ref{conj:nonamen}.

In this regard, Benjamini and Schramm proved that every transitive, nonamenable, planar graph with one end has a nonuniqueness phase \cite{benjamini2001percolation}. This generalizes S. Lalley's earlier work about Fuchsian groups \cite{MR1614583}. We also note that the nonuniqueness is known by D. Gaboriau \cite{MR2221157} for graphs admitting nonconstant harmonic Dirichlet functions and R. Lyons (\cite{MR1757952}, \cite{MR3009109}) for every Cayley graph of groups admitting an action with cost $> 1$.

Beyond planar graphs, Hutchcroft showed that the non-uniqueness phase exists for every Cayley graph of non-amenable word hyperbolic groups \cite{hutchcroft2019percolation}. In fact, Hutchcroft proved the result for more general quasi-transitive Gromov hyperbolic graphs, by using the Bonk-Schramm embedding of such graphs into a real hyperbolic space $\mathbb{H}^{d}$ \cite{bonk2000embeddings}. See also J. Czajkowski's independent work on certain groups acting on $\mathbb{H}^{3}$ \cite{czajkowski2024non-uniqueness}.

The main point of this paper is to generalize Hutchcroft's result to other non-amenable groups. Namely, we have:

\begin{theorem}\label{thm:main}
Let $G$ be an acylindrically hyperbolic group and let $\Gamma$ be its Cayley grpah. Then we have $p_{c}(\Gamma) < p_{u}(\Gamma)$; in particular, there exist uncountably many $p \in (0, 1)$ such that $\Gamma[p]$ has infinitely many infinite clusters.
\end{theorem} 

Acylindrically hyperbolic groups encompass word hyperbolic groups, relatively hyperbolic groups and many other groups that act on a Gromov hyperbolic space in a nontrivial way. These groups have shown to exhibit interesting dynamical and group-theoretical properties, (\cite{bestvina2002bounded}, \cite{hamenstadt2008bounded}, \cite{dahmani2017hyperbolically}), as well as probabilistic behaviour (\cite{sisto2018contracting}, \cite{mathieu2020deviation}, \cite{choi2025acylindrically}), that are shared with word hyperbolic groups. We list some examples of acylindrically hyperbolic groups:\begin{itemize}
\item  (non-elementary) relatively hyperbolic groups;
\item non-elementary Kleinian groups (possibly with $\mathbb{Z}^{d}$ subgroups);
\item free products of nontrivial groups;
\item the mapping class group of a finite-type hyperbolic surface;
\item the outer automorphism group $\Out(F_{n})$ of the free group $F_{n}$ \cite{bestvina2014hyperbolicity};
\item the automorphism group $\Aut(G)$ of a hyperbolic group $G$ (\cite{genevois2019negative}, \cite{genevois2021acylindrical});
\item rank-1 CAT(0) groups such as irreducible CAT(0) cubical groups \cite{bestvina2009higher}, 
\item many Artin groups and 3-manifold groups \cite{hagen2024extra-large}, \cite{minasyan2015acylindrical}
\end{itemize}
We refer readers to D. Osin's survey \cite{osin2016acylindrically} for more details.

Fix a vertex $v \in G$. We denote by $\chi_{p}(v)$ the expected size (the number of vertices) of the cluster of $v$. This does not depend on the choice of $v$ in the case of Cayley graphs, so we will drop $v$ and write $\chi_{p}$. Note that $\chi_{p} < +\infty$ for $p < p_{c}$ and $\chi_{p} \gtrsim (p_{c} - p)^{-1}$ are by M. Aizenman and C. M. Newman \cite[Proposition 3.1]{aizenman1984tree}. In the course of the proof, we show that $\chi_{p} \lesssim (p_{c} - p)^{-1}$, following Hutchcroft's criterion. This is one example of the so-called \emph{mean-field critical behavior} of the percolation in acylindrically hyperbolic groups. More consequences follow from the following fact:

\begin{theorem}\label{thm:mainDelta}
Let $G$ and $\Gamma$ be as in Theorem \ref{thm:main}. Then at $p_{c} = p_{c}(\Gamma)$, we have $\nabla_{p_{c}}(\Gamma) < +\infty$.
\end{theorem} 
See Subsection \ref{subsection:Hutchcroft} for the definition of the triangle diagram $\nabla_{p_{c}}$. This notion is introduced by Aizenman and Newman to study various mean-field critical behavior.

In fact, the triangle condition $\nabla_{p_{c}}<+\infty$ is derived from yet another fact, namely, the $L^{2}$-boudedness of an operator associated with the transition probability. In \cite[Section 2]{hutchcroft2019percolation}, Hutchcroft considered a quantity $p_{2 \rightarrow 2}(\Gamma)$ for a connected, locally finite, quasi-transitive graph $\Gamma$. This quantity has a property that $p_{2\rightarrow 2} \le p_{u}$ and $\nabla_{p} < +\infty$ for each $p < p_{2 \rightarrow 2}$. Hence, once $p_{c} < p_{2 \rightarrow 2}$ holds, then both the triangle condition $\nabla_{p_{c}} < +\infty$ and the inequality $p_{c} < p_{u}$ follow: \[
\big[ \textrm{$L^{2}$-boundedness} : p_{2}(\Gamma) < p_{2 \rightarrow 2}(\Gamma) \big] \,\,\Rightarrow \,\, [\nabla_{p_{c}}(\Gamma) <+\infty ] \,\,\wedge  \,\, [p_{c}(\Gamma) < p_{u} (\Gamma)].
\]
We will not explain the $L^{2}$-boundedness in detail. Nonetheless, the $L^{2}$-boundedness and the triangle condition guarantee several mean-field critical behaviors. We summarize them in Subsection \ref{subsection:Hutchcroft}. We refer to \cite{hutchcroft2019percolation} and \cite{hutchcroft2020the-l2-boundedness} for a deeper story.

\subsection{The hitchhicker's guide to the nonuniqueness}

This paper is concerned with probabilistic phenomena on geometric objects that entail hyperbolicity. Luckily, probabilistic ingredients were already given by Hutchcroft \cite{hutchcroft2019percolation}. Namely, the gap $p_{c} < p_{u}$ and the triangle diagram bound $\nabla_{p_{c}} < +\infty$ follow from Equation \ref{eqn:hutchcroftGamma1} and \ref{eqn:hutchcroftGamma2}. Theorem \ref{thm:hutchcroft1plus2} and \ref{thm:hutchcroftIotaGen} provide a way to guarantee these equations. The rest of the paper will focus on the proof that acylindrically hyperbolic groups satisfy the assumptions of Theorem \ref{thm:hutchcroft1plus2} and \ref{thm:hutchcroftIotaGen}.

Thus, for readers familiar with the theory of acylindrically hyperbolic groups, the quickest way to read this paper is as follows: \begin{enumerate}
\item read Definition \ref{dfn:barrier} and \ref{dfn:branching}, 
\item read Theorem \ref{thm:hutchcroft1plus2} and Theorem \ref{thm:hutchcroftIotaGen}, 
\item read Section \ref{section:supporting} and study Proposition \ref{prop:supportingWPD},  \item read Section \ref{section:branching} and study Proposition \ref{prop:barrierContracting}, and 
\item read Section \ref{section:barrierAcyl} and combine Proposition \ref{prop:magicAcyl}, \ref{prop:NFRough} and \ref{prop:DEThenEPrime}. 
\end{enumerate}

Notwithstanding, we recommend the readers to read from Section \ref{section:probPre} to Section \ref{section:barrierAcyl}. This is because:
\begin{itemize}
\item Section \ref{section:probPre} contains the basics of percolation theory and overview of Hutchcroft's theory. This helps understand Theorem \ref{thm:hutchcroft1plus2} and Theorem \ref{thm:hutchcroftIotaGen}.
\item Subsection \ref{subsection:plans} describes in detail the intuition behind our strategy.
\item Section \ref{section:hypPre} provides necessary hyperbolic geometry.
\item Section \ref{section:properMagic} gives a different proof of Hutchcroft's hyperbolic magic lemma without using Benjamini-Schramm's Euclidean magic lemma nor Bonk-Schramm's embedding theorem. This is not needed for general acylindrically hyperbolic group but is an important prototype containing essential ideas.

Moreover, when restricted to \emph{groups acting properly on a Gromov hyperbolic space}, Proposition \ref{prop:magicProper} and \ref{prop:barrierContracting} are sufficient for the main theorem. Hence, readers who are mainly interested in relatively hyperbolic groups (such as non-elementary Kleinian groups and free products of groups) may read Section \ref{section:properMagic} and Section \ref{section:branching} only.
\end{itemize}

This paper is mostly self-contained but hides two secret ingredients. First, Hutchcroft's approach is eventually based on the analysis of the transfer operator. We invite readers to \cite[Section 2]{hutchcroft2019percolation} for details of operator analysis. Second, the argument for barriers (Proposition \ref{prop:barrierContracting} and \ref{prop:DEThenEPrime}) can be explained by the fact that acylindrically hyperbolic groups act on a quasi-tree that enjoys the bottleneck property, which arises from Bestvina-Bromberg-Fujiwara's construction (\cite{bestvina2015quasi}, \cite{MR4057354}). 

\subsection{Discussion}

Our result in this article can be seen as a positive response to Benjamini--Schramm conjecture in the realm of acylindrically hyperbolic groups.
As we see, this provokes a variety of problems that relate percolation theory to geometric group theory.
Here we introduce some of the problems that have been chosen based on our preferences.

In Theorem \ref{thm:main}, we prove that nonuniqueness phase exists for every Cayley graph of certain class of groups. However, such an argument cannot be directly generalized to general finitely generated groups.

\begin{question}
Is the existence of non-uniqueness phase a group-invariant? That is to say, for a finitely generated group $G$, if there exists a finite generating set $S$ such that there exists a nonuniqueness phase for $Cay(G, S)$, does the same phenomenon happens for all Cayley graph of $G$?
\end{question}

If this is the case, Pak and Smirnova-Nagnibeda's theorem will settle Benjamini-Schramm's conjecture.

As we mentioned, many groups are known to be acylindrically hyperbolic.
Nevertheless, on the other hand, we have known a plenty of groups that are nonamenable and not acylindrically hyperbolic. In particular, we note that Theorem \ref{thm:main} deals with no \emph{products}. Given a non-amenable group, a cheap recipe to create a new non-amenable group is to take product with another group. Our method does not handle such groups, including $F_{2} \times \Z$. Hence, an innocent question follows: 

\begin{question}\label{ques:f2z}
Does every Cayley graph of $F_{2} \times \Z$ have a nonuniqueness phase?
\end{question}

This question can be viewed from different perspectives, as we now elaborate.

A group $G$ is called \emph{free-by-cyclic} if the following short exact sequence \[ 1 \to F \to G \xrightarrow{\phi} \mathbb{Z} \to 1 \] holds for some free group $F$.
The easiest example is $F \times \mathbb{Z}$.
Genevois and Horbez \cite{MR4503954} and Ghosh \cite{MR4541452} show that a free-by-cyclic group $G$ is acylindrically hyperbolic if and only if $\phi$ is of infinite order.
Currently, there is no strategy to prove the non-uniqueness of infinite clusters when $\phi$ is of finite order.
We note that generalizing Grimmett and Newman \cite{grimmett1990percolation}'s study on $T_{k} \times \mathbb{Z}^{d}$ for a regular tree $T_{d}$ with valency $k \gg 1$, Hutchcroft \cite{hutchcroft2020nonuniqueness} showed that Benjamini-Schramm's conjecture holds for all quasi-transitive graphs that admit a nonunimodular automorphism group, which includes $T_{k} \times \mathbb{Z}^{d}$ for $k \ge 3$.

\begin{question}
Does Benjamini--Schramm conjecture hold for all free-by-cyclic groups?
\end{question}

If $\Gamma=(V, E)$ is a simple graph and $\mathcal{G} = (G_v)_{v \in V}$ is a system of groups labelled by vertices of $\Gamma$, the \emph{graph product} $\Gamma\mathcal{G}$ of $\mathcal{G}$ is defined by \[ \Gamma\mathcal{G} := \left( *_{v \in V} G_v \right) / \llangle [g, h] \mid g \in G_u, h \in G_v, \{u, v\} \in E \rrangle. \]
In particular, if all groups in the system $\mathcal{G}$ are infinite cyclic, the corresponding graph product is known as a right-angled Artin group.
A. Minasyan and D. Osin \cite[Corollary 2.13]{minasyan2015acylindrical} show $\Gamma\mathcal{G}$ is acylindrically hyperbolic in many cases.
Precisely, $\Gamma\mathcal{G}$ satisfies one of the following:
\begin{enumerate}
\item $\Gamma$ is a join of two nontrivial subgraphs, that is, $\Gamma\mathcal{G}$ cannot be decomposed into a direct product of groups.
\item $\Gamma\mathcal{G}$ is virtually cyclic.
\item $\Gamma\mathcal{G}$ is acylindrically hyperbolic.
\end{enumerate}
The first case includes direct products of free groups where Benjamini--Schramm conjecture is still unknown.

\begin{question}
Does Benjamini--Schramm conjecture hold for all graph product of groups?
\end{question}

We are also interested in Benjamini--Schramm conjecture in the class of fundamental groups of (connected compact) $3$-dimensional manifolds, briefly, 3-manifold groups.
Minasyan--Osin \cite[Theorem 2.8]{minasyan2015acylindrical} reveal that most 3-manifold groups are acylindrically hyperbolic.
In detail, they show that every 3-manifold whose fundamental group is neither polycyclic nor acylindrically hyperbolic is an extension of an acylindrically hyperbolic group by an infinite cyclic group.
What we want to know is the following.

\begin{question}
Does Benjamini--Schramm conjecture hold for all $3$-manifold groups? In particular, for every Seifert fibered space $M$, does Benjamini--Schramm conjecture hold for $\pi_1(M)$?
\end{question}

Lastly, $F_{2} \times \Z$ serves as the simplest non-amenable, reducible CAT(0) cubical group. Contrary to general groups acting on a Hadamard manifold, the rank rigidity theorem is known for CAT(0) cube complexes by P-E. Caprace and M. Sageev \cite{caprace2011rank}. In particular, all irreducible, non-virtually cyclic CAT(0) cubical groups are subject to Theorem \ref{thm:main}. Thus, the nonuniqueness problem for percolation in CAT(0) cubical groups will be completely understood once it is understood for metric products. The groups $F_{2} \times \Z$ and $F_{2} \times F_{2}$ serve as barometers for this question. Let us summarize this section with the following question: \begin{question}
Does Benjamini--Schramm conjecture hold for CAT(0) cubical groups, or more generally, quasi-transitive median graphs? In particular, for a CAT(0) cubical group $G$, does $p_{c}(Cay(G, S)) = p_{u}(Cay(G, S))$ hold for every finite generating set $S$ of $G$ if and only if $G$ is virtually $\Z^{d}$?
\end{question}

\subsection*{Acknowledgement}

The authors thank Tom Hutchcroft, Mahan Mj and Alessandro Sisto for their comments on the earlier draft of the current manuscript, and Itai Benjamini for helpful discussion on percolation theory. In particular, Hutchcroft has informed the authors that Question \ref{ques:f2z} is still not answered, and have explained to the authors deeper aspects of percolation theory.

This work was initiated while the first author was visiting the June E Huh Center for Mathematical Challenges of KIAS. The first author is partly supported by the Mid-Career Researcher Program (RS-2023-00278510) through the National Research Foundation funded by the government of Korea.

The second author was supported by the National Research Foundation of Korea (NRF) grant funded by the Korea government(MSIT) 2021R1C1C2005938.

\section{Preliminaries I: Probability theory}\label{section:probPre}

We assume that readers are familiar with finitely generated groups and simplicial graphs. When a group $G$ is given a generating set $S$, we define the word metric $d_{S}$ as \[\begin{aligned}
d_{S}(g, h) &:= \inf \big\{ n \ge 0 : \exists s_{1}, \ldots, s_{n} \in S \cup S^{-1} [h=gs_{1} \cdots s_{n}] \big\}, \\
 \|g\|_{S} &:= d_{S}(id, g).
 \end{aligned}
\]

\subsection{Basics of percolation theory}\label{subsection:perc}

This subsection is intended as a quick introduction to percolation theory. We refer to \cite{grimmett1989percolation} for further details. Readers who are familiar with percolation theory or want to keep it as a blackbox can skip this subsection.

Let $\Gamma = (\mathcal{V}, \mathcal{E})$ be a connected simplicial graph. We will focus on the case that $\mathcal{V}$ is countable and $\Gamma$ has uniformly bounded valence. Let $0 \le p \le 1$. On the product space $\Omega = \{0, 1\}^{\mathcal{E}}$ indexed by edges of $\Gamma$, we can endow the product  of Bernoulli measures with expectation $p$. That means, $\Prob(\omega \in \Omega : \omega(e) = 1) = p$ for each $e$, and $\omega(e)$ and $\omega(e')$ are independent for $e \neq e'$. Each $\omega \in \Omega$ gives rise to a graph $\Gamma(\omega)$, which is the subgraph of $\Gamma$ after removing those edges $e$ with $\omega(e) = 0$.

When $p$ is small, we remove many edges in probability. Hence the clusters, the connected components of $\Gamma(\omega)$, are likely to be bounded. One can imagine that the expected It  is convenient to use the notation:\[
v \leftrightarrow_{\omega} w\,\, \Leftrightarrow \,\,\textrm{``$v$ and $w$ are connected in $\Gamma(\omega)$"}
\]
for $v, w \in \mathcal{V}$. Given $v \in \mathcal{V}$ we define the cluster \[
C_{\omega}(v) := \{ w \in \mathcal{V} : w \leftrightarrow_{\omega} v\}.
\]
Then by Fubini's theorem, we have \[
\E_{p} \#C_{\omega}(v) = \sum_{w \in \mathcal{V}} \Prob_{p}(v \leftrightarrow_{\omega} w).
\]
By convention, we will write \[
v \leftrightarrow_{\omega} \infty \,\,\Leftrightarrow \#C_{\omega}(v) = +\infty.
\]
(Recall that we focus on locally finite graphs.)

Note that the space $\Omega = \{0, 1\}^{\mathcal{E}}$ and the random graph $\Gamma(\omega) \subseteq \Gamma$ for $\omega \in \Omega$ are defined without reference to $p$. The parameter $p$ affects the underlying probability measure only. To express the role of $p$ more explicitly, we denote the random graph by $\Gamma[p]$ and the underlying measure by $\Prob_{p}$.

As we mentioned just before, one can ask if the expected size of clusters increase as $p$ increases. Furthermore, one can ask if the expected size of clusters depends on the choice of the root vertex. These questions can be answered using the following tools.

The space $\Omega$ is given a natural order: for $\omega, \omega' \in \Omega$ we write $\omega \le \omega'$ if $\omega(e) \le \omega(e')$ for each $e \in \mathcal{E}$. We say that an event $A \subseteq \Omega$ is \emph{increasing} if \[
\forall \omega, \omega' \in \Omega \big[ [ \omega \in A \wedge \omega \le \omega'] \Rightarrow \omega' \in A\big].
\]
Standard examples of increasing events include $\{\omega : v \leftrightarrow_{\omega} w\}$ for given $v, w\in \mathcal{V}$, or $\{\omega : v \leftrightarrow_{\omega} +\infty\}$ for a given $v \in \mathcal{V}$.
\begin{fact}\label{fact:increasing}
Let $A \subseteq \Omega$ be an increasing event. Then $\Prob_{p}(A) \le \Prob_{p'}(A)$ holds for each $0 \le p \le p' \le 1$.
\end{fact}

We now state the Harris-FKG inequality, which was first described by T. E. Harris \cite{harris1960a-lower} and later generalized by C. M. Fortuin, P. W. Kasteleyn and J. Ginibre \cite{fortuin1971correlation}: \begin{prop}\label{prop:FKG}[Harris-FKG]
Let $A, B\subseteq \Omega$ be increasing events. Then \[
\Prob_{p}(A \cap B) \ge \Prob_{p}(A) \cdot \Prob_{p}(B)
\]
for each $p$. 
\end{prop}

The Harris-FKG inequality can be used, for example, to show that the average size of clusters does not depend on the choice of the root vertex.

The Harris-FKG inequality looks like a generalization of the strict equality for independent events. One can ask if the reverse inequality also holds in certain circumstances. The BK inequality partially explains this.

Given an increasing event $A$ and $\omega \in A$, there can be a set $W \subseteq \{e \in \mathcal{E} : \omega(e) = 1\} \subseteq \mathcal{E}$ such that $1_{W} \subseteq A$ holds, i.e., \[
\forall \omega' \in \Omega \big[ \forall e \in W [\omega'(e) = 1] \Rightarrow \omega' \in A \big].
\] 
In this situation, we call $W$ a \emph{witness} for $A$ in $\omega$. If $W' \subseteq W$ are both witnesses for $E$ in $\omega$, we say that $W'$ is a \emph{sub-witness} of $W$.

For example, let $v, w \in \mathcal{V}$, let $\omega \in \Omega$ and suppose that there exists a path $(e_{1}, \ldots, e_{n})$ in $\Gamma(\omega)$ connecting $v$ to $w$. Then this path becomes a witness for $A := \{v \leftrightarrow w\}$ in $\omega$: $\{e_{1}, \ldots, e_{n}\}$ are all given value 1 by $\omega$, and for $\omega' \in \Omega$, $\omega'(e_{i}) = 1$ for each $i$ implies $\omega' \in A$.

Now, for two increasing events $A, B \subseteq \Omega$, we define another measurable set $A \circ B \subseteq \Omega$, called the \emph{disjoint occurrence} of $A$ and $B$, as follows: \[
A \circ B := \left\{ \omega \in \Omega : \begin{array}{c} \omega \in A \cap B, \,\, \exists \,\textrm{witness $W \subseteq \mathcal{E}$ for $A$ in $\omega$ and} \\
\exists\, \textrm{witness $W' \subseteq \mathcal{E}$ for $B$ in $\omega$ such that $W\cap W' = \emptyset$}\end{array} \right\}.
\]
For example, in Figure \ref{fig:disjtOcc} shows two different configurations in $A \cap B$, where $A := \{u \leftrightarrow v\}$ and $B := \{v \leftrightarrow w\}$. The left configuration lies in $A \circ B$, whereas the right one does not.

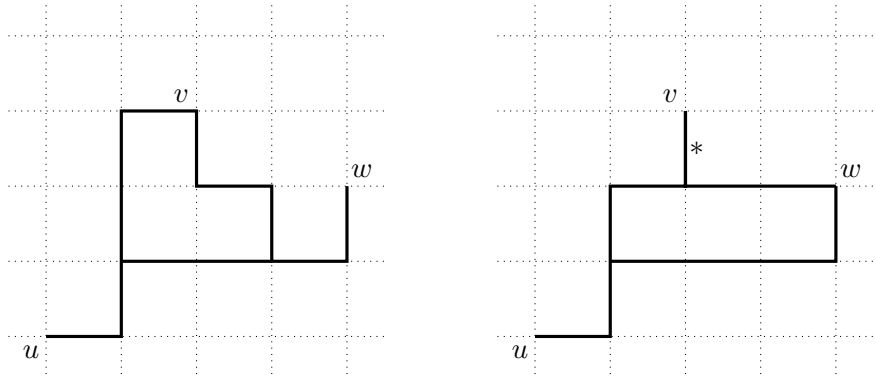
\begin{figure}
\begin{tikzpicture}

\begin{scope}
\foreach \i in {-2, ..., 2}{
\draw[dotted] (-2.5, \i) -- (2.5, \i);
\draw[dotted] (\i, -2.5) -- (\i, 2.5);
}
\draw (-2.2, -2.2) node {$u$};
\draw (-0.2, 1.2) node {$v$};
\draw (2.2, 0.2) node {$w$};

\draw[very thick] (-2, -2) -- (-1, -2) -- (-1, 1) -- (0, 1) -- (0, 0) -- (1, 0) -- (1, -1) -- (2, -1) -- (2, 0);
\draw[very thick] (-1, -1) -- (1, -1);

\end{scope}

\begin{scope}[shift={(6.5, 0)}]
\foreach \i in {-2, ..., 2}{
\draw[dotted] (-2.5, \i) -- (2.5, \i);
\draw[dotted] (\i, -2.5) -- (\i, 2.5);
}
\draw (-2.2, -2.2) node {$u$};
\draw (-0.2, 1.2) node {$v$};
\draw (2.2, 0.2) node {$w$};

\draw[very thick] (-2, -2) -- (-1, -2) -- (-1, 0) -- (0, 0) -- (0, 1) -- (0,0) --(2, 0);
\draw[very thick] (-1, -1) -- (2, -1) -- (2, 0);
\draw (0.15, 0.5) node {$\ast$};

\end{scope}

\end{tikzpicture}
\caption{Two configurations from percolation in $\mathbb{Z}^{2}$. In the left configuration, the path $u \rightarrow \uparrow \uparrow \uparrow \rightarrow v$ and $v \downarrow \rightarrow \downarrow \rightarrow \uparrow w$ are disjoint witnesses for $A:= \{u \leftrightarrow v\}$ and $B:= \{v \leftrightarrow w\}$, respectively.
In the right configuration, there are several witnesses for the events $A$ and $B$ in $\omega$, but none of them are disjoint; they all contain Edge $(\ast$).}
\label{fig:disjtOcc}
\end{figure}

We can now state the BK inequality, which is due to J. van den Berg and H. Kesten \cite{berg1985inequalities}. \begin{prop}\label{prop:BK}[BK inequality]
Let $A, B\subseteq \Omega$ be increasing events for which every witness has a finite sub-witness. Then for each $p$ we have \[
\Prob_{p}(A \circ B) \le \Prob_{p}(A) \cdot \Prob_{p}(B).
\]
\end{prop}

Often, we want a more precise information about the growth of $\Prob_{p}(A)$ for a given event $A$. For this it is beneficial to have a formula for derivatives of $\Prob_{p}(A)$. Russo's formula serves this purpose.

Given an increasing event $A \subseteq \Omega$ and $w \in \Omega$, we say that $e \in \mathcal{A}$ is \emph{pivotal} for the event $A$ if $\omega$ enters $A$ after turning $e$ on, and $\omega$ is excluded from $A$ by turning $e$ off. More formally, $e$ is pivotal for $A$ (given $\omega$) if $\omega^{e} \in A$ and $\omega_{e} \notin A$ for \[
\omega^{e}(f) := \left\{ \begin{array}{cc} 1 & f = e \\ \omega(f) & f \in \mathcal{E} \setminus \{e\} \end{array}\right., \quad \omega_{e}(f) := \left\{ \begin{array}{cc} 0 & f = e \\ \omega(f) & f \in \mathcal{E} \setminus \{e\}  \end{array}\right..
\]
We now record Russo's formula \cite{russo1981on-the-critical}: \begin{prop}\label{prop:Russo}
Let $A \subseteq \Omega$ be an increasing event. Then 
 \[
\left(\frac{d}{dp} \right)_{+} \Prob_{p}(A) \ge \sum_{e \in \mathcal{E}} \Prob_{p}(\textrm{$e$ is pivotal for $A$}) = \frac{1}{1-p} \sum_{e \in \mathcal{E}} \Prob_{p} (\textrm{$e$ is closed and pivotal for $A$}).
\]
\end{prop}
Here, $(d/dp)_{+}$ denotes the lower right Dini derivative.

We now focus on percolation in groups. For this purpose, let $G$ be a finitely generated group and let $\Gamma$ be its Cayley graph. We then consider the Bernoulli percolation on $\Gamma$. In this setting, $\Prob_{p}$ is $G$-ergodic: every $G$-invariant event occurs with probability 0 or 1.

We now define the \emph{critical parameter} for $\Gamma$: \[
p_{c} = p_{c}(\Gamma) := \inf\{ p \in [0, 1] : \Prob_{p}(\textrm{$C(id)$ is infinite}) > 0 \}.
\]
Then all clusters are almost surely finite for $p < p_{c}$. By ergodicity of $\Prob_{p}$ under the $G$-action, there is almost surely an infinite cluster for $p > p_{c}$.

We then define the \emph{uniqueness threshold} for $\Gamma$: \[
p_{u} = p_{u}(\Gamma) := \inf\{ p \in [0, 1] : \Prob_{p} (\textrm{there is a unique infinite cluster}) >0\}.
\]
For every value of $p$, the number $N_{\infty}(p)$ of infinite clusters is almost surely constant and is among $\{0, 1, \infty\}$. Furthermore, $N_{\infty}(p) = +\infty$ holds throughout $p_{c} < p < p_{u}$, and $N_{\infty}(p) = 1$ holds throughout $p > p_{u}$. What happens at $p = p_{c}$ and $p = p_{u}$ is mysterious: we refer the readers to \cite{hutchcroft2024percolation} for a summary of history and  recent progress.

\subsection{Overview of Hutchcroft's strategy} \label{subsection:Hutchcroft}

We now explain Hutchcroft's theory in \cite{hutchcroft2019percolation}. Throughout, $\Gamma$ will be a Cayley graph of a finitely generated group $G$. This graph is connected, has a uniformly bounded valency and is vertex-transitive. 

Recall that for a given parameter $0 \le p \le 1$ we defined the random graph $\Gamma[p]$ by randomly deleting edges from $\Gamma$. We define the two-point function \[
\tau_{p} (g, h) := \Prob_{p} ( g \leftrightarrow h) = \Prob_{p} (\exists \, \textrm{path connecting $g$ and $h$ in $\Gamma[p]$}).
\]
We abbreviate $\tau_{p}(id, g)$ by $\tau_{p}(g)$. Then $\tau_{p}(g, h) = \tau_{p}(g^{-1} h)$ for each $g, h \in G$. We now introduce the \emph{triangle diagram} \[
\nabla_{p} := \sup_{g \in G} \sum_{h, k \in G} \tau_{p} (g, h) \tau_{p}(h, k) \tau_{p}(k, g).
\]
We call the expected size of the identity cluster the \emph{susceptibility}: \[
\chi_{p} := \E_{p} [\#C(id)]= \sum_{g \in G} \tau_{p}(g).
\]
It is a fact that $\chi_{p} <+\infty$ for $0 \le p < p_{c}$ and $\lim_{p \nearrow p_{c}}\chi_{p} = +\infty$.

Let us now define \[
\iota_{p} := 1 - \sup \left\{ \frac{\sum_{g, h \in K} \tau_{p} (g, h)}{\chi_{p} \cdot \#A} : A \subseteq G\,\,\textrm{finite} \right\}.
\]
A naive counting shows that $\iota_{p} \ge 0$ always holds. It is however nontrivial to show that $\iota_{p}$ gets closer to $1$ as $p\nearrow p_{c}$, which is one of our main  goals.

We can now state:

\begin{thm}[{\cite[Proposition 2.7]{hutchcroft2019percolation}}]\label{thm:hutchcroft}
Let $\Gamma$ be a Cayley graph of a finitely generated group. If \begin{equation}\label{eqn:nablaPre}
\liminf_{p \nearrow p_{c}} \frac{p_{c} - p}{1-p} \chi_{p} \sqrt{1 - \iota_{p}^{2}} = 0,
\end{equation}
then $p_{c}(\Gamma) < p_{u}(\Gamma)$ and $\nabla_{p_{c}}(\Gamma) < +\infty$.

Moreover, we have the following mean-field critical behaviors: \begin{align}
\chi_{p} := \E_{p}(\#C(id)) &\asymp (p_{c} - p)^{-1} & p \nearrow p_{c}, \\
\E_{p} \left[\left( \#C(id) \right)^{k+1}\right] / 
\E_{p} \left[\left( \#C(id) \right)^{k}\right] &\asymp (p_{c} - p)^{-2} & k \ge 1, p \nearrow p_{c},\\
\Prob_{p}(\#C(id) = +\infty) &\asymp p-p_{c} & p \searrow p_{c}, \\
\Prob_{p_{c}}(\#C(id) \ge n ) &\asymp n^{-1/2} & n \nearrow +\infty, \\
\Prob_{p_{c}} (\textrm{intrinsic radius of $K_{id}$} \ge n) &\asymp n^{-1} & n \nearrow +\infty.
\end{align}
\end{thm}
To see how the triangle condition implies the mean-field critical behaviors, see \cite[Section 7]{hutchcroft2020nonuniqueness} and \cite{hutchcroft2022on-the-derivation} and the references therein. In fact, there are some remarkably precise estimates of the implied constants for $\asymp$ in the above display: see \cite{hutchcroft2022slightly} for example. 

In fact, Hutchcroft derives the triangle equation from Equation \ref{eqn:nablaPre} by means of the so-called \emph{$L^{2}$ boundedness}. We will not explain this operator-theortical part and refer the interested readers to \cite[Section 2]{hutchcroft2019percolation}. We will nonetheless record some consequences of the $L^{2}$ boundedness:

\begin{thm}[{\cite[Proposition 2.3]{hutchcroft2019percolation}, \cite{hutchcroft2020the-l2-boundedness},  \cite{hutchcroft2022slightly}, \cite{hutchcroft2024slightly}}]\label{thm:hutchcroftMean}
Let $\Gamma$ be a Cayley graph of a finitely generated group. If \[
\liminf_{p \nearrow p_{c}} \frac{p_{c} - p}{1-p} \chi_{p} \sqrt{1 - \iota_{p}^{2}} = 0,
\]
then we have the following: \begin{enumerate}
\item For $p \simeq p_{c}$, the two-point function $\tau_{p}(id, g)$ decays exponentially in $\|g\|_{S}$.
\item At the critical percolation, the intrinsic radius and the extrinsic radius exhibits similar critical behaviours. More precisely, we have \[
\Prob_{p_{c}} \big(\sup\{r : B_{S}(id, r) \subseteq C(id)\} \ge n\big) \asymp n^{-1} \asymp \Prob_{p_{c}}\big(\diam_{S}(C(id)) \ge n\big).
\]
Furthermore, the implied constant for $\asymp$ in the above display depends exponentially on $|p - p_{c}|$.
\item For $p \simeq p_{c}$, the intrinsic distance in $\Gamma[p]$ is not hugely distorted from the word metric in $\Gamma$: huge distortions are exponentially unlikely.
\item For $p \simeq p_{c}$ with $p > p_{c}$, the volume growth of the ball in $C(id)$ is purely exponential.
\end{enumerate}
\end{thm}

We now turn to the strategy for Equation \ref{eqn:nablaPre}. It will follow from \begin{align}\label{eqn:hutchcroftGamma1}\limsup_{p \nearrow p_{c}} (p_{c} - p) \chi_{p} &<+\infty, \\
\label{eqn:hutchcroftGamma2}
\lim_{p \nearrow p_{c}} \sup \left\{ \frac{\sum_{g, h \in A} \tau_{p} (g, h)}{\chi_{p} \cdot \#A} : A \subseteq G\,\,\textrm{finite} \right\} &=0.
\end{align}

In order to show Equation \ref{eqn:hutchcroftGamma1}, Hutchcroft proved the following for Gromov hyperbolic graph that admits a vertex-transitive action by a unimodular group. We restrict ourselves to the case of Cayley graphs.

\begin{prop}[Supporting Hyperplane Theorem, {\cite[Corollary 4.3]{hutchcroft2019percolation}}]\label{prop:hutchcroft1}
Let $G$ be a non-elementary word hyperbolic group with a finite generating set $S$. Then there exists $r>0$ such that the following holds.

For each finite set $A \subseteq G$ there exists $A' \subseteq A$ with $\#A' \ge \#A / 2$ such that for each $u \in A'$, there exists $v \in G$ with $ d_{S}(u, v) \le r$ such that $H_{G}(u, v)$ is a proper discrete halfspace with $A \subseteq H_{G}(u, v)$. 
\end{prop}

This is accompanied by:

\begin{prop}\label{prop:hutchcroft2}
Let $G$ be a non-elementary word hyperbolic group with a finite generating set $S$. Then there exists $R>0$ such that, for each $u, v \in G$ that gives rise to a proper discrete halfspace $H_{G}(u, v)$, there exists $g \in G$ with $\|g\|_{S} \le 2d_{S}(u, v) + R$ such that $H_{G}(u, v)$ and $g H_{G}(u, v)$ are disjoint.
\end{prop}

The precise shape of $H_{G}(u, v)$ is not important. We only need that $H_{G}(u, v)$ is large enough to contain $A$, but also small enough such that some reasonably close translates of $H_{G}(u, v)$ do not overlap. Let us put them in a more abstract language:

\begin{thm}[{\cite[Section 5.1]{hutchcroft2019percolation}}]\label{thm:hutchcroft1plus2}
Let $\Gamma = Cay(G, S)$ be the Cayley graph of a finitely generated group $G$. Let $\mathscr{H} = \{H(g) : g \in G\}$ be a collection of subsets of $G$. Suppose that there exists $R>0$ such that the following holds: \begin{quote}
For each finite set $A \subseteq G$ there exists $A' \subseteq A$ with $\#A' \ge \#A / 2$ such that for each $a \in A'$, there exists $g, h \in G$ such that $\|g\|_{S}, \|h\|_{S} \le R$,  $A \subseteq a H(g)$ and $H(g) \cap h H(g) = \emptyset$. 
\end{quote}
Then Equation \ref{eqn:hutchcroftGamma1} holds for $\Gamma$.
\end{thm}

This is proven in \cite[Subsection 5.1]{hutchcroft2019percolation} for proper discrete halfspaces in $G$. We present Hutchcroft's proof in Appendix \ref{appendix:hutchcroft1} for completeness.

The proof of Equation \ref{eqn:hutchcroftGamma2} is more involved. Using Benjamini and Schramm's magic lemma for Euclidean spaces, Hutchcroft proved a magic lemma for real hyperbolic space $\mathbb{H}^{d}$:

\begin{prop}[hyperbolic magic lemma, {\cite[Proposition 4.1]{hutchcroft2019percolation}}]\label{prop:hutchcroftMagic}
Let $X$ be a closed convex set of $\mathbb{H}^{d}$ and let $Y$ be a coarsely dense and uniformly locally finite subset of $X$. Then for every $\epsilon >0$ there exists a constant $N(\epsilon)$ such that for every finite set $A \subset Y$ there exists a subset $A' \subseteq A$ with the following properties:

\begin{enumerate}
\item $\#A' \ge (1-\epsilon)\#A$.
\item For each $v \in A'$, there exists a pair of halfspaces $H_{1}, H_{2} \subseteq \mathbb{H}^{d}$ such that $d_{\mathbb{H}^{2}} (v, H_{1} \cup H_{2}) \ge \epsilon^{-1}$ and $\#\big(A \setminus (H_{1} \cup H_{2}) \big) \le N(\epsilon)$.
\end{enumerate}
\end{prop}

The precise shape of halfspaces $H_{i}$ in Proposition \ref{prop:hutchcroftMagic} is again not important, but we will have to impose some ``smallness" of $H_{i}$ in terms of $\epsilon$. More precisely, we need that $\E_{p} \# \{g \in C(id) : gx_{0} \in H_{i}\} \lesssim \epsilon \chi_{p}$. Let us introduce some terminology.

\begin{dfn}\label{dfn:barrier}
Let $G$ be a group with a finite generating set $S$. For subsets $A, B, C \subseteq G$, we say that $B$ is a \emph{$d_{S}$-barrier between $A$ and $C$} if every $d_{S}$-path $(g_{0}, g_{1}, \ldots, g_{n}) \subseteq G$ starting at $A$ (i.e., $g_{0} \in A$) and ending at $C$ (i.e., $g_{n} \in C$) intersects $B$ (i.e., $\exists i [g_{i} \in B]$).
\end{dfn}

We record Hutchcroft's observation about barriers:
\begin{lem}[{\cite[Proof of Lemma 5.4]{hutchcroft2019percolation}}]\label{lem:BKBarrierChi}
Let $G$ be a group with a finite generating set $S$. Let $A,B \subseteq G$ be such that $B$ is a $d_{S}$-barrier between $id$ and $A$. Then \[
\E_{p} \# \big( C(id) \cap A \big) \le \E_{p} \# \big( C(id)\cap B \big) \cdot \chi_{p}
\]
for each $0 \le p < p_{c}$.
\end{lem}

\begin{proof}
For each $a \in A$ and $b \in B$ we define $E_{b} := \{ \omega : id \leftrightarrow b\}$ and $F_{b, a} := \{ \omega : b \leftrightarrow a\}$. Then \[
\begin{aligned}
\sum_{a \in A, b \in B} \Prob_{p} (E_{b} \circ F_{b, a}) &\le 
\sum_{a \in A, b \in B} \Prob_{p} (E_{b}) \Prob_{p}(F_{b, a}) = \sum_{b \in B} \Prob_{p}(E_{b}) \cdot \sum_{a \in A} \Prob_{p} (F_{b, a}) \\
&\le \sum_{b\in B} \Prob_{p}(E_{b}) \cdot \E_{p} \#C(b) = \sum_{b \in B} \Prob_{p}(E_{b}) \cdot \chi_{p} = \E_{p} \# \big( C(id)\cap B \big) \cdot \chi_{p}.
\end{aligned}
\]
Meanwhile, for each $a \in A$ we claim that \[
\cup_{b \in B} (E_{b} \circ F_{b, a}) = \{\omega : id \leftrightarrow a\}.
\]
The inclusion ``$\subseteq$" is clear. Now for ``$\supseteq$", let $\omega \in \Omega$ be a configuration such that $id \leftrightarrow a$. Take a shortest path $P$ in $\Gamma(\omega)$ connecting $id$ and $a$, which does not revisit a vertex twice. Since $B$ is a $d_{S}$-barrier between $id$ and $A$, $P$ visits a vertex $b \in B$. Then the subpaths of $P$ between $id$ and $b$, and between $b$ and $a$, are disjoint (finite) witnesses for $E_{b}$ and $F_{b, a}$, respectively. Hence, $\omega \in E_{b} \circ F_{b, a}$ as desired.

In conclusion, we have \[\begin{aligned}
\E_{p} \#\big(C(id) \cap A \big) &= \sum_{a \in A} \Prob_{p} \{ id \leftrightarrow a\} \\
&\le \sum_{a \in A,  b\in B} \Prob_{p} (E_{b} \circ F_{b, a}) \le \E_{p} \# \big( C(id)\cap B \big) \cdot \chi_{p}.\qedhere
\end{aligned}
\]
\end{proof}

Having Lemma \ref{lem:BKBarrierChi} in hand, it is desirable to construct a barrier $B$ between the origin and a halfspace whose ``capacity" $\E_{p} \#(C(id) \cap B)$ is uniformly small for all $0 \le p < p_{c}$. For example, $B$ should not be the entire $G$; the susceptibility $\chi_{p} = \E_{p}\#C(id)$ tends to infinity as $p \nearrow p_{c}$. Likewise, $B$ should not contain an arbitrarily large $d_{S}$-metric ball. A geometric intuition is that if $B$ is a \emph{codimension 1} subset of the ambient set, then the portion of the cluster in $B$ is finite because the cluster tends to escape $B$ before growing large in it. The following notion captures this phenomenon.

\begin{dfn}\label{dfn:branching}
Let $G$ be a group with a finite generating set $S$. We say that a set $B \subseteq G$ is \emph{$r$-roughly branching} if there exists a subset $B' \subseteq G$ such that: 
\begin{enumerate}
\item $B$ is contained in the $r$-neighborhood of $B'$ in the word metric $d_{S}$.
\item For every $k \ge 1$, if $g_{1}, \ldots, g_{k}$ and $h_{1}, \ldots, h_{k}$ are distinct sequences of elements of $B'$, then $g_{1} \cdots g_{k} \neq h_{1} \cdots h_{k}$.
\end{enumerate}
\end{dfn}

\begin{lem}[{\cite[Lemma 5.5]{hutchcroft2019percolation}}]\label{lem:branchingCapacity}
Let $G$ be a group with a finite generating set $S$. Then for each $r > 0$ there exists $M$ such that for every $r$-roughly branching subset $B \subseteq G$ and for every $0 \le p \le p_{c}$ we have $\E_{p} \#\big(C(id) \cap B\big) \le M$.
\end{lem}

We sketch the proof for a $0$-roughly branching set $B$; see \cite{hutchcroft2019percolation} for a full proof. By the Harris-FKG inequality, $\tau_{p}(gh) \ge \tau_{p}(g) \cdot \tau_{p}(h)$ for each $g, h \in G$. Hence, for $g_{1}, \ldots, g_{k} \in B$, we have $\tau_{p} (g_{1} \cdots g_{k}) \ge \tau_{p}(g_{1}) \cdots \tau_{p}(g_{k})$. Meanwhile, the $k$-th convolution map from $B^{k}$ to $b$:  $(g_{1}, \ldots, g_{k}) \mapsto g_{1} \cdots g_{k}$  is injective by the assumption. This implies \[
\chi_{p} \ge \sum_{g \in B^{k}} \tau_{p}(g) = \sum_{g_{1}, \ldots, g_{k} \in B} \tau_{p}(g_{1} \cdots g_{k}) \ge \prod_{i=1}^{k} \left( \sum_{g_{i} \in B} \tau_{p}(g_{i}) \right).
\]
For a given $0 \le p <p_{c}$, this is true regardless of $k$. Note that $\chi_{p} < +\infty$. This forces that $\sum_{g \in B} \tau_{p}(g) \le 1$ for each $0 \le p < p_{c}$. Since $p \mapsto \sum_{g \in B} \tau_{p}(g)$ is a lower semicontinous function on $[0, 1]$ (cf. \cite[Lemma 5]{hutchcroft2016critical}), the same bound holds for $p = p_{c}$ as well.

We finally state the ``smallness" of halfspaces in $\mathbb{H}^{d}$ in terms of nested barriers.

\begin{prop}[{\cite[Lemma 5.6]{hutchcroft2019percolation}}]\label{prop:hutchcroftHdNested}
Let $X \ni x_{0}$ be a closed convex set of $\mathbb{H}^{d}$ and suppose that $G \le \Isom(X)$ properly and coboundedly embeds into $X$ by the orbit map. Let $S$ be a finite generating set of $G$. Then there exist $r, R>0$ such that for each halfspace $H \subseteq \mathbb{H}^{d}$, there exists an $r$-roughly branching subset \[
B = B_{1} \sqcup B_{2} \sqcup \ldots \sqcup B_{\lfloor d(x_{0}, H) / R\rfloor} \subseteq G
\]
such that $B_{i}$ is a $d_{S}$-barrier between $id$ and $\{g : gx_{0} \in H\}$ for each $i=1, \ldots, \lfloor d(x_{0}, H) / R\rfloor$.
\end{prop}

In the above, the capacity of $B$ is uniformly bounded in $p$ and $H$; hence, there exists $B_{i}$ such that $\E_{p}(\#C(id) \cap B_{i}) \lesssim 1/d(x_{0}, H)$. By Lemma \ref{lem:BKBarrierChi} we conclude

\begin{cor}[{\cite[Lemma 5.4]{hutchcroft2019percolation}}]\label{cor:hutchcroftHdNested}
Let $X \subseteq \mathbb{H}^{d}$ and $G \le \Isom(X)$ be as in Proposition \ref{prop:hutchcroftHdNested}. Then there exists $K>0$ such that for each halfspace $H \subseteq \mathbb{H}^{d}$ we have \[
\E_{p} \big(\#C(id) \cap \{ g \in G : gx_{0} \in H\}\big) \le \frac{K}{d(x_{0}, H)} \chi_{p}
\]
for each $0 \le p < p_{c}$.
\end{cor}

Proposition \ref{prop:hutchcroftMagic} and Corollary \ref{cor:hutchcroftHdNested} describe all we need for halfspaces. Let us now state an abstract version:

\begin{thm}[{\cite[Proof of Proposition 5.2]{hutchcroft2019percolation}}]\label{thm:hutchcroftIotaGen}
Let $\Gamma = Cay(G, S)$ be the Cayley graph of a finitely generated group $G$. Suppose that there exists $r>0$, and for each $D, E > 0$ there exist\[
S_{D} = \sqcup_{i=1}^{\infty} S_{D; i} \subseteq G, \quad \mathcal{G}_{D, E} \subseteq G
\]
and a collection $\mathscr{H}_{D}$  of subsets  of $G$ such that \begin{enumerate}
\item $S_{D}$ is $r'$-roughly branching for some $r' = r'_{D}$,
\item for each $\mathcal{H} \in \mathscr{H}_{D}$ there exists an $r$-roughly branching subset $B = B_{1} \sqcup \ldots \sqcup B_{D} \subseteq G$ such that $B_{i}$ is a $d_{S}$-barrier between $id$ and $\mathcal{H}$ for $i=1, \ldots, D$;
\item for each $D, E>0$, $\sqcup_{i \ge E} S_{D; i}$ is a $d_{S}$-barrier between $id$ and $\mathcal{G}_{D, E}$.
\end{enumerate}

Suppose that for each $\epsilon>0$ and $D, E > 0$, there exists a constant $N = N(\epsilon, D, E)$ such that for every finite set $A \subseteq G$ there exists $A' \subseteq A$ satisfying: \begin{enumerate}
\item $\#A' \ge (1-\epsilon)\#A$;
\item For each $a \in A'$ there exist $\mathcal{H}_{1}, \mathcal{H}_{2} \in \mathscr{H}_{D}$ such that \[
\# \big( A \setminus a \cdot ( \mathcal{H}_{1} \cup \mathcal{H}_{2} \cup \mathcal{G}_{D, E}) \big) \le N.
\]
\end{enumerate}
Then Equation \ref{eqn:hutchcroftGamma2} holds for $\Gamma$.
\end{thm}

In Hutchcroft's original formulation for Gromov hyperbolic graphs, the set $S_{D}$ and $\mathcal{G}_{D, E}$ are not needed. It is not hard to adapt Hutchcroft's proof to the current version; we include it for completeness.

\begin{proof}
By Lemma \ref{lem:branchingCapacity}, for each $D$ we have $\sum_{g \in S_{D}} \tau_{p_{c}}(g) < +\infty$. Furthermore, note that $\sqcup_{i=1}^{E} S_{D;i}$ exhausts $S_{D}$ as $E$ increases. Hence, for each $D>0$ and $\eta>0$ there exists $E=E(D, \epsilon)>0$ such that $\sum_{g \in \sqcup_{i \ge E} S_{D;i}} \tau_{p_{c}}(g) \le \epsilon$. Then by Lemma \ref{lem:BKBarrierChi}, we have $\sum_{g \in \mathcal{G}_{D, E}} \tau_{p}(g) \le \epsilon \cdot \chi_{p}$ for each $0 < p< p_{c}$.

Now, let $M = M(r)$ for $r$ as in Lemma \ref{lem:branchingCapacity}. Then by Assumption (2) and Lemma \ref{lem:BKBarrierChi}. we have $\sum_{g \in \mathcal{H}} \tau_{p}(g) \le M \chi_{p}/D$ for each $0<p<p_{c}$ and  for each $\mathcal{H} \in \mathscr{H}_{D}$.

Let us now fix $\epsilon>0$. We take $D > M/\epsilon$,  and then $E=E(D, \epsilon)$. Now let $N = N(\epsilon, D, E)$. Lastly, recall that $\lim_{p \nearrow p_{c}} \chi_{p} = +\infty$; there exists $p_{0}$ such that $\chi_{p} \ge N/\epsilon$ for $p_{0} < p < p_{c}$.

We now claim that \[
\frac{\sum_{g, h \in A} \tau_{p}(g, h)}{\#A} \le 5\epsilon \chi_{p} \quad (\forall \textrm{ finite $A \subseteq G$}, \forall p_{0} < p < p_{c}).
\]
To observe this, let $A \subseteq G$ be a finite set and let $p_{0} < p < p_{c}$. Let $A' \subseteq A$ be as in the proposition for $\epsilon, D, E$. Then we have \[\begin{aligned}
\sum_{g, h \in A} \tau_{p}(g, h) &\le \sum_{g \in A \setminus A', h \in G} \tau_{p}(g, h) + \sum_{g \in A', h \in  A \setminus g \cdot ( \mathcal{H}_{1}(g) \cup \mathcal{H}_{2}(g) \cup \mathcal{G}_{D, E})} \tau_{p}(g, h) \\
&+ \sum_{g \in A', k \in  \mathcal{H}_{1}(g) \cup \mathcal{H}_{2}(g) \cup \mathcal{G}_{D, E}} \tau_{p}(id, k) \\
&\le \epsilon (\#A)\cdot \chi_{p} +  (\#A')\cdot N + (\#A') \cdot \left( \frac{M}{D} \chi_{p} + \frac{M}{D}\chi_{p} + \epsilon \chi_{p}\right) \le 5\epsilon (\#A)\chi_{p}.
\end{aligned}
\]Since $\epsilon$ is arbitrary, we conclude that  Equation \ref{eqn:hutchcroftGamma2} holds.
\end{proof}

Combining the aforementioned facts about word hyperbolic groups and convex subsets of $\mathbb{H}^{d}$, together with the Bonk-Schramm embedding theorem, Hutchcroft showed that word hyperbolic groups satisfy Equation \ref{eqn:hutchcroftGamma1} and \ref{eqn:hutchcroftGamma2} in Theorem \ref{thm:hutchcroft}.

\subsection{Intuition and examples} \label{subsection:plans}

We now explain our strategy in detail.

Our primary example will be the free group $F_{2} \simeq \langle a, b\rangle$ with the generating set $S = \{a, b\}$. Its Cayley graph $\Gamma = Cay(F_{2}, S)$ is a regular 6-valent tree whose each edge is labeled with $a$ or $b$. Now, if we quotient out all the edges labeled with $a$, then the resulting graph $\Gamma'$ becomes a regular $\infty$-valent tree. The identity vertex $id$ is now connected with countably infinitely many vertices $\{a^{i} b^{\pm 1}: i \in \Z\}$. One can instead consider the Cayley graph with respect to an infinite generating set $S' = S \cup \{ a^{i} : i \in \Z\}$; this Cayley graph and $\Gamma'$ are quasi-isometric. 

At first it seems confusing to consider this $\infty$-valent tree instead of the original 6-valent tree.
But this construction is natural for acylindrically hyperbolic groups. Acylindricallly hyperbolic groups may have non-Gromov hyperbolic Cayley graphs, but they act on a Gromov hyperbolic space that comes from this construction.

 Let us first discuss the strategy for Equation \ref{eqn:hutchcroftGamma1}.

The classical halfspaces in $\mathbb{H}^{d}$ or Gromov hyperbolic spaces work for Proposition \ref{prop:hutchcroft2}. To be precise, given a  $\delta$-hyperbolic space $X \ni x_{0}$ and $x, y \in X$, we define \[
\mathcal{H}_{half}(x, y) := \left\{ g \in G : d_{X}(gx_{0}, x) \le d_{X} (gx_{0}, y) \right\}.
\]
Then for every non-elementary isometry group $G \le \Isom(X)$, there exist independent loxodromics $\{f_{1}, f_{2}, f_{3}\} \subseteq G$ and $R>0$ such that, for every pair of elements $u \in G$ such that $d_{X}(x_{0}, ux_{0})  \ge R$, there exists $i \in \{1, 2, 3\}$ such that $\mathcal{H}_{half}(x_{0}, ux_{0})$ and $uf_{i}u^{-1} \mathcal{H}_{half}(x_{0}, ux_{0})$ are disjoint. Hence, it is straightforward to generalize Proposition \ref{prop:hutchcroft2} to non-elementary isometry groups of Gromov hyperbolic spaces.

Meanwhile, it is harder to generalize Proposition \ref{prop:hutchcroft1} in terms of $\mathcal{H}_{half}(x, y)$ for non-elementary actions on a Gromov hyperbolic space. To illustrate this, consider $G =F_{2} \times \Z$, a group acting on the Cayley graph $\Gamma = Cay(F_{2}, S)$ by left multiplication of the first factor: $(a, b) \cdot x := ax$. Let $x_{0} = id \in \Gamma$. Now given $D>0$, consider a set $A \subseteq G$ whose $>99\%$ is concentrated on $id \in \Gamma$ and the remaining $<1\%$ covers $\{ x \in \Gamma : d_{S}(id, x) < D\}$. In other words, we consider \[
A = \{ (id, k) : 0 \le k \le 5^{D+10}\} \cup \{(g, 0) : g \in F_{2}, \|g\|_{S} \le D\}.
\]
Let us give a word metric on $G$, say, by the generating set $S' := \{(a, 0), (b, 0), (0, 1)\}$. Then for each $g =(id, k)$ for some $k$, there is no $h \in G$ such that $d_{S'}(g, h)\le D$ and $\mathcal{H}_{D}( g x_{0}, hx_{0})$ contains $A$. So $>99\%$ of elements of $A$ cannot satisfy the condition in Proposition \ref{prop:hutchcroft1}.

Roughly speaking, this is because of the distortion between the geometry of $G$ and $\Gamma$. It is possible to charge a single vertex $u$ in $\Gamma$ with arbitrarily many elements of $G$. For each $vx_{0} \in \Gamma$ with $d_{S}(vx_{0}, ux_{0}) \le D$, it is also easy to make $\mathcal{H}_{half}(ux_{0}, vx_{0})$ fail the condition in Proposition \ref{prop:hutchcroft1}; we just charge $vx_{0}$ with one element of $G$. This does not cost too much, as the number of $D$-neighbors of $ux_{0}$ in $\Gamma$ is bounded. 

This pathology is remedied when we impose the so-called \emph{weak proper discontinuity (WPD)}. Let us go back to the example $F_{3}$ acting on $\Gamma' \ni x_{0}=id$. It is possible that a single vertex $id \in \Gamma'$ can be charged by many elements of $G$, namely, $\{a^{i} : i \in \Z\}$. But for these elements, $\{\Gamma' \setminus \mathcal{H}_{half}(a^{i} x_{0}, a^{i} b^{D} x_{0}): i \in \Z\}$ are all \emph{disjoint}, as $x_{0}$ has valency $\infty$ and the edges $\overrightarrow{a^{i} x_{0} \, a^{i} b^{D} x_{0}}$ are distinct. Hence, it costs a lot to charge $\Gamma' \setminus \mathcal{H}_{half}(a^{i} x_{0}, a^{i} b^{D} x_{0})$ for each $i$: it cannot be done with $1\%$ of $A$.

Indeed, for $F_{2}$ acting on $\Gamma'$, and more for generally WPD actions, Proposition \ref{prop:hutchcroft1} does hold. We will prove this in Section \ref{section:supporting}.

Let us now discuss Equation \ref{eqn:hutchcroftGamma2}. In Section \ref{section:properMagic} we will prove an analogue of Proposition \ref{prop:hutchcroftMagic} for proper actions on a Gromov hyperbolic space. We sketch the idea for $F_{2} = \langle a, b \rangle$ acting on the Cayley graph $\Gamma = Cay(F_{2}, \{a, b\})$. Let $x_{0} = id \in \Gamma$. A relevant animation is available on the authors' webpage\footnote{\url{https://inhyeokchoi48.github.io/research/binary}}, and readers are invited to play with it.

Suppose that $A \subseteq F_{3}$ is the sphere $\{g \in F_{2} : \|g\|_{S} = R\}$. Then from the viewpoint of each $a \in A$, most elements of $A$ are in the direction of $id$. It is hence sensible to pick \[
H_{1}(g) = \mathcal{H}_{D}(g, id) := \{ u \in F_{2} : \textrm{$g^{-1} u$ and $g^{-1} \cdot id$ share the initial $D$-long subword}\}
\]
and remove it from $A$. Then we have $\#(A \setminus H_{1}(g)) \le \#\{u \in G : d_{S}(u, g) \le 2D\} \le (2\#S)^{2D}$ for each $a \in A$. This bound is independent of $R$. 

Let us now consider the ball $A = \{g \in F_{2} : \|g\|_{S} \le R\}$. The same bound holds for elements in the outmost sphere. But the bound gets worse as we go into deeper inner sphere. Nonetheless, it suffices to consider only the $10$ outmost spheres $\{g : R-10 \le \|g\|_{S} \le R\}$, as they account for $>99\%$ of $A$. In summary, we have \[
\#(A \setminus H_{1}(g)) \le \#\{u \in G : d_{S}(u, g) \le 2(D+L)\} \le (2\#S)^{2(D+L)}
\]
for $g \in \{u : R-L \le \|u\|_{S} \le R\}$, whose number is at least $(1- 5^{-L})\#A$. This bound depends on the choice of $D$ and $L$ but not on $R$.

From this example we can try the following. In an arbitrary finite set $A \subseteq F_{3}$, for each $g \in A$ we  take $H_{1}(g) = \mathcal{H}_{D}(g, id)$ and see if $A \setminus H_{1}(g)$ has uniformly bounded cardinality. If it does not, we regard $g$ as an element ``deep inside" and remove it. This removal is not critical as long as there are exponentially fewer ``inner" elements than ``outer" elements.

This strategy unfortunately does not work for an arbitrary finite set $A \subseteq F_{3}$. As a counterexample, consider $A = \{a^{10i} : i =0, \ldots, R\}$, a sequence of points along a geodesic from $id$. Then the cardinality of $A \setminus \mathcal{H}_{5}(g, id)$ is bounded by $N$ for only $N$ many $g$'s at the end of $A$, which  compose a negligible portion of $A$. Geometrically, this subset has linear growth instead of exponential growth; the outmost spheres are negligible compared to the inner part. Indeed, for any $N>0$ there exists $R$ such that \[
\# \{ a \in A : \exists\, \textrm{halfspace $H$ such that}\,\, d(a, H) = 5 \,\, \textrm{and}\,\, \# (A \setminus H) \le N \} \le  0.1\#A
\]
for $A = \{a^{i} : i=0, \ldots, R\}$.

This is the reason we need to exclude two halfspaces for each $g \in A$ instead of one. In the example $A = \{a^{10i} : i=0, \ldots, R\}$, $a^{5R}$ is considered a ``pre-inner" point, as $A \setminus \mathcal{H}_{5}(a^{5R}, id)$ contains $R/2$-many elements of $A$. But there is only one direct ``child" of $a^{5R}$ when viewed from $id$, namely, $a^{5R + 10}$. Having only one child is not desirable for the exponential growth. Thus, we will regard $a^{5R}$ as not genuinely inner. Then how do we cope with the largeness of $A \setminus \mathcal{H}_{5}(a^{5R}, id)$? We simply erase $H_{2}(a) := \mathcal{H}_{5}(a^{5R}, a^{5R + 10})$. Then the number of elements of  $A \setminus \big(\mathcal{H}_{5}(a^{5R}, id) \cup\mathcal{H}_{5}(a^{5R}, a^{5R + 10})\big)$ will be bounded.

Let us refine this strategy. We fix a bound $N$ not depending on the size of $A$. Let us collect problematic points \[
\mathcal{A} := \left\{ g \in A : \begin{array}{c} \#\big(A \setminus (H_{1}(g) \cup H_{2}(g))\big) \ge N\,\,\textrm{for all halfspaces} \\ \textrm{$H_{1}, H_{2}$ that are $D$-far from $id$}\end{array}\right\}.
\]
Then each $g \in \mathcal{A}$ is either an``outmost" element in $\mathcal{A}$ or might have some ``children" in $\mathcal{A}$. In the former case, $A \setminus \mathcal{H}_{D}(g, id)$ is supposed to be small, and there should be only few ``problematic" such elements. That means, we wish that the number of elements of $\mathcal{A}$ without ``children" in $\mathcal{A}$ will be bounded.

In the latter case, if $g$ has a lone child $h \in \mathcal{A}$, then $g$ is considered not ``deeply inner", and $A \setminus (\mathcal{H}_{D}(g, id) \cup \mathcal{H}_{D}(g, h))$ is morally small. Thus, we wish that the number of ``inner but not deeply inner" elements in $\mathcal{A}$ is also bounded. If $g \in \mathcal{A}$ has more than two children in $\mathcal{A}$, then we declare that $g$ is ``deeply inner". We give up such $g$, but this will not be a huge loss.

After this procedure, we are left with some non-deeply-inner points $\mathcal{G} = \{g_{1}, g_{2}, \ldots\} \subseteq A$. For those elements $g \in \mathcal{G}$, $A \setminus (H_{1}(g) \cup H_{2}(g))$ has at least $N$ elements. Now, if $A \setminus (H_{1}(g) \cup H_{2}(g))$ are disjoint for distinct $g \in \mathcal{G}$, then we can bound $\#\mathcal{G}$ in terms of $\#A$. Moreover, if deeply inner points are much fewer than not deeply inner points, then we can bound the cardinality of $\mathcal{A}$ in terms of $A$,

This strategy indeed works for locally finite subsets of Gromov hyperbolic spaces, which is the content of Proposition \ref{prop:hutchcroftMagic}. There are some concerns. What if different ``inner" points share a direct child? The hyperbolicity prevents this from happening. In a Gromov hyperbolic space, every ``lineage" is ``linearly ordered", and no bypass is allowed. What if different  non-deeply-inner points $g_{1}, g_{2}$ have non-disjoint $A \setminus (H_{1}(g_{i}) \cup H_{2}(g_{i}))$? Again, hyperbolicity is at play. If these sets overlap, then $g_{1}$ and $g_{2}$ are aligned when viewed from $x_{0}$. This means that one of $g_{1}, g_{2}$ is the descendent of the other one. With more care, it can be shown that one of $g_{1}, g_{2}$ is ``deeply inner" and should have been removed from $\mathcal{G}$. These technical points will be studied in depth in Section \ref{section:properMagic}.

Let us now talk about ``smallness" of halfspaces $H$ in terms of the capacity $\E_{p}[\#C(id) \cap H]$. In the real hyperbolic space $\mathbb{H}^{d}$, a halfspace $H$ that is $D$-far from the origin $x_{0}$ is barred by roughly branching disjoint union of $\sim D$ barriers. This intuitively makes sense because ``codimension-1" submanifolds disconnect $\mathbb{H}^{d}$ into two parts. There is a notion of codimension 1 subgroups for certain class of hyperbolic groups (such as cubical groups), but we will employ more general and abstract machinery.

Consider $F_{2} = \langle a, b\rangle$ once again. In between $id$ and $H = H_{100}(id, a^{100}):= \{ a^{100} \cdot w\textrm{, $w$ does not start with $a$}\}$, which are spaced horizontally, we can place nine disjoint sets $B'_{i} := \{a^{10i} w\textrm{, $w$ does not start with $a^{\pm}$}\}$. Equivalently, $B'_{i}$ is the collection of points $p$ whose projection $\pi_{[id, a^{100}]}(p)$ onto $[id, a^{100}]$ is precisely $a^{10i}$. Then each of $B'_{1}, \ldots, B_{9}'$ is  indeed a $d_{S}$-barrier between $id$ and $H$. The issue is that $B'_{i}$'s are too large and are not ``codimension-1". In fact, $B'_{i}$'s contain an arbitrarily large $d_{S}$-metric balls, and indeed $\E_{p}[\#C(id) \cap B_{i}']$ is not uniformly bounded in $p \in (0, p_{c})$.

We can instead consider ``vertical" barriers $B_{i} := \{a^{10i} b^{k} : k \in \Z\}$ that are ``thin" and are branching. Let us explain why $B_{1}$ is indeed a barrier. Suppose to the contrary that  a $d_{S}$-path $(id = g_{0}, g_{1}, \ldots, g_{N} \in H)$ avoids $B_{1}$. In this path ``the initial power of $a$ appearing in $g_{i}$" grows from $id$ to $a^{100}$ along $P$. Hence, there is a moment $i(1)$ where $g_{i(1)} = a^{10} w$ for some $w$ not starting with $a^{\pm 1}$. Since $P$ avoids $B_{1}$, $w$ contains some $a$. That means, $g_{i(1)} = a^{10} b^{k} a \cdot v$ in its reduced form for some $v$. 

We claim that the letter $a$ after $a^{10} b^{k}$ cannot be erased along the path, i.e.,  $g_{i}$ starts with $a^{10} b^{k} a$ for every $i \ge i(1)$. If $a$ is to be erased at some step $i, i \ge i(1)$,   the only possibility is $g_{i} = a^{10} b^{k} a$ and $g_{i+1} = a^{10} b^{k} \in B_{1}$, a contradiction to the assumption. But if every $g_{i}$ starts with $a^{10} b^{k} a$, including at $i = N$, then $g_{N}$ cannot land in $H = H_{100}(id, a^{100})$. This is a contradiction.

It might look like this strategy hinges on the fact that $F_{2}$ is a (quasi-)tree and is not applicable to, say, a surface group. In fact, this strategy can be applied to WPD actions on a Gromov hyperbolic space. This is secretly related to the fact that acylindrically hyperbolic groups act on a quasi-tree thanks to Bestvina-Bromberg-Fujiwara's construction \cite{bestvina2015quasi}. 
We explain this in Section \ref{section:branching}.

Proposition \ref{prop:hutchcroftMagic} is about locally finite sets of Gromov hyperbolic spaces. Thus, it can handle proper group actions on a Gromov hyperbolic space. For non-proper actions, Proposition \ref{prop:hutchcroftMagic} does not give an effective bound for element counting (as opposed to orbit counting). We hence need to exclude elements that contributes to non-properness from $A \setminus (\mathcal{H}_{1} \cup \mathcal{H}_{2})$. This is the reason we introduce $S_{D}$ and $\mathcal{G}_{D, E}$ in Theorem \ref{thm:hutchcroftIotaGen}. 

How do we define $S_{D}$ and $\mathcal{G}_{D, E}$? Recall that $G$ contains a loxodromic isometry $f$ of $X \ni x_{0} $, whose orbit $\{f^{i} x_{0}\}_{i \in \Z}$ is quasi-isometrically embedded in $X$. Hence, the powers of $f$ are witnesses of ``properness". In contrast, there can be elements $g \in G$ such that $[x_{0}, gx_{0}]$ are not fellow traveling with a translate of $[x_{0}, f^{i}x_{0}]$ for a long time. Such elements are manifestation of non-properness, and it is best to remove them from consideration.

With this in mind, we informally define \[
S_{D} := \left\{ g \in G : \begin{array}{c}\textrm{$[x_{0}, gx_{0}]$ does not fellow travel with} \\ \textrm{a translate of $Ax(f)$ for more than length $D$}\end{array}\right\}.
\]
Then $S_{D}$ becomes a roughly branching set (Proposition \ref{prop:NFRough}). This set corresponds to the collection $\mathcal{N}_{D} \subseteq F_{2}$ of words that do not have $a^{D}$ as a subword. In the Cayley graph of $F_{2}$, words in $\mathcal{N}_{D}$ are reached from $id$ by moving in an ``almost vertical" direction. This resembles ``vertical hyperplanes" in $Cay(F_{2})$, and it is easy to escape these hyperplanes by adjoining a long enough horizontal step $a^{2D}$.

Next, we informally define \[
\mathcal{G}_{D, E} := \left\{ g \in G :  \begin{array}{c}\textrm{$[x_{0}, gx_{0}]$ does not fellow travel with} \\ \textrm{$hAx(f)$ for more than length $D$} \\\textrm{for some $h \in B_{S}(id, E)$}\end{array}\right\}.
\]
This is the collection of words that do not fellow travel with $Ax(f)$ ``in the beginning". In the example of $F_{2}$, this corresponds to the halfspace $H$ that is $E$-far from $id$: in order to reach $H$, one has to move in the ``vertical" direction for length $E$ at first, but is allowed to move freely afterwards. Intuitively, $\mathcal{G}_{D, E}$ is barriered by $\mathcal{N}_{D} \cap \{ g : \|g\|_{S} \ge E/2\}$. We formally prove this in Proposition \ref{prop:DEThenEPrime}.

\section{Preliminaries II: Hyperbolic geometry} \label{section:hypPre}

Given three real numbers $A, B, C$, we write $A =_{C} B$ if $|A-B| < C$. 

A \emph{geodesic} on a metric space $(X, d)$ is an isometric embedding $\gamma : I \rightarrow X$ of a closed connected subset $I \subseteq \mathbb{R}$ into $X$. We will frequently refer to the image of $\gamma$ as $\gamma$. Throughout, every metric space $(X, d)$ is assumed to be  \emph{geodesic}, i.e., every pair of points are connected by a geodesic segment. We will however not assume that $(X, d)$ is locally compact or complete.

\subsection{Gromov hyperbolicity}

Let $(X, d_{X})$ be a geodesic metric space. Given a set $A \subseteq X$, we define its $R$-neighborhood \[
\mathcal{N}_{R}(A) := \{ x \in X : \exists a \in A \,[d_{X}(a, x) \le R]\}
\]
for each $R>0$. We define the Hausdorff distance between two sets \[
d_{X}(A, B) := \inf \{ R \ge 0 :A \subseteq  \mathcal{N}_{R}(B)\, \wedge\, B  \subseteq \mathcal{N}_{R}(A)\}.
\]
We will say that two sets $A, B \subseteq X$ are \emph{$R$-equivalent} if they are within Hausdorff distance $R$.

For $x, y \in X$, we denote by $[x, y]$ an arbitrary geodesic between $x$ and $y$. Note that such a geodesic may not be unique.

We now recall the notion of Gromov hyperbolicity due to M. Gromov \cite{gromov1987hyperbolic}.  The version we present here is E. Rips' one. Comprehensive expositions can be found in \cite{coornaert1990geometrie} and \cite{bridson1999metric}.

\begin{dfn}\label{dfn:Grom}
Let $(X, d)$ be a metric space. For a given $\delta >0$, we say that $(X, d)$ is \emph{$\delta$-hyperbolic} if every geodesic triangle is $\delta$-thin, that means,  \[
\forall x, y, z \in X \,\big[ [x, z] \subseteq \mathcal{N}_{\delta}([x, y]) \cup \mathcal{N}_{\delta}([y, z])\big].
\]
We say that $(X, d)$ is \emph{Gromov hyperbolic} if it is $\delta$-hyperbolic for some $\delta>0$.
\end{dfn}

The following is immediate:

\begin{lem}\label{lem:GromHausdorff}
Let $x, y, x', y'$ be points on a $\delta$-hyperbolic space $X$ such that $d_{X}(x, x'), d_{X}(y, y') < D$. Then $[x, y]$ and $[x', y']$ are $(2\delta + D)$-equivalent.
\end{lem}

Model examples of Gromov hyperbolic spaces are simplicial/$\mathbb{R}$-trees and real hyperbolic space $\mathbb{H}^{n}$. In these spaces, the following phenomenon happens: if you walk forward for some distance, and walk into another direction \emph{without huge backtracking}, and walk into yet another direction without huge backtracking, and so on, then you will never come back to the original place. In order to formulate the property rigorously, let us define: 

\begin{dfn}\label{dfn:GromProd}
Let $(X, d)$ be a geodesic metric space. For $x, y, z \in X$, we define the \emph{Gromov product of $y$ and $z$ based at $x$} by \[
(y| z)_{x} := \frac{1}{2}[d_{X}(x, y) + d_{X}(x, z) - d_{X}(y, z)].
\]
\end{dfn}

For example, in the standard Cayley graph of $F_{2} = \langle a, b \rangle$, we have $(aaba | aab^{-1} ab)_{id} = 2$ since $a a ba$ and $aab^{-1} ab$ share the first two letters.

We now formulate the local-to-global phenomenon mentioned above: 

\begin{lem}\label{lem:stability}
Let $x_{0}, x_{1}, \ldots, x_{n}$ be points on a $\delta$-hyperbolic space where \[
(x_{i-1} | x_{i+1})_{x_{i}} + (x_{i} | x_{i+2})_{x_{i+1}} \le d_{X}(x_{i}, x_{i+1})  - 24\delta \quad (i=1, \ldots, n-2).
\]
Then there are points $y_{1}, \ldots, y_{n-1}$ on $[x_{0}, x_{n}]$, in the order  \[
d_{X}(x_{0}, y_{1}) \le d_{X}(x_{0}, y_{2}) \le \ldots \le d_{X}(x_{0}, y_{n-1}),
\]
such that \[
d_{X}(x_{i}, y_{i}) =_{12\delta} (x_{i-1} | x_{i+1})_{x_{i}}  \quad (i=1, \ldots, n-1).
\]In particular, we have \[
d_{X}(x_{0}, x_{n}) \ge \sum_{i=1}^{n} d_{X}(x_{i-1}, x_{i}) - 2 \cdot \sum_{i=1}^{n-1} \Big( (x_{i-1} | x_{i+1})_{x_{i}} + 12 \delta\Big).
\]
\end{lem}

Another useful lemma is: 
\begin{lem}[{\cite[Lemma 1.3]{bonk1996quasi-geodesic}, \cite[Prop III.H.1.17]{bridson1999metric}}]\label{lem:fellow}
Let $x, y, z$ be points on a $\delta$-hyperbolic space. Then the initial $(y | z)_{x}$-long subsegments of $[x, y]$ and $[x, z]$ are $4\delta$-fellow traveling in a  synchronized manner. 

That means, if $\gamma : [0, d_{X}(x, y)] \rightarrow X$ represents $[x, y]$ and $\eta : [0, d_{X}(x, z)] \rightarrow X$ represents $[x, z]$, then $d_{X}(\gamma(t), \eta(t)) < 4\delta$ for $0 \le t \le (y | z)_{x}$.
\end{lem}

From this property, it follows that:

\begin{lem}[{\cite[Prop III.H.1.22]{bridson1999metric}}]\label{lem:GromIneq}
Let $x, y, z, w$ be points on a $\delta$-hyperbolic space. Then we have \[
(x|y)_{w} \ge \min \big( (x|z)_{w}, (z|y)_{w}\big) - 4\delta.
\]
\end{lem}

In fact, Lemma \ref{lem:stability} can be deduced from Lemma \ref{lem:fellow} and Lemma \ref{lem:GromIneq}. Indeed, an induction implies that $(x_{i-1} | x_{n+1})_{x_{i}} =_{4\delta}(x_{i-1}|x_{i+1})_{x_{i}}$ for each $1 \le i < n$. Another induction implies that $(x_{i}|x_{k})_{x_{j}} =_{8\delta} (x_{j-1} | x_{j+1})_{x_{j}}$ for each $i \le j \le k$. Yet another induction implies that $(x_{i}|x_{n})_{x_{0}}$ increases in $i$, and the points $y_{i}$ on $[x_{0}, x_{n}]$ whose distance from $x_{0}$ is $(x_{i} | x_{n})_{x_{0}}$ realize the desired property.

Let us now turn to isometries. Let $(X, d)$ be a Gromov hyperbolic space and let $g$ be its isometry. We say that $g$ is \emph{loxodromic} if there exists $\tau >0$ such that $d_{X}(x_{0}, g^{n} x_{0}) \ge \tau n$ for each $n$.

Prototypes of loxodromic isometries are the loxodromic isometries of $\mathbb{H}^{n}$. They act as a translation along an infinite geodesic. An isometry $g$ of a Gromov hyperbolic space $X$ is called an \emph{axial loxodromic} if there exists $\tau>0$ and a geodesic $\gamma : \mathbb{R} \rightarrow X$ such that $g(\gamma(t)) = \gamma(t+\tau)$ for each $t \in \mathbb{R}$. In this case, we call $\gamma$ an \emph{axis} of $g$ and denote it by $Ax(g)$. By rescaling the metric $d_{X}$ globally, it is not hard to render $g$ \emph{unital}, i.e., $\tau=1$.

In general, given a group $G$ acting on a Gromov hyperbolic space $X$ and a loxodromic isometry $g \in G$, it is not hard to put another metric on $X$ that is $G$-equivariantly quasi-isometric to the original one, so that $g$ becomes a unital axial loxodromic isometry. See e.g. \cite[Proposition 6.(2)]{bestvina2002bounded}.

\begin{dfn}\label{dfn:nonelementary}
Let $(X, d)$ be a Gromov hyperbolic space. We say that $G \le \Isom(X)$ is \emph{non-elementary} if there exist two loxodromic elements $g, h \in G$ such that \[
\sup_{i, j \in \Z} (g^{i}x_{0} | h^{j} x_{0})_{x_{0}} < +\infty
\]
for some (equivalently, every) $x_{0} \in X$. In this case, we say that $g$ and $h$ are \emph{independent}.
\end{dfn}

We now introduce the nearest point projection.

\begin{dfn}\label{dfn:projection}
Let $(X, d)$ be a metric space and let $A \subseteq X$ be a locally compact subset. We define the \emph{nearest point projection} $\pi_{A}(\cdot) : X \rightarrow 2^{A}$ as \[
\pi_{A}(x) := \{ a \in A : d_{X}(x, a) = \min_{y \in A} d_{X}(x, y)\}.
\]
For $B, C \subseteq X$, we use the notation $\diam_{A}(B) := \diam_{X}(\pi_{A}(B))$ and $d_{A}(B, C) := \diam_{X}(\pi_{A}(B \cup C))$.
\end{dfn}

\begin{lem}\label{lem:projection}
Let $x, y, z$ be points on a $\delta$-hyperbolic space $X$. Let $p \in [x,y]$ be such that $d_{X}(x, p) = (y|z)_{x}$. Then $\pi_{[x, y]}(z)$ is contained in $\mathcal{N}_{8\delta}(p)$.
\end{lem}

In Gromov hyperbolic spaces, geodesics exhibit the so-called \emph{contracting property}. The following is one formulation of the contracting property.

\begin{lem}[{\cite[Proposition 10.2.1]{coornaert1990geometrie}}]\label{lem:projGrom}
Let $(X, d_{X})$ be a $\delta$-hyperbolic space, let $x, y \in X$, let $\gamma$ be a geodesic and let $p \in \pi_{\gamma}(x)$, $q \in \pi_{\gamma}(y)$. Then we have \[
d_{X}(p, q) \le \max \big( 12\delta, 12 \delta + d_{X}(x, y) - d_{X}(x, p) - d_{X}(q, y) \big).
\]
\end{lem}

This lemma has the following corollary.
\begin{cor}\label{cor:projGrom}
Let $(X, d_{X})$ be a $\delta$-hyperbolic space, let $x, y \in X$ and let $\gamma$ be a geodesic in $X$. \begin{enumerate}
\item (coarse Lipschitzness) We have $\diam (\pi_{\gamma}(x)) \le 12\delta$ and $d_{\gamma}(x, y) \le d_{X}(x, y) + 12\delta$.
\item (constriction) Let $p \in \pi_{\gamma}(x)$, $q \in \pi_{\gamma}(y)$. Suppose that $d_{X}(p, q) > 12\delta$. Then $[p, q]$ is within Hausdorff distance $12\delta$ from some subsegment $[x', y']$ of $[x, y]$, where $d_{X}(x', p), d_{X} (y', q) < 10\delta$.
\item (no backtracking) Let $z \in [x, y]$. Then $\pi_{\gamma}(z)$ is contained in the $12\delta$-neighborhood of $[\pi_{\gamma}(x), \pi_{\gamma}(y)]$.
\item (equivalent geodesics) Let $\gamma'$ be a geodesic whose endpoints are pairwise $D$-near with the ones of $\gamma$. Then $\pi_{\gamma}(x)$ and $\pi_{\gamma'}(x)$ are $(2D + 28\delta)$-equivalent.
\item Let $\gamma'$ be a subgeodesic of $\gamma$ and suppose that $d_{\gamma'}(x, y) > 12\delta$. Then $d_{\gamma}(x, y) \ge d_{\gamma'}(x, y) - 64\delta$.
\end{enumerate}
\end{cor}
We include the proof in Appendix \ref{appendix:proj} for completeness.

We now review the notion of weak proper discontinuity introduced in \cite{bestvina2002bounded}.

\begin{dfn}\label{dfn:proper}
Let $(X, d_{X})$ be a Gromov hyperbolic space and let $G \le \Isom(X)$. We say that the action of $G$ on $X$ is \emph{proper} if \[
\forall R > 0\Big[ \# \Big\{ g \in G : d_{X}(x_{0}, gx_{0}) < R\big\} < +\infty\Big].
\]

Let $f \in G$ be a loxodromic isometry. We say that $G$ has \emph{weakly properly discontinuous (WPD)} along $f$, or that $f$ has the WPD property, if  \[
\forall R>0\, \exists L>0 \Big[\# \big\{ g \in G : d_{X}(x_{0}, gx_{0}) < R \,\,\textrm{and}\,\, d_{X}(f^{L}x_{0}, gf^{L} x_{0}) < R\big\} < +\infty \Big].
\]
If $f$ is axial in addition, we call it an \emph{axial WPD loxodromic}.

If $G$ has WPD action on a Gromov hyperbolic space and is not virtually cyclic, then we call it an \emph{acylindrically hyperbolic group}.
\end{dfn}

If $G$ acts properly on a Gromov hyperbolic space, then every loxodromic element has the WPD property automatically. We record a theorem by M. Bestvina, K. Bromberg and K. Fujiwara. \begin{dfn}
\label{dfn:eclosure}
Let $G$ be a group and let $f \in G$. We define the \emph{elementary closure} of $f$ by \[
EC(f) := \big\{ g \in G :\exists N>0  [g f^{N} g^{-1} = f^{N} ] \vee [gf^{N} g^{-1} = f^{-N}] \big\}.
\]
\end{dfn}

\begin{thm}[{\cite[Theorem H]{bestvina2015quasi}}]\label{thm:elemClos}
Let $G$ be an acylindrically hyperbolic group. Then $G$ contains an element $f$ and admits an isometric action on a Gromov hyperbolic space $X$, and there exists a constant $K>0$ such that the following holds:
\begin{enumerate}
\item  $f$ is a unital, axial WPD loxodromic isometry of $X$;
\item for each $g \in G$, either\begin{enumerate}
\item (bounded projection) the nearest point projection of $Ax(f)$ onto $g Ax(f)$ has diameter $\le K$, or
\item $g \in EC(f)$ and $gAx(f)= Ax(f)$.
\end{enumerate}
Moreover, the cyclic subgroup $\langle f \rangle$ is a finite-index subgroup of $EC(f)$.
\end{enumerate}
\end{thm}

Note that in Theorem \ref{thm:elemClos}, every finite generating set $S$ of $G$ contains an element $g \in G$ that falls into Case (2-a), as $S$ generates a non-virtually cyclic group.

The following is a well-known fact about acylindrically hyperbolic groups and is a basic ingredient of the quasi-tree construction in \cite{bestvina2015quasi}. We sketch the proof for reader's convenience.

\begin{lem}\label{lem:chainBBF}
Let $X$ be a $\delta$-hyperbolic space and let $\gamma_{1}, \ldots, \gamma_{n}$ be geodesics with mutually $K_{0}$-bounded projections. Let $z \in X$ be a point such that $d_{\gamma_{i}}(z, \gamma_{i+1}) \ge 5K_{0} + 100\delta$ for each $i=1, \ldots, n-1$. Then $d_{\gamma_{i}}(z, q) \ge  d_{\gamma_{i}}(z, \gamma_{i+1}) -(2K_{0} + 112\delta)$ for each $i$ and for each $q \in \gamma_{n}$.
\end{lem}

Before proving it, let us observe a simple fact:

\begin{lem}\label{lem:neighborGromProd}
Let $X$ be a geodesic metric space, let $x, y, z \in X$ and let $w \in [x, y]$. Then $(w|z)_{y} \le (x|z)_{y}$ holds.

Let $u \in \mathcal{N}_{K}([x, y])$. Then $(u | z)_{y} \le (x|z)_{y} + K$ holds.
\end{lem}
\begin{proof}
The first statement follows from $d_{X}(x, y) =d_{X}(x, w) + d_{X}(w, y)$ and $d_{X} (x, z) \le d_{X}(x, w) + d_{X}(w, z)$. The second statement follows from the first one and the triangle inequality.
\end{proof}

\begin{proof}[Proof of Lemma \ref{lem:chainBBF}]
Let $p_{i} \in \pi_{\gamma_{i}}(z)$ and $q_{i} \in \pi_{\gamma_{i}}(\gamma_{i+1})$ be the ones such that $d_{X}(p_{i}, q_{i}) = d_{\gamma_{i}}(z, \gamma_{i+1})$; for $i = n$, we pick arbitrary $q_{n} \in \gamma_{n}$. 

Then for $i=1, \ldots, n-1$, $p_{i}$ and $q_{i}$ are $(5K_{0} + 100\delta)$-far. Furthermore, $q_{i}\in \pi_{\gamma_{i}}(\gamma_{i+1})$ and $\pi_{\gamma_{i}}(q_{i+1}))\in \pi_{\gamma_{i}}(\gamma_{i+1})$ are $K_{0}$-close because of the bounded projection assumption. Corollary \ref{cor:projGrom}(2) tells us that $[z, q_{i+1}]$ passes through the $10\delta$-neighborhood of $p_{i}$ and $(10\delta + K_{0})$-neighborhood of $q_{i}$, in order. This implies that $(z | q_{i+1})_{q_{i}} < K_{0} + 10\delta$ for $i=1, \ldots, n-1$.

For the same reason, $q_{i-1}$ is $(K_{0}+10\delta)$-close to $[z, q_{i}]$ for each $i \ge 2$. Lemma \ref{lem:neighborGromProd} tells us that \[
(q_{i-1}|q_{i+1})_{q_{i}} \le (z | q_{i+1})_{q_{i}} + (K_{0} + 10\delta) \le 2K_{0} + 20\delta.
\]
Moreover, $q_{i-1}$ is also $(K_{0} +10\delta)$-close to $[z, p_{i}]$ for each $i \ge 2$, and $p_{i}$ is $10\delta$-close to $[z, q_{i}]$ by Corollary \ref{cor:projGrom}(2). This implies that  \[
(q_{i-1} | q_{i})_{p_{i}} \le (z | q_{i})_{p_{i}} + (K_{0} + 10\delta) \le K_{0} + 20\delta.
\]
Since $d_{X}(p_{i}, q_{i}) \ge 5K_{0} + 100\delta$, we conclude $d_{X}(q_{i-1}, q_{i}) \ge 4K_{0} + 80\delta$ for $i=2, \ldots, n-1$.

We can now apply Lemma \ref{lem:stability} to the points \[
(z, q_{1}, \ldots, q_{n}).
\]
It follows that $(z | q_{n})_{q_{i}} < 2K_{0} + 32\delta$. Meanwhile, since $d_{X}(p_{i}, q_{i}) > 10K_{0} + 130\delta$ and $d_{X}(p_{i}, [z, q_{i}]) < 10\delta$, we have $(p_{i} | z)_{q_{i}} \ge 10K_{0} + 120\delta$. Lemma \ref{lem:GromIneq} tells us that $(p_{i} | q_{n})_{q_{i}} \le 2K_{0} + 36\delta$. By Lemma \ref{lem:projection}, $q_{i}$ is $(2K_{0} + 44\delta)$-close to $\pi_{[p_{i}, q_{i}]}(q_{n})$ and $d_{[p_{i}, q_{i}]}(z, q_{n})\ge d_{\gamma_{i}}(z, \gamma_{i+1}) - (2K_{0} + 44\delta)$. By Corollary \ref{cor:projGrom}(5), we have $d_{\gamma_{i}} (z, q_{n}) \ge d_{\gamma_{i}}(z, \gamma_{i+1}) - (2K_{0} + 110\delta)$ as desired.
\end{proof}

We now record another consequence of the WPD property. This intuitively tells us that the axis $\gamma$ of a WPD element does not behave like a factor of a metric product. That means, 
 there exists a uniform threshold $D_{0}$ such that, if an object is too far from $\gamma$ compared to its size, then the object has $D_{0}$-small projection on $\gamma$. Here, it does not matter if the diameter of object is larger than $D_{0}$.

\begin{lem}[{\cite[Lemma 3.3]{sisto2016quasi-convexity}, \cite[Lemma 3.2]{choi2025acylindrically}}]\label{lem:amplifyWPD}
Let $G$ be a non-virtually cyclic group with a finite generating set  $S \subseteq G$. Suppose that $G$  acts on a Gromov hyperbolic space $X \ni x_{0}$ with a WPD loxodromic element $f \in G$. Then there exists $D_{0} > 0$, and for each $k, M>0$ there exists $R=R(k, M)>0$, such that the following holds.

Let $g, h \in G$ be such that $\|g\|_{S} > R$ and $\|h\|_{S} \le M$. Then $\pi_{[x_{0}, f^{k} x_{0}]} (\{gx_{0}, ghx_{0}\})$ has diameter at most $D_{0}$. 
\end{lem}

\section{Hyperbolic magic lemma} \label{section:properMagic}

The following is called a hyperbolic magic lemma \cite[Proposition 4.1]{hutchcroft2019percolation}.
Hutchcroft proved it under the assumption that $X$ is the real hyperbolic space $\mathbb{H}^{n}$ and $A$ lies in a quasi-convex set. Our aim is to generalize it to general Gromov hyperbolic spaces. Readers are once again invited to play with the animation\footnote{\url{https://inhyeokchoi48.github.io/research/binary}}.

A subset $Y \subseteq X$ of a metric space is \emph{uniformly locally finite} if \[
\sup_{y \in Y} \# \big( \mathcal{N}_{R}(y) \cap Y \big) < +\infty \quad (\forall R > 0).
\]
The vertex set of a Cayley graph of a  finitely generated group is uniformly locally finite. More generally, if a group $G$ acts properly on a metric space $X\ni x_{0}$, then the $G$-orbit $G\cdot x_{0}$ is uniformly locally finite.

Given $x, y\in X$ and $D>0$, we define \[
\mathcal{H}_{D}(x, y) := \{ z \in X : (z|y)_{x} \ge D\}.
\]
(Note that this set can well be empty). Sets of this sort are called \emph{halfspaces} rooted at $x$  with radius parameter $D$.

\begin{prop}\label{prop:magic}
Let $X$ be a $\delta$-hyperbolic space and let $Y$ be a uniformly locally finite subset of $X$. Then for each $\epsilon, D>0$ there exists a constant $N = N(\epsilon, D, Y)$ such that for every finite set $A \subseteq Y$ there exists a subset $A' \subseteq A$ satisfying:

\begin{enumerate}
\item $\#A' \ge (1-\epsilon) \#A$;
\item For each $a \in A'$ there exist halfspaces $\mathcal{H}_{1}, \mathcal{H}_{2} \in X$ rooted at $a$ with radius parameter $D$ such that $\#\big( A \setminus (\mathcal{H}_{1} \cup \mathcal{H}_{2}) \big) \le N$.
\end{enumerate}
\end{prop}

\begin{proof}
For simplicity, we assume $1 \le \delta \le 0.0001D$. Because $Y$ is locally uniformly finite, we have \[
\sup_{y \in Y} \# \big( \mathcal{N}_{100D}(y) \cap Y\big) =: M < +\infty.
\]
We set $N = 2M/\epsilon$. Note that $N$ depends on $\epsilon, D, Y$ but not on the choice of $A \subseteq Y$.

Now let $A \subseteq Y$ be a finite set. Let $A'$ be the collection of the elements of $A$ that satisfy the condition (2) in the statement. Our goal is to show $\#A' \ge (1-\epsilon)\#A$.

For a technical reason we introduce a variation of the notion of halfspace. Given $x, y \in X$ and $r > 0$, let us define the \emph{anti-halfspace}\[
\mathfrak{A}_{D}(x, y) := \left\{ z \in X \setminus \mathcal{N}_{6D}(x) : \begin{array}{c} \textrm{$\exists$ a geodesic $\gamma : [0, \tau] \rightarrow X$ and $0 \le \tau_{1} \le \tau$ such that}\\
\textrm{$\gamma(0) = y$, $d_{X}(\gamma(\tau_{1}), x) < D+100\delta$ and $d_{X}(\gamma(\tau), z) < D + 200\delta$} \end{array}\right\}.
\]
This is morally the complement of $\mathcal{H}_{D}(x, y)$ but not quite exactly.

\begin{figure}
\begin{tikzpicture}
\fill (0, 0) circle (0.07);
\fill (3, 0) circle (0.07);
\fill[opacity=0.15] (2.3, 3) -- (2.3, 0.7*1.7230508) arc (120:-120:1.4) -- (2.3, -3) -- (6.5, -3) -- (6.5, 3) -- cycle;
\draw[thick] (3, 0) circle (0.5);
\draw[thick] (0, 0) -- (6, -0.75);
\draw[thick] (5.2, -1.2) circle (0.7);
\fill[thick] (5.2, -1.2) circle (0.07);
\draw (2.85, -0.8) node {$\gamma(\tau)$};
\draw (5.3, -0.2) node {$\gamma(\sigma)$};
\draw (3, 0.22) node {$x$};
\draw (0, 0.25) node {$y$};
\draw (5.2, -1.42) node {$z$};
\end{tikzpicture}
\caption{Definition of anti-halfspace $\mathfrak{A}_{D}(x, y)$.}
\end{figure}
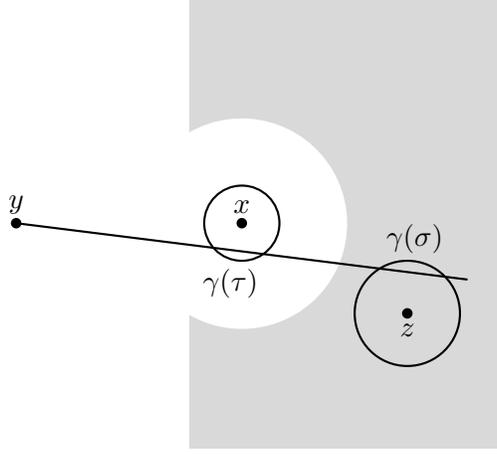

We record two elementary observations.
\begin{obs}\label{obs:magicObs1}
For every $x, y \in X$ we have $x  \notin \mathfrak{A}_{D}(x, y)$.  If $d_{X}(x, y) \ge 100D$ moreover, then $y \notin \mathfrak{A}_{D}(x, y)$.
\end{obs}

\begin{obs}\label{obs:magicObs2}
Let $x, y \in X$ and  let $z \in \mathfrak{A}_{D}(x, y)$. Then $d_{X}(y, z) \ge d_{X}(x, y) + D$.
\end{obs}

Let us now collect problematic elements, i.e., \[
\mathcal{A}_{1} := A \setminus A' = \left\{a \in A :  \begin{array}{c}\#\big( A  \setminus (\mathcal{H}_{1} \cup \mathcal{H}_{2}) \big) \ge N\,\, \textrm{for every halfspaces }\\
\textrm{$\mathcal{H}_{1}, \mathcal{H}_{2}$  rooted at $a$ with distance parameter $D$}\end{array}\right\}.
\]
We now pick a maximally $100D$-separated subset $\mathcal{A}_{2}$ of $\mathcal{A}_{1}$, i.e., we have \begin{enumerate}
\item $d_{X}(a, a') \ge 100D$ for each pair of distinct elements $a, a' \in \mathcal{A}_{2}$;
\item $\mathcal{A}_{2}$ is a maximal subset of $\mathcal{A}_{1}$ satisfying this property.
\end{enumerate}
Then $\bigcup_{a \in \mathcal{A}_{2}} (\mathcal{N}_{100D}(a) \cap Y)$ covers entire $\mathcal{A}_{1}$ (if not, a missed element can be added to $\mathcal{A}_{2}$ and break the maximality). Since $\mathcal{N}_{100D}(a) \cap Y$ has at most $M$ elements for each $a \in \mathcal{A}_{2}$, we have \[
\#\mathcal{A}_{2} \ge \frac{1}{M} \cdot \#\mathcal{A}_{1}.
\]

The proof will be done once we show that $\#\mathcal{A}_{2} \le \frac{\epsilon}{M} \# A$. For this one might wish to create disjoint complements of halfspaces rooted at  each element of  $\mathcal{A}_{2}$. However, the complement of halfspaces rooted at distinct elements of $\mathcal{A}_{2}$ might intersect. Our next goal is to extract some portion of $\mathcal{A}_{2}$ for which we can create disjoint complements of halfspaces.

Let us first prepare empty collections $\mathcal{B}= \mathcal{U} = \mathcal{G} = \emptyset$. They are meant to be collections of bad, undecided and good elements. Fix a basepoint $x_{0} \in X$. Enumerate $\mathcal{A}_{2}$ by the distance from $x_{0}$, i.e., let $\mathcal{A}_{2} = \{a_{1}, a_{2}, \ldots, a_{\#\mathcal{A}_{2}}\}$ be such that $d_{X}(x_{0}, a_{i}) \le d_{X}(x_{0}, a_{i+1})$ for each $i$. At each step $i=1, \ldots, \#\mathcal{A}_{2}$, we will define a point $b_{i} \in X$ and  put $a_{i}$ in either $\mathcal{B}$ or $\mathcal{G}$; this decision is final and shall not be modified further. We may put some other elements of $\mathcal{A}_{2}$ in $\mathcal{U}$, whose their classification will change later. 
We will keep the balance $\#\mathcal{B} \le \#\mathcal{U} + \#\mathcal{G}$ throughout. After the last step there will be no element of $\mathcal{U}$, so we will have $\#\mathcal{B} \le \#\mathcal{G}$.

We now describe the procedure. At step $i$, we first declare $\mathfrak{A}_{i}:= \mathfrak{A}_{D}(a_{i}, x_{0})$. \begin{enumerate}
\item If $\mathcal{A}_{2} \cap \mathfrak{A}_{i}$ has no element, then we declare that $a_{i} \in \mathcal{G}$ and $b_{i} := x_{0}$.
\item If not, pick $b_{i} \in \mathcal{A}_{2} \cap \mathfrak{A}_{i}$ that is the \emph{closest} to $x_{0}$. We then declare $\mathfrak{A}_{i}' := \mathfrak{A}_{D}( a_{i}, b_{i})$. 
\begin{enumerate}
\item If $\mathcal{A}_{2} \cap \mathfrak{A}_{i} \cap \mathfrak{A}_{i}'$ has no element, then we declare that $a_{i} \in \mathcal{G}$. 
\item If not, we pick $c_{i} \in \mathcal{A}_{2} \cap \mathfrak{A}_{i} \cap \mathfrak{A}_{i}'$ that is the \emph{closest} to $x_{0}$. We then declare $a_{i} \in \mathcal{B}$ and $b_{i}, c_{i} \in \mathcal{U}$.
\end{enumerate}
\end{enumerate}
(If an element in $\mathcal{U}$ is declared good or bad, it is not undecided anymore; we remove it from $\mathcal{U}$.)

Till step $i$, $\mathcal{G} \cup \mathcal{B}$ comprises of elements from $\{a_{1}, \ldots, a_{i}\}$; they do not contain any of $a_{i+1}, a_{i+2}, \ldots$. ($\ast$) Let us observe what happens at step $i$.

In case (1), $\mathcal{G}$ gains one more element that might be from  $\mathcal{U}$ or not. $\mathcal{B}$ does not change. Overall, $\#\mathcal{B}$ stays the same and $\#\mathcal{U} + \#\mathcal{G}$ does not decrease. Similar situation happens in Case (2-a).

In case (2-b), $\mathcal{B}$ gains one element $a_{i}$, which might be from $\mathcal{U}$. In exchange, $\mathcal{U}$ gains elements $b_{i}$ and $c_{i}$. Observation \ref{obs:magicObs2} guarantees that $d_{X}(x_{0}, b_{i}), d_{X}(x_{0}, c_{i}) \ge d_{X}(x_{0}, a_{i}) + D$. Since $\mathcal{A}_{2}$ was labelled with respect to the distance from $x_{0}$, we conclude that $b_{i}, c_{i} \in \{a_{i+1}, a_{i+2}, \ldots\}$; in other words, neither $b_{i}$ nor $c_{i}$ come from $\mathcal{G} \cup \mathcal{B}$. We thus confirm that elements are never re-classified once they are put in $\mathcal{G} \cup \mathcal{B}$.

Furthermore, note that $b_{i} \in \mathfrak{A}_{D}(a_{i}, x_{0}) \cap \mathcal{A}_{2}$ and $c_{i} \in \mathfrak{A}_{D}(a_{i}, b_{i}) \cap \mathcal{A}_{2}$. By Observation \ref{obs:magicObs1}, the former membership implies that $b_{i} \neq a_{i}$ and $d_{X}(b_{i}, a_{i}) \ge 100D$ (as $A_{2}$ is $100D$-separated), and the latter membership implies that $c_{i} \notin \{a_{i}, b_{i}\}$. In particular, $b_{i}, c_{i}$ are distinct. If $b_{i}, c_{i}$ are not from $\mathcal{U}$ at step $i-1$ and are genuinely new additions to $\mathcal{U}$, then we can conclude that $\#\mathcal{U}$ increases at least by 1 in Case (2-b). It remains to show\begin{claim}\label{claim:unredun}
For $i< j$ such that $a_{i}, a_{j} \in \mathcal{B}$, we have $\{b_{i}, c_{i}\} \cap \{b_{j}, c_{j}\} =\emptyset$.
\end{claim} 

\begin{proof}[Proof of Claim \ref{claim:unredun}]
Suppose first to the contrary that $b_{i} \in \{b_{j}, c_{j}\}$. Then by the construction of $b_{i}$, $b_{j}$ and $c_{j}$, we have \[
b_{i} \in \mathfrak{A}_{D}(a_{i}, x_{0}) \cap \mathfrak{A}_{D}(a_{j}, x_{0}).
\]
Let $\gamma : [0, \tau] \rightarrow X$ be a geodesic starting at $x_{0}$ and $0 \le \tau_{1} \le \tau$ be such that \[
d_{X}\big(\gamma(\tau_{1}), a_{i}\big) < D+100\delta, \quad d_{X}\big(\gamma(\tau), b_{i}\big) < D+200\delta.
\] Let $\gamma' : [0, \sigma] \rightarrow X$ be a geodesic starting at $x_{0}$ and $0 \le \sigma_{1} \le \sigma$ be such that \[
d_{X}\big(\gamma'(\sigma_{1}), a_{j}\big) < D+100\delta, \quad d_{X}\big(\gamma'(\sigma), b_{i}\big) < D+200\delta.
\] 

Let $L_{min} := \big(\gamma(\tau) \big| \gamma'(\sigma) \big)_{x_{0}}$. Lemma \ref{lem:fellow} tells us that $d_{X}(\gamma(t), \gamma'(t)) < 4\delta$ for $0 \le t \le L_{min}$. Note that  \[
\tau- L_{min} =  \big(x_{0} \big| \gamma'(\sigma) \big)_{\gamma(\tau) } \le d_{X}\big(\gamma(\tau), \gamma'(\sigma)\big) \le 2D + 400\delta.
\] We conclude that $L_{min} \ge \tau - 3D$. Similarly $L_{min} \ge \sigma- 3D$.

Meanwhile, recall that $a_{j}$ and $b_{i} \in \{b_{j}, c_{j}\}$ are distinct elements of a $100D$-separated set  $\mathcal{A}_{2}$. It follows that  $d_{X}(a_{j}, b_{i}) > 100D$ and $\sigma - \sigma_{1} \ge 97D$. In other words, we have $\sigma_{1} \le \sigma - 97D \le L_{min} - 94D \le \tau - 94D$. 

We now have \begin{equation}\label{eqn:GammaSigma}
d_{X}\big(\gamma(\sigma_{1}), a_{j}\big) \le d_{X}\big(\gamma(\sigma_{1}), \gamma'(\sigma_{1})\big) + d_{X}\big(\gamma'(\sigma_{1}), a_{j}\big) \le D + 120\delta.
\end{equation}
This implies that \[
100D \le d_{X}(a_{i}, a_{j}) \le d_{X}\big(a_{i}, \gamma(\tau_{1})\big) + d_{X}\big(\gamma(\tau_{1}), \gamma(\sigma_{1})\big) + d_{X}\big(\gamma(\sigma_{1}), a_{j}\big) \le 3D + |\tau_{1} - \sigma_{1}|,
\]
i.e., $\tau_{1} \ge \sigma_{1} + 97D$ or $\tau_{1} \le \sigma_{1} - 97D$. In the former case, we have \[
\begin{aligned}
d_{X}(x_{0}, a_{i}) &\ge \tau_{1} - d_{X}\big(\gamma(\tau_{1}), a_{i}\big) \\
&\ge \tau_{1} - (D + 100\delta) \ge \sigma_{1} +95D\\
& \ge d_{X}\big(x_{0}, \gamma'(\sigma_{1})\big) +d_{X}\big(\gamma'(\sigma_{1}), a_{j}\big) + 90D \ge d_{X}(x_{0}, a_{j})+90D,
\end{aligned}
\]
contradicting the ordering of $\{a_{1}, a_{2}, \ldots\}$. Hence, the latter case holds.

We now have timing $0\le \tau_{1} \le \sigma_{1} \le L_{min} \le \tau$ for the geodesic $\gamma$. Recall also Inequality \ref{eqn:GammaSigma}. We conclude $a_{j} \in \mathcal{A}_{2} \cap \mathfrak{A}_{D}(a_{i}, x_{0})$. Moreover, \[\begin{aligned}
d_{X}(x_{0}, a_{j}) &\le d_{X}(x_{0},  \gamma'(\sigma_{1}')) + d_{X}(\gamma'(\sigma_{1}'), a_{j}) \\
&\le \sigma_{1} + D + 100\delta \le \sigma - 95D\\
& \le d_{X}(x_{0}, \gamma'(\sigma)) - d_{X}(\gamma'(\sigma), b_{i})- 90D \le d_{X}(x_{0}, b_{i}) - 90D.
\end{aligned}
\]
This contradicts the minimality of $b_{i}$ with respect to the distance from $x_{0}$. Hence, we have $b_{i} \notin \{b_{j}, c_{j}\}$.

Now suppose to the contrary that $c_{i} \in \{b_{j}, c_{j}\}$. This implies that \[
c_{i} \in \mathfrak{A}_{D}(a_{i}, b_{i}) \cap \mathfrak{A}_{D}(a_{i}, x_{0}) \cap  \mathfrak{A}_{D}(a_{j}, x_{0}).
\]
We pick a geodesic $\gamma : [0, \tau] \rightarrow X$ starting at $x_{0}$ and $0 \le \tau_{1} \le \tau$ such that \[
d_{X}(\gamma(\tau_{1}), a_{i}) < D + 1000\delta, \quad d_{X}(\gamma(\tau), c_{i}) < D + 200\delta.
\] The previous argument tells us the following: since $c_{i} \in \mathfrak{A}_{D}(a_{i}, x_{0}) \cap \mathfrak{A}_{D}(a_{j}, x_{0})$, there exists $\tau_{1} +97D \le \sigma_{1} \le \tau - 90D$ such that $d_{X}(\gamma(\sigma_{1}), a_{j}) \le D + 120\delta$. In particular, $a_{j} \in \mathfrak{A}_{D}(a_{i}, x_{0})$. Moreover, $a_{j}$ is closer than $c_{i}$ to $x_{0}$.

Now consider a geodesic $\eta : [0, L'] \rightarrow X$ starting at $b_{i}$ and $0 \le \tau_{1}' \le \tau'$ such that \[
d_{X}(\eta(\tau_{1}'), a_{i}) < D + 100\delta, \quad d_{X}(\eta(\tau'), c_{i}) < D + 200\delta.
\]

We first consider the geodesic triangle connecting $\gamma(\tau_{1}), \gamma(\tau)$ and $\eta(\tau_{1}')$. By the $\delta$-slimness of the triangle, $\gamma(\sigma_{1}) \in \gamma|_{[\tau_{1}, \tau]}$ is $\delta$-close to either $[\gamma(\tau_{1}), \eta(\tau_{1}')]$ or $[\eta(\tau_{1}'), \gamma(\tau)]$. Meanwhile, the former one is contained in $\mathcal{N}_{3D}(\gamma(\tau_{1}))$, whereas $\gamma(\sigma_{1})$ is at least $97D$-far from $\gamma(\tau_{1})$. Hence, $\gamma(\sigma_{1})$ is $\delta$-close to some point $p \in [\eta(\tau_{1}'), \gamma(\tau)]$.

Next, we observe the geodesic triangle connecting $\eta(\tau_{1}'), \eta(\tau')$ and $\gamma(\tau)$. This time, $p$ is $\delta$-close to either $[\eta(\tau_{1}'), \eta(\tau')]$ or $[\eta(\tau'), \gamma(\tau)]$. The latter one is contained in $\mathcal{N}_{1.5D}(c_{i})$ and hence in $\mathcal{N}_{3D}(\gamma(\tau))$. Meanwhile, $\gamma(\sigma_{1})$ is $90D$-far from $\gamma(\tau)$, so $p$ is $89D$-far from $\gamma(\tau)$. Hence, $p$ cannot be $\delta$-close to $[\eta(\tau'), \gamma(\tau)]$, and is rather $\delta$-close to $[\eta(\tau_{1}'), \eta(\tau')]$.

In conclusion, $\gamma(\sigma_{1})$ is $2\delta$-close to some point $q \in \eta|_{[\tau_{1}', \tau']}$. This $q$ is $(D + 122\delta)$-close to $a_{j}$. It follows that $a_{j} \in \mathfrak{A}_{D}(a_{i}, b_{i})$.

In conclusion, $a_{j} \in \mathfrak{A}_{D}(a_{i}, b_{i}) \cap \mathfrak{A}_{D}(a_{i}, x_{0})$ and is closer than $c_{i}$ to $x_{0}$. This contradicts the minimality of $c_{i}$. 
\end{proof}

Thanks to the claim, we conclude that $\#\mathcal{B} \le \#\mathcal{U} + \#\mathcal{G}$ at each step. Recall that $a_{i} \in \mathcal{A}_{2}$ is declared good or bad at step $i$ and is not affected thereafter. Hence, after the last step, there is no element of $\mathcal{U}$ left. This means that $\#\mathcal{B} \le \#\mathcal{G}$, and $\mathcal{G}$ takes up at least half of $\mathcal{A}_{2}$.

Now, with the final $\mathcal{G}$ in hand, for each $a_{i} \in \mathcal{G}$ we define \[
K_{i} := X \setminus \big(\mathcal{H}_{D}(a_{i}, x_{0}) \cup \mathcal{H}_{D}(a_{i}, b_{i})\big).
\]
Since $a_{i} \in \mathcal{G} \subseteq \mathcal{A}_{2}$, we have $\#(K_{i} \cap A) \ge N$. The remaining claim is: 

\begin{claim}\label{claim:magicDisjt}
For every pair of distinct elements $a_{i}, a_{j} \in \mathcal{G}$, $K_{i}$ and $K_{j}$ do not intersect.
\end{claim}

To check this claim, suppose to the contrary that $K_{i}$ and $K_{j}$ has a common element $z$ for some $i < j$ such that $a_{i}, a_{j} \in \mathcal{G}$. This means that $(x_{0} | z)_{a_{i}}, (x_{0} | z)_{a_{j}} < D$. Now let $\gamma : [0, L] \rightarrow X$ be the geodesic connecting $x_{0}$ to $z$. Lemma \ref{lem:stability} guarantees timings $\tau, \tau' \in [0, L]$ such that $d_{X}(\gamma(\tau), a_{i}), d_{X}(\gamma(\tau'), a_{j}) < D+12\delta$. 

Recall that $a_{i}$ and $a_{j}$ are $100D$-apart. This implies that $|\tau - \tau'| > 97D$. If $\tau' \le \tau - 97D$, then we have \[
d_{X}(x_{0}, a_{j}) < d_{X}(x_{0}, a_{i}) - 97D + 2(D + 12\delta) \le d_{X}(x_{0}, a_{i}) - 90D,
\]
which contradicts our labelling convention of elements of $\mathcal{A}_{2}$. Hence, $\tau \le \tau' - 97D$ holds. In particular, $a_{j} \in \mathfrak{A}_{D}(a_{i}, x_{0})$.

Now note that $(x_{0} | z)_{a_{j}} < D$ and that \[(x_{0} | a_{i})_{a_{j}} \ge  \big(\gamma(0) \big| \gamma(\tau)\big)_{\gamma(\tau')} - d_{X}(\gamma(\tau), a_{i}) - d_{X}(\gamma(\tau'), a_{j}) \ge 97D - (2D + 24\delta) \ge 90D.
\]
Gromov's 4-point condition (Lemma \ref{lem:GromIneq}) tells us that $(a_{i} | z)_{a_{j}} \le D + 4\delta$. 

Meanwhile,  $(b_{i} | z)_{a_{i}} < D$ because $z \notin \mathcal{H}_{D} (a_{i}, b_{i})$. This time, $(z | a_{j})_{a_{i}}$ is similar to $(z | \gamma(\tau'))_{\gamma(\tau)} \ge 97D$; we have $(z| a_{j})_{a_{i}} \ge 90D$. Another application of Gromov's 4-point condition leads to $(b_{i} | a_{j})_{a_{i}} < D + 4\delta$.

Now, Lemma \ref{lem:stability} applies to the sequence $(b_{i}, a_{i}, a_{j}, z)$ as $d_{X}(a_{i}, a_{j}) \ge 90D \ge 2 (D+ 4\delta)$. We obtain a geodesic $\eta$ from $b_{i}$ to $z$ that passes through the $(D + 16\delta)$-neighborhoods of $a_{i}$ and $a_{j}$ in order. We conclude $a_{j} \in \mathfrak{A}_{D}(a_{i}, b_{i})$. 

In summary, $a_{j} \in \mathfrak{A}_{2} \cap \mathfrak{A}_{D}(a_{i}, x_{0}) \cap \mathfrak{A}_{D}(a_{i}, b_{i})$. This contradicts the goodness of $a_{i}$. Hence, $z$ cannot exist, and $K_{i}$ and $K_{j}$ are disjoint.

With Claim \ref{claim:magicDisjt} in hand, we have \[
\# A \ge \sum_{i : a_{i} \in \mathcal{G}} \# (K_{i} \cap A) \ge N \cdot \#\mathcal{G} \ge N \cdot \frac{\#\mathcal{A}_{2}}{2} \ge N \cdot \frac{\#\mathcal{A}_{1}}{2M} \ge \frac{1}{\epsilon}(\#A - \#A').
\]
This ends the proof.\end{proof}

Now suppose that a group $G$ is acting properly on $X \ni x_{0}$. Then the stabilizer of $x_{0}$ is finite, and the $G$-orbit of $x_{0}$ is uniformly locally finite. By Proposition \ref{prop:magic}, we conclude that:

\begin{prop}\label{prop:magicProper}
Let $X$ be a $\delta$-hyperbolic space with a basepoint $x_{0}$ and let $G$ be a group acting properly on $X$. Then for each $\epsilon, D>0$ there exists a constant $N = N(\epsilon, D, Y)$ such that for every finite set $A \subseteq G$ there exists a subset $A' \subseteq A$ satisfying:
\begin{enumerate}
\item $\#A' \ge (1-\epsilon) \#A$;
\item For each $a \in A'$ there exist halfspaces $\mathcal{H}_{1}, \mathcal{H}_{2} \subseteq  X$ rooted at $ax_{0}$ with radius parameter $D$ such that $\#\big( \{ g \in A: gx_{0} \notin  (\mathcal{H}_{1} \cup \mathcal{H}_{2})\} \big) \le N$.
\end{enumerate}
\end{prop}

\section{Supporting hyperplane lemma and the critical exponent $\gamma$} \label{section:supporting}

With an additional assumption that $G$ is non-elementary, the hyperbolic magic lemma implies the following supporting hyperplane lemma. Still, we will prove it for general acylindrical actions.

We will work with the following form of halfspaces: given $x, y \in X$, let \[
\mathcal{H}_{half}(x, y) := \{ z \in X : \textrm{$z$ is closer to $x$ than $y$}\}.
\]

\begin{prop}\label{prop:supportingWPD}
Let $X$ be a $\delta$-hyperbolic space with a basepoint $x_{0}$ and let $G$ be a non-virtually cyclic group acting on $X$ with a WPD loxodromic element $f$. Let $S$ be a finite generating set. Then there exists $D_{0}$ such that, for each $\epsilon>0$ and $D > D_{0}$ there exists a constant $N =N(\epsilon, D)$  such that for every finite set $A \subseteq G$ there exists a subset $A' \subseteq A$ satisfying:
\begin{enumerate}
\item $\#A' \ge (1-\epsilon) \#A$;
\item For each $a \in A'$ there exist $b \in G$ such that $d_{S}(a, b) \le N$ and 
 \[\begin{aligned}
\{gx_{0} : g \in A\}  &\subseteq  \mathcal{H}_{half}( bx_{0}, b\cdot f^{D} x_{0})
\end{aligned}
\]
and such that $\mathcal{H}_{half}(bx_{0}, bf^{D} x_{0})$ and $bf^{D} w f^{-D} b^{-1} \cdot \mathcal{H}_{half}(bx_{0}, bf^{D} x_{0})$ are disjoint.
\end{enumerate}
\end{prop}

This will follow from a weaker statement:

\begin{prop}\label{prop:supportingWPDWeak}
Let $X$ be a $\delta$-hyperbolic space and let $G$ be a non-virtually cyclic group acting on $X$ with a unital, axial WPD loxodromic element $f$. Let $x_{0} \in Ax(f)$. Let $S$ be a finite generating set. Then there exists $D_{0}$ such that, for each $\epsilon>0$ and $D > D_{0}$ there exists a constant $N =N(\epsilon, D)$  such that for every finite set $A \subseteq G$ there exists a subset $A' \subseteq A$ satisfying:
\begin{enumerate}
\item $\#A' \ge (1-\epsilon) \#A$;
\item For each $a \in A'$ there exist $b \in G$ such that $d_{S}(a, b) \le N$ and  
 \[
\# \big(\{gx_{0} : g \in A\}  \setminus  \mathcal{H}_{half}( bx_{0}, bf^{D} x_{0}) \big) \le N.
 \]
\end{enumerate}
\end{prop}

\begin{proof}
Let $K$ be as in Theorem \ref{thm:elemClos} and let $K_{0} \ge K+100\delta$. By enlarging $K_{0}$ if necessary, we may assume that $d_{X}(x_{0}, sx_{0}) \le K_{0}$ for each $s\in S$.

Note that $\{ g \in G : g^{2} \in EC(f)\}$ contains $EC(f)$ as an index-2 subgroup, which is virtually cyclic. Since $G$ is not virtually cyclic, we can take $w \in S \setminus \{ g \in G : g^{2} \in EC(f)\}$. Then Theorem \ref{thm:elemClos} guarantees that $\diam_{\gamma}(\gamma') \le K$ for distinct axes $\gamma, \gamma' \in \{ Ax(f), w^{-1} Ax(f), w Ax(f)\}$. By Corollary \ref{cor:projGrom}(5), $\diam_{\kappa}(\kappa') \le K_{0}$ for any subgeodesics $\kappa, \kappa'$ of $\gamma$, $\gamma'$ as well. 

We then have: \begin{obs}\label{obs:wFProj}
For $k, l \in \Z$ and  distinct $m, n \in \{1, 0, -1\}$, we have \[
\big(w^{m} f^{k} x_{0} \, \big|\, w^{n} f^{l} x_{0}\big)_{x_{0}} \le 6K_{0} + 8\delta.
\]
\end{obs}
To see this, note that $w^{m} [x_{0}, f^{k} x_{0}]$ has $K_{0}$-small projection onto $w^{n}[x_{0}, f^{l} x_{0}]$, which is $2d_{X}(w^{m} x_{0}, w^{n} x_{0})$-close to $w^{n} x_{0}$. It follows that the projection of $w^{m} f^{k} x_{0}$ onto $w^{n} [x_{0}, f^{l} x_{0}]$ is $(K_{0} + 2K_{0} |m-n|)$-close to $w^{n} x_{0}$, This implies \[
\big(w^{m} f^{k} x_{0} \, \big| \,  w^{n} f^{l} x_{0}\big)_{w^{n} x_{0}} \le 5K_{0} + 8\delta.
\]
Now the desired inequality follows from $d_{X}(x_{0}, w^{n} x_{0}) \le K_{0}$.

It follows that: \begin{obs}\label{obs:wIndepG}
There exists $K_{1} > 0$ such that for each $g \in G$ either\begin{enumerate}
\item  $(gx_{0} | f^{i} x_{0})_{x_{0}} < K_{1}$ for every $i \in \Z$ or 
\item $(g x_{0} | wf^{i} x_{0})_{x_{0}} < K_{1}$ for every $i \in \Z$.
\end{enumerate}
\end{obs}
We define $\mathscr{W} : G \rightarrow \{id, w\}$ using the above observation. Namely, for each $g \in G$ we pick $\mathscr{W}(g) \in \{id, w\}$ such that $(g^{-1} x_{0} |  \mathscr{W}(g)f^{i} x_{0})_{x_{0}} < K_{1}$ for each $i \in \Z$, i.e., $(x_{0} | g\mathscr{W}(g) f^{i} x_{0})_{gx_{0}} < K_{1}$. Let  $D_{0} = 10^{4}(K_{1} + K_{0} + \delta+1)$. 

We now begin the proof. Let $D > D_{0}$. Recall the definition of the elementary closure $EC(f)$ of $f$. Since $EC(f)$ is a finite extension of $\langle f \rangle$ and since $\{f^{i} x_{0} \}_{i \in \Z}$ is locally finite, the set\[
 \big\{ g \in EC(f) : d_{X}(x_{0}, gx_{0}) \le 2D +4\delta\big\}
\]
is finite. Hence, they are contained in $B_{S}(R') \subseteq G$ for some $R' > 0$.  Now let $R := R' + 3 + D \cdot \|f\|_{S}$.

Given $g\in G$ we define the \emph{anti-halfspace}\[
\mathfrak{A}_{\pm }(g) := \left\{ h \in G :  \begin{array}{c} \textrm{$\exists 0\le \tau_{1} \le \tau$, $\exists$ geodesic $\gamma : [0, \tau] \rightarrow X$ starting at $x_{0}$ such that}\\
 \textrm{$d_{X}\big(\gamma(\tau_{1}), \, g \mathscr{W}(g) f^{D} w^{\pm 1} f^{D} x_{0}\big) < 0.02D$}, \\
  \textrm{$d_{X}\big (\gamma(\tau), \,h\mathscr{W}(h) f^{D} x_{0}\big) < 0.02D$}.\end{array}  \right\}.
\]

\begin{obs}\label{obs:supportingObs0}
For each $g \in G$, each element $h \in \mathfrak{A}_{\pm}(g)$ satisfies that $d_{X}(x_{0}, hx_{0}) > d_{X}(x_{0}, gx_{0})+ 0.5D$. Moreover, $ \mathfrak{A}_{+}(g)$  and $\mathfrak{A}_{-}(g)$ are disjoint.
\end{obs}

\begin{proof}[Proof of Observation \ref{obs:supportingObs0}]
Suppose first that $h \in \mathfrak{A}_{+}(g)$. Let $ 0\le \tau_{1} \le \tau$ and let $\gamma : [0, \tau]\rightarrow X$ be the geodesic realizing the membership of $h$ in $\mathfrak{A}_{+}(g)$. Now for the sequence \[
(y_{0}, y_{1}, y_{2}, y_{3}, y_{4}) := 
\big(x_{0}, \,g\mathscr{W}(g) x_{0},  \,g\mathscr{W}(g) f^{D}x_{0},  \,g\mathscr{W}(g) f^{D} w f^{D} x_{0}, \,\gamma(\tau_{1})\big),
\]
we observe that \begin{enumerate}
\item $(y_{0} | y_{2})_{y_{1}} \le 0.01D, (y_{1} | y_{3})_{y_{2}} \le 0.01D, (y_{2} | y_{4})_{y_{3}} \le 0.02D$.
\item $d_{X}(y_{1}, y_{2}), d_{X}(y_{2}, y_{3}) \ge 0.95D$.
\end{enumerate}
The first item follows from the fact that $(y_{0} | y_{2})_{gx_{0}} \le K_{1}$, $d_{X}(y_{1}, gx_{0}) < K_{0}$, $(y_{1} | y_{3})_{y_{2}} \le 6K_{0} + 8\delta$, $d_{X}(y_{3}, y_{4}) \le 0.02D$. The second item is due to the fact $d_{X}(x_{0}, f^{D} x_{0}) = D$ and $d_{X} (x_{0}, sx_{0}) \le K_{0}$.

By Lemma \ref{lem:stability}, there exist $0 \le t_{1} \le t_{2} \le t_{3} \le \tau_{1}$ such that \begin{equation} \label{eqn:closegW}
\begin{aligned}
d_{X}\big( \gamma(t_{1}), y_{1} \big), d_{X}\big( \gamma(t_{2}), y_{2}\big) \le 0.011D, \,\,d_{X}\big( \gamma(t_{3}), y_{3}\big) \le 0.021D.
\end{aligned}
\end{equation}
This implies that
\[\begin{aligned}
\tau_{1} &= d_{X}(x_{0}, \gamma(\tau_{1})) = d_{X}( x_{0}, \gamma(t_{1})) + d_{X}( \gamma(t_{1}), \gamma(t_{2})) + d_{X}( \gamma(t_{2}), \gamma(t_{3})) + d_{X}( \gamma(t_{3}),  \gamma(\tau))\\ 
&\ge d_{X}(x_{0}, y_{1}) + d_{X}(y_{1}, y_{2}) + d_{X}(y_{2}, y_{3}) -2 (0.011D + 0.011D + 0.021D)\\
&\ge \big(d_{X}(x_{0}, gx_{0}) - d_{X}(gx_{0}, g\mathscr{W}(g)x_{0})\big) + 1.9D - 0.09D \ge d_{X}(x_{0}, gx_{0}) + 1.8D.
\end{aligned}
\]
Meanwhile, note that \[\begin{aligned}
\tau_{1} \le \tau &\le d_{X}(x_{0}, h\mathscr{W}(h)f^{D}x_{0}) + 0.02D \\
&\le d_{X}(x_{0}, hx_{0}) +  d_{X}(hx_{0},h\mathscr{W}(h)f^{D}x_{0}) + 0.02D& \\
&\le d_{X}(x_{0}, hx_{0}) + d_{X}(x_{0}, f^{D}x_{0}) + 0.021D = d_{X}(x_{0}, hx_{0}) + 1.021D.
\end{aligned}
\]Comparing these two inequalities lead to $d_{X}(x_{0}, gx_{0}) +0.5D<d_{X}(x_{0}, hx_{0})$.

Note that $\big( \gamma(t_{2})\, \big| \, \gamma(\tau) \big)_{\gamma(\tau_{1})} = 0$ as $t_{2} \le \tau_{1} \le \tau$. Since $y_{2}$, $y_{3}$ and $h\mathscr{W}(h) f^{D} x_{0}$ are $0.02D$-close to $\gamma(t_{2})$, $\gamma(\tau)$ and $\gamma(\tau')$, respectively, we have \begin{equation}\label{eqn:gWgWideOpen}
\big(g\mathscr{W}(g) f^{D} x_{0} \, \big| \, h\mathscr{W}(h) f^{D} x_{0} \big)_{ g \mathscr{W}(g) f^{D} w f^{D} x_{0}} < 0.06D.
\end{equation}

Now, let us suppose that $h \in \mathfrak{A}_{-} (g)$ in addition and deduce contradiction. Just as we had Inequality \ref{eqn:gWgWideOpen}, we have  \begin{equation}\label{eqn:gWgWideOpen2}
(g\mathscr{W}(g) f^{D} x_{0} | h\mathscr{W}(h)f^{D} x_{0})_{g\mathscr{W}(g) f^{D} w^{-1} f^{D} x_{0}} \le 0.06D.
\end{equation}
Finally, Observation \ref{obs:wFProj} tells us that \[
(g\mathscr{W}(g)f^{D} wf^{D} x_{0} | g\mathscr{W}(g)f^{D} w^{-1} f^{D} x_{0} )_{g\mathscr{W}(g)f^{D} x_{0}} \le 0.01D.
\]
Furthermore, we know that $d_{X}(g\mathscr{W}(g)f^{D} x_{0}, g\mathscr{W}(g) f^{D}w^{\pm 1} f^{D} x_{0} ) \ge 0.95D$. By Lemma \ref{lem:stability}, we have $d_{X}(h\mathscr{W}(h) f^{D} x_{0}, h\mathscr{W}(h)f^{D} x_{0}) \ge 0.95D + 0.95D - 2\cdot (0.06D + 0.01 D + 0.06D + 100\delta) \ge 1.8D$, a contradiction.
\end{proof}

We then observe:

\begin{obs}\label{obs:supportingObs1}
Suppose that $\mathfrak{A}_{+}(g)$ and $\mathfrak{A}_{+}(h)$  has nonempty intersection, and suppose that $d_{S}(g, h) > R$. Then either $g \in \mathfrak{A}_{+}(h)$ or $h \in \mathfrak{A}_{+}(g)$.
\end{obs}

\begin{proof}[Proof of Observation \ref{obs:supportingObs1}]
Let us pick an element $a \in \mathfrak{A}_{+}(g) \cap \mathfrak{A}_{+}(h)$.

Without loss of generality, we suppose that $d_{X}(x_{0}, gx_{0}) \ge d_{X}(x_{0}, hx_{0})$. Let $\gamma$ ($\gamma'$, resp.) be the geodesic and let $\tau_{1} \le \tau$ ($\sigma_{1}, \sigma$, resp.) be the timing that realize the membership of $a$ in $\mathfrak{A}_{+}(g)$ ($\mathfrak{A}_{+}(h)$, resp.). Then $\tau$ and $\sigma$ differ by at most $0.04D$, as \[
\tau =_{0.02D} d_{X}\big(x_{0}, a \mathscr{W}(a) f^{D} x_{0} \big) =_{0.02D} \sigma.
\]

We now let $L_{min} := \big(\gamma(\tau)\, \big|\, \gamma'(\sigma)\big)_{x_{0}}$, which satisfies \[
L_{min} : = d\big(x_{0}, \gamma(\tau)\big) - \big(x_{0} \, \big|\, \gamma'(\sigma) \big)_{\gamma(\tau)} \ge \tau - 0.04 D.
\]
Similarly, $L_{min}$ is greater than $\sigma - 0.04D$. By Lemma \ref{lem:fellow}, we have \[
d_{X}(\gamma(t), \gamma'(t)) < 4\delta \quad (0 \le t \le L_{min}).
\]

We observed earlier that Lemma \ref{lem:stability} applies to the sequence  \[
\big(x_{0},\, g\mathscr{W}(g) x_{0},\, g\mathscr{W}(g) f^{D}x_{0},\,  g\mathscr{W}(g) f^{D} w f^{D} x_{0}, \,\gamma(\tau_{1})\big).
\]
In particular, there exist $0\le t_{1} \le t_{2} \le  \tau_{1}$ such that \[\begin{aligned}
d_{X}\big(g\mathscr{W}(g) x_{0}, \, \gamma(t_{1})\big) , d_{X}\big( g\mathscr{W}(g) f^{D}x_{0}, \,\gamma(t_{2})\big) &\le 0.011D.
\end{aligned}
\]
Note that \[
|t_{1}-d_{X}(x_{0}, gx_{0})| \le 0.011D + d_{X}(x_{0}, \mathscr{W}(g)x_{0}) \le 0.012D.
\]

By Lemma \ref{lem:GromHausdorff}, $g\mathscr{W}(g)[x_{0}, f^{D} x_{0}]$ and $\gamma([t_{1}, t_{2}])$ are $0.012D$-equivalent. This forces $t_{2} - t_{1} =_{0.022D} d_{X}(x_{0}, f^{D}x_{0}) = D$. Similarly, $g\mathscr{W}(g) f^{D} w [x_{0}, f^{D}x_{0}]$ and $\gamma([t_{2}, \tau_{1}])$ are $0.021D$-equivalent and  $\tau_{1} - t_{2} =_{0.032D} D$.

Similarly, there exist $0 \le s_{1} \le s_{2} \le  \sigma_{1}$ for $\gamma'$ such that \[
d_{X}\big(h \mathscr{W}(h) x_{0}, \gamma(s_{1}) \big) d_{X}\big( h \mathscr{W}(h) f^{D} x_{0}, \gamma'(s_{2}) \big) \le 0.01D
\]
 Note here that $t_{2}$ or $s_{2}$ are much smaller than $\tau_{1}$ or $\sigma_{1}$, respectively, so they are smaller than $L_{min}$. In particular, $d_{X}(\gamma(t_{2}), \gamma'(t_{2})) < 4\delta$ holds. Moreover, $s_{1}$ is $0.012D$-close to $d_{X}(x_{0}, hx_{0})$. Since we assumed that $hx_{0}$ is closer than $gx_{0}$ to $x_{0}$, we obtain \[
t_{1} \ge s_{1} - 0.024D.
\]

Observe that the geodesic $\gamma'$ and the two timing $\sigma_{1}, t_{2}$ satisfy\[\begin{aligned}
d_{X}\big(\gamma'(\sigma_{1}), h \mathscr{W}(h)f^{D} w f^{D}  x_{0}\big) &< 0.02D, \\
d_{X}\big(\gamma'(t_{2}), g\mathscr{W}(g)f^{D} x_{0}\big) &\le 
d_{X}(\gamma'(t_{2}), \gamma(t_{2})) +
d_{X}\big(\gamma(t_{2}), g\mathscr{W}(g)f^{D} x_{0}\big)\\
& \le 4\delta + 0.011D < 0.02D.
\end{aligned}
\]
In the remaining, we will show that $\sigma_{1} \le t_{2}$. This will guarantee  $g \in \mathfrak{A}_{+}(h)$ and end the proof.

Suppose to the contrary that  $\sigma_{1}  > t_{2}$. Then we have \[
\big[t_{1}+0.024D, t_{2}\big] \subseteq [s_{1}, \sigma_{1}] = [s_{1}, s_{2}] \cup[s_{2}, \sigma_{1}].
\]
In particular, one of $[s_{1}, s_{2}]$ and $[s_{2}, \sigma_{1}]$ should overlap with $[t_{1}+0.024D, t_{2}]$ for length at least $\frac{1}{2} (t_{2} - t_{1} - 0.024D) \ge 0.46D$.
\begin{enumerate}
\item $I:=  [t_{1}, t_{2}] \cap[s_{1}, s_{2}]$ is longer than $0.46D$. Recall that $g\mathscr{W}(g)[x_{0}, f^{D}x_{0}]$ and $\gamma([t_{1}, t_{2}])$ are $0.012D$-equivalent. Hence, there exist $p_{1}, q_{1} \in g\mathscr{W}(g)[x_{0}, f^{D} x_{0}]$ that are $0.012D$-close to $\gamma(\min I)$ and $\gamma(\max I - 0.05D)$, respectively. We can also take $p_{2}, q_{2}\in h \mathscr{W}(h) [x_{0}, f^{D} x_{0}]$ that are $0.012D$-close to $\gamma'(\min I)$ and $\gamma'(\max I- 0.05D)$, respectively.  

Since  $\min I \le \max I - 0.05D \le\tau - 0.05D \le L_{min}$,  $\gamma$ and $\gamma'$ are $4\delta$-fellow traveling at $t = \min I, \max I - 0.05D$. We thus have \[
d_{X}(p_{1}, p_{2}), d_{X}(q_{1}, q_{2}) \le 0.025D, \quad d_{X}(p_{2}, q_{2}) \ge |I| - 0.05D - 2\cdot 0.012D \ge 0.4D.
\]
This means that the projection of $\{p_{1}, q_{1}\} \subseteq g\mathscr{W}(g) [x_{0}, f^{D}x_{0}]$ onto $h\mathscr{W}(h)[x_{0}, f^{D}x_{0}]$ is larger than $0.2D$. By Corollary \ref{cor:projGrom}(5), we have $d_{h\mathscr{W}(h) Ax(f)} (g\mathscr{W}(g) Ax(f)) \ge 0.1D \ge K_{0}$. Theorem \ref{thm:elemClos} implies that $\mathscr{W}(g)^{-1} \cdot g^{-1} h \mathscr{W}(h) \in EC(f)$. Moreover, note that \[\begin{aligned}
d_{X}(g\mathscr{W}(g)x_{0}, h\mathscr{W}(h)x_{0}) &\le d_{X}(g\mathscr{W}(g)x_{0}, p_{1}) + d_{X}(p_{1}, p_{2}) + d_{X}(p_{2},  h\mathscr{W}(h)x_{0})\\
&\le d_{X}(x_{0}, f^{D} x_{0}) + 4\delta + d_{X}(x_{0}, f^{D}x_{0}).
\end{aligned}
\]
In summary, we have \[
\mathscr{W}(g)^{-1} \cdot g^{-1} h \mathscr{W}(h) \in EC(f) \cap \{ u \in G : d_{X}(x_{0},ux_{0}) \le 2D + 4\delta\}.
\] In other words, $\mathscr{W}(g)^{-1} \cdot g^{-1} h \mathscr{W}(h)  \in B_{S}(R')$ and $g^{-1} h \in B_{S}(R)$. This is a contradiction.

\item $I:= [t_{1}, t_{2}] \cap [s_{2}, \sigma_{1}]$ is longer than $0.46D$. In this case, we can similarly take points $p_{1}, q_{1} \in g\mathscr{W}(g)[x_{0}, f^{D} x_{0}]$ that are $0.012D$-close to $\gamma(\min I)$ and $\gamma(\max I-0.05D)$, respectively. We can also pick $p_{2}, q_{2}$ on $h \mathscr{W}(h) f^{D} w [x_{0}, f^{D} x_{0}]$ that are $0.021D$-close to $\gamma'(\min I)$ and $\gamma'(\max I- 0.05D)$, respectively. Again, $d_{X}(\gamma(t), \gamma'(t)) \le 4\delta$ for $t = \min I, \max I - 0.05D$. Then we have \[
d_{X}(p_{1}, p_{2}), d_{X}(q_{1}, q_{2}) \le 0.034D, \quad d_{X}(p_{2}, q_{2}) \ge |I| - 0.05D - 2 \cdot 0.021D \ge 0.36D.
\]
Then the projection of $\{p_{1}, q_{1}\} \subseteq g\mathscr{W}(g) [x_{0}, f^{D}x_{0}]$ onto $h\mathscr{W}(h) f^{D} w [x_{0}, f^{D}x_{0}]$ is larger than $0.2D$. By Corollary \ref{cor:projGrom}(2), $g\mathscr{W}(g) Ax(f)$ has projection $\ge 0.1D$ onto $h \mathscr{W}(h) f^{D} w Ax(f)$. Theorem \ref{thm:elemClos} implies that $\mathscr{W}(g)^{-1} \cdot g^{-1} h \mathscr{W}(h) f^{D} w \in EC(f)$. Moreover, note that \[\begin{aligned}
d_{X}(g\mathscr{W}(g)x_{0}, h\mathscr{W}(h)f^{D} w x_{0}) &\le d_{X}(g\mathscr{W}(g)x_{0}, p_{1}) + d_{X}(p_{1}, p_{2}) + d_{X}(p_{2},  h\mathscr{W}(h)f^{D} w x_{0})\\
&\le d_{X}(x_{0}, f^{D} x_{0}) + 4\delta + d_{X}(x_{0}, f^{D}x_{0}).
\end{aligned}
\]
In summary, we have \[
\mathscr{W}(g)^{-1} \cdot g^{-1} h \mathscr{W}(h)f^{D} w  \in EC(f) \cap \{ u \in G : d_{X}(x_{0}, ux_{0}) \le 2D + 4\delta\}.
\] In other words, $\mathscr{W}(g)^{-1} \cdot g^{-1} h \mathscr{W}(h)  f^{D} w \in B_{S}(R')$ and $g^{-1} h \in B_{S}(R)$. This is a contradiction.
\end{enumerate}
Hence, neither  situation can happen and we conclude $g \in \mathfrak{A}_{+}(h)$.
\end{proof}

For the same reason, we have \begin{obs}\label{obs:supportingObs1.5}
Let $\epsilon, \epsilon' \in \{+, -\}$. Suppose that $\mathfrak{A}_{\epsilon}(g)$ and $\mathfrak{A}_{\epsilon'}(h)$  has nonempty intersection, and suppose that $d_{S}(g, h) > R$. Then either $g \in \mathfrak{A}_{\epsilon'}(h)$ or $h \in \mathfrak{A}_{\epsilon}(g)$.
\end{obs}

We finally need:

\begin{obs}\label{obs:supportingObs2}
For each $g \in G$, we have \[
G \setminus \mathfrak{A}_{\pm}(g) \subseteq 
\mathcal{H}_{half}\big(g \mathscr{W}(g) f^{D} w^{\pm 1} f^{2D} x_{0}, g \mathscr{W}(g) f^{D} w^{\pm 1} f^{3D} x_{0}\big).
\]
\end{obs}

\begin{proof}[Proof of Observation \ref{obs:supportingObs2}]
For convenience, let us write $\mathfrak{g} := g \mathscr{W}(g) f^{D} w^{\pm 1}$. 

Let us pick $u \in G \setminus  \mathcal{H}_{half}(\mathfrak{g}f^{2D} x_{0}, \mathfrak{g} f^{3D} x_{0})$.  Let $k, l \in \Z$ be such that \[
\pi_{\mathfrak{g} Ax(f)}(ux_{0}) \cap[\mathfrak{g}f^{k}, \mathfrak{g}f^{k+1}] \neq \emptyset, \quad 
\pi_{\mathfrak{g} Ax(f)}(u \mathscr{W}(g) f^{D}x_{0}) \cap[\mathfrak{g}f^{l}, \mathfrak{g}f^{l+1}] \neq \emptyset.
\]
 Corollary \ref{cor:projGrom}(2) tells us that \[d_{X}(\mathfrak{g}f^{i} x_{0}, ux_{0}) =_{30\delta+1} d_{X}(\mathfrak{g}f^{i}x_{0},  \mathfrak{g}f^{k} x_{0}) + d_{X}( \mathfrak{g} f^{k} x_{0}, ux_{0}) = |i-k| +d_{X}( \mathfrak{g} f^{k} x_{0}, ux_{0}). \quad (\forall i \in \Z)
\]
Since $ux_{0}$ is not closer to $\mathfrak{g}f^{2D} x_{0}$ than to $\mathfrak{g}f^{3D} x_{0}$, we conclude that $k \ge 2.45D$. Meanwhile, note that $d_{X}(ux_{0}, u \mathscr{W}(u) f^{D}x_{0}) \le 1.01D$. By Corollary \ref{cor:projGrom}(1), $l \ge 1.42D$. This means that \begin{equation}\label{eqn:abitToo}
\big(\mathfrak{g} f^{D}x_{0} \, \big| \,  u \mathscr{W}(u) f^{D}x_{0} \big)_{\mathfrak{g} f^{2D}x_{0}} \le 0.6D.
\end{equation}
Now for \[
(y_{0}, y_{1}, y_{2}, y_{3}, y_{4}) := 
\big(x_{0}, \, g\mathscr{W}(g) x_{0}, g \mathscr{W}(g)f^{D} x_{0}, \mathfrak{g}f^{D}x_{0}, \mathfrak{g} f^{2D}x_{0})
\]
we have $(y_{i-1}|y_{i+1})_{y_{i}} \le 0.01D$ for $i=1, 2, 3$ and $d_{X}(y_{i-1}, y_{i}) \ge 0.95D$ for $i=2, 3, 4$. Combining this with Inequality \ref{eqn:abitToo}, we can apply Lemma \ref{lem:stability} and conclude that $[x_{0}, u \mathscr{W}(u) f^{D}x_{0} ]$ is $0.011D$-close to $g \mathscr{W}(g) f^{D} w^{\pm 1} f^{D} x_{0}$. Hence, $x \in \mathfrak{A}^{\pm}(g)$.
\end{proof}

This time, we will define \[
\mathcal{A}_{1} := \{ a \in A : \# \big(A\setminus \mathcal{H}_{half}(vx_{0}, vf^{D} x_{0}) \big) \ge N \,\,\textrm{for $v = a \mathscr{W}(a)f^{D} w f^{2D}, a \mathscr{W}(a)f^{D} w^{-1} f^{2D}$} \}.
\]

We now pick a subset $\mathcal{A}_{2}$ of $\mathcal{A}_{1}$ that is maximally $R$-separated \emph{in the word metric $d_{S}$}, i.e., we have \begin{enumerate}
\item $d_{S}(a, a') \ge R$ for each pair of distinct elements $a, a' \in \mathcal{A}_{2}$;
\item $\mathcal{A}_{2}$ is a maximal subset of $\mathcal{A}_{1}$ satisfying this property.
\end{enumerate}
Then $\bigcup_{a \in \mathcal{A}_{2}} (a \cdot B_{S}(R)\cap A)$ covers entire $\mathcal{A}_{1}$. We conclude \[
\#\mathcal{A}_{2} \ge \frac{1}{\#B_{S}(R)} \cdot \#\mathcal{A}_{1} \le \frac{1}{(2\#S)^{R}} \#\mathcal{A}_{1}.
\]

As before, we prepare empty collections $\mathcal{B}= \mathcal{U} = \mathcal{G} = \emptyset$. Enumerate $\mathcal{A}_{2}$ by the distance from $x_{0}$, i.e., let $\mathcal{A}_{2} = \{a_{1}, a_{2}, \ldots, a_{\#\mathcal{A}_{2}}\}$ be such that $d_{X}(x_{0}, a_{i}) \le d_{X}(x_{0}, a_{i+1})$ for each $i$. At each step $i=1, \ldots, \#\mathcal{A}_{2}$, we will put $a_{i}$ in either $\mathcal{B}$ or $\mathcal{G}$; this decision is final and shall not be modified further. We may put some other elements of $\mathcal{A}_{2}$ in $\mathcal{U}$, whose their classification will change later. When $a_{i}$ is declared good, then we will also define its sign $\sigma(a_{i}) \in \{+1, -1\}$. This way, we will obtain a  function $\sigma : \mathcal{G} \rightarrow \{+1, -1\}$ in the end.

We will keep the balance $\#\mathcal{B} \le \#\mathcal{U} + \#\mathcal{G}$ throughout. Finally, after the last step there will be no $\mathcal{U}$-element. At the end we will have $\#\mathcal{B} \le \#\mathcal{G}$.

We now describe the procedure. At step $i$, \begin{enumerate}
\item if $\mathcal{A}_{2} \cap \mathfrak{A}_{+}(a_{i})$ has no element, then we declare $a_{i} \in \mathcal{G}$ and $\sigma(a_{i}) = +1$;
\item if not (1) and if $\mathcal{A}_{2} \cap \mathfrak{A}_{-}(a_{i})$ has no element, then we declare $a_{i} \in \mathcal{G}$ and $\sigma(a_{i}) = -1$;
\item if not (1) and (2), we pick $b_{i} \in \mathcal{A}_{2} \cap \mathfrak{A}_{+}(a_{i})$ and $c_{i} \in \mathcal{A}_{2} \cap \mathfrak{A}_{-}(a_{i})$ whose orbit points are the closest to $x_{0}$. We declare $a_{i} \in \mathcal{B}$ and $b_{i}, c_{i} \in \mathcal{U}$.
\end{enumerate}

Till step $i$, $\mathcal{G} \cup \mathcal{B}$ comprises of elements from $\{a_{1}, \ldots, a_{i}\}$; they do not contain any of $a_{i+1}, a_{i+2}, \ldots$. ($\ast$) We now describe what happens at step $i$.

In case (1) or (2), $\mathcal{G}$ gains one more element that might be from  $\mathcal{U}$ or not. $\mathcal{B}$ does not change. Overall, $\#\mathcal{B}$ stays the same and $\#\mathcal{U} + \#\mathcal{G}$ does not decrease. Similar situation happens in Case (2-a).

In case (2-b), $\mathcal{B}$ gains one element $a_{i}$, which might be from $\mathcal{U}$. In exchange, $\mathcal{U}$ gains elements $b_{i}$ and $c_{i}$. Here Observation \ref{obs:supportingObs0} guarantees that $d_{X}(x_{0}, b_{i}x_{0}), d_{X}(x_{0}, c_{i}x_{0}) > d_{X}(x_{0}, a_{i}x_{0})$ and that $b_{i}$ and $c_{i}$ are distinct. Since $\mathcal{A}_{2}$ was labelled with respect to the distance from $x_{0}$, we conclude that $b_{i}, c_{i} \in \{a_{i+1}, a_{i+2}, \ldots\}$; in other words, neither $b_{i}$ nor $c_{i}$ come from $\mathcal{G} \cup \mathcal{B}$. Hence, we conclude that elements in $\mathcal{G} \cup \mathcal{B}$ are never re-classified.

As before, we need to show that for each $i<j$ such that $a_{i}, a_{j} \in \mathcal{B}$, $\{b_{i}, c_{i}\}$ and $\{b_{j}, c_{j}\}$ are disjoint. Suppose to the contrary that $b_{i} = b_{j}$. Then Observation \ref{obs:supportingObs1} tells us that either (1) $a_{i} \in \mathfrak{A}_{+}(a_{j})$ or (2) $a_{j} \in \mathfrak{A}_{+}(a_{i})$. In the latter case, we have $d_{X}(x_{0}, a_{i}) > d_{X}(x_{0}, a_{i})$, contradicting the labelling scheme. In the former case, we have \[
d_{X}(x_{0}, a_{i} x_{0}) < d_{X}(x_{0}, a_{j}x_{0}) < d_{X}(x_{0}, b_{i} x_{0})
\]
by Observation \ref{obs:supportingObs0}. This violates the minimality of $b_{i}$. Hence, $b_{i} = b_{j}$ cannot happen. Likewise, using Observation \ref{obs:supportingObs1.5} we can exclude the cases $b_{i} = c_{j}$, $c_{i} = b_{j}$ and $c_{i} = c_{j}$. Thus, $\{b_{i}, c_{i}\}$ and $\{b_{j}, c_{j}\}$ are disjoint.

By the same logic as in the previous proof, we have $\#\mathcal{G} \ge \#\mathcal{B}$ in the end. Moreover, for distinct $a, b \in \mathcal{G}$, $\mathfrak{A}_{\sigma(a)} (a)$ and $\mathfrak{A}_{\sigma(b)}(b)$ are disjoint; if not, Observation \ref{obs:supportingObs1.5} implies either $b \in \mathfrak{A}_{\sigma(a)}(a)$ or $a \in \mathfrak{A}_{\sigma(b)}(b)$, contradicting the goodness of $a$ and $b$.

Since $X \setminus  \mathcal{H}_{half}(a \mathscr{W}(a) f^{D} w^{\sigma(a)} f^{2D} x_{0}, a \mathscr{W}(a) f^{3D} w^{\sigma(a)} f^{2D} x_{0}) \subseteq \mathcal{A}_{\sigma(a)}(a)$ has at least $N$ elements for $a \in \mathcal{G} \subseteq \mathcal{A}_{1}$, we conclude \[
\#\mathcal{A}\ge N \cdot \#\mathcal{G} \ge N \cdot \frac{\#\mathcal{A}_{2}}{2} \ge \frac{N}{(2\#S)^{R}}\#\mathcal{A}_{1} \ge \frac{1}{\epsilon} \#\mathcal{A}_{1}.
\]
This ends the proof.
\end{proof}

We now need:

\begin{lem}\label{lem:pigeon}
Let $X$ be a $\delta$-hyperbolic space with a basepoint $x_{0}$ and let $G$ be non-virtually cyclic group acting on $X$ with an axial WPD loxodromic element $f$. Let $x_{0} \in Ax(f)$. Then there exists $D_{0}$ such that the following holds for each $D>D_{0}$.

Let $A\subseteq X$ be a finite set in $ X \setminus \mathcal{H}_{half}(x_{0}, f^{D} x_{0})$ and let $N = \#A$. Then there exists $a_{1}, \ldots, a_{N} \in \{f^{D}, wf^{D}\}$ such that \begin{enumerate}
\item$H:= \mathcal{H}_{half}(f^{D} a_{1}\cdots a_{N} x_{0}, f^{D} a_{1} \cdots a_{N} f^{D} x_{0})$ contains $A$, and
\item $H$ and $f^{D} a_{1} \cdots a_{N} \cdot f^{D} \cdot w \cdot f^{-D} \cdot  a_{N}^{-1} \cdots a_{1}^{-1} f^{-D} H$ are disjoint.
\end{enumerate}
\end{lem}

\begin{proof}
Let $K_{0}, K_{1}, D_{0}$ be as in the proof of Proposition \ref{prop:supportingWPDWeak}. Suppose $D>D_{0}$. We claim that the $2^{N}$ halfspaces \begin{equation}\label{eqn:halfspaces2N}
\Big\{ X \setminus \mathcal{H}_{half}(f^{D} a_{1}\cdots a_{N} x_{0}, f a_{1}^{D} \cdots a_{N} f^{D} x_{0}) : a_{1}, \ldots, a_{N} \in \{f^{D}, w f^{D} \} \Big\}
\end{equation}
are mutually disjoint subsets of $X \setminus \mathcal{H}_{half}(x_{0}, f^{D} x_{0})$.

To see the disjointness, let $z \notin \mathcal{H}_{half}(f^{D} a_{1}\cdots a_{N} x_{0}, f a_{1}^{D} \cdots a_{N} f^{D} x_{0})$ and $z' \notin \mathcal{H}_{half}(f^{D} b_{1}\cdots b_{N} x_{0}, f b_{1}^{D} \cdots b_{N} f^{D} x_{0})$ for some $(a_{1}, \ldots, a_{N})\neq (b_{1}, \ldots, b_{N}) \in \{f^{D}, w f^{D}\}^{N}$. Let $m$ be the minimal one such that $a_{m} \neq b_{m}$. 

Let $z_{i} := f^{D} a_{1} \cdots a_{i} x_{0}$ and $z_{i}' := f^{D} b_{1} \cdots b_{i} x_{0}$ for $i \ge m-1$. We observe that Lemma \ref{lem:stability} applies to \[
( z, z_{N}, z_{N-1}, \ldots, z_{m}, z_{m-1} := z_{m-1}', z_{m}', \ldots, z_{N}', z').
\]
Indeed, we check that \[\begin{aligned}
(z_{i} | z_{i-2})_{z_{i-1}}, (z_{i}' | z_{i-2}')_{z_{i-1}'} &\le 6K_{0} +8\delta\le 0.01D& (i=m+1, \ldots, N), \\
(z | z_{N-1})_{z_{N}}, (z' | z'_{N-1})_{z'_{N}} &\le \frac{1}{2} d_{X}(x_{0}, f^{D}x_{0}) & \le 0.5D,& \\
(z_{m} | z_{m}')_{z_{m-1}}& \le 0.01D,&
\\ d_{X}(z_{i}, z_{i-1}), d_{X}(z'_{i}, z_{i-1}') &\ge 0.99D  &(i=m, \ldots, N).
\end{aligned}
\]
Consequently, we have $d_{X}(z, z') \ge 2 \cdot 0.9D$ and $z \neq z'$.

For the same reason, for each $a_{1}, \ldots, a_{N} \in \{f^{D}, wf^{D}\}$,  $\mathcal{H}_{half}(x_{0}, f^{D}x_{0})$ and $X \setminus \mathcal{H}_{half}( f^{D} a_{1} \cdots a_{N} x_{0}, f^{D} a_{1} \cdots a_{N} f^{D} x_{0})$ are disjoint. This implies that the latter is contained in $X \setminus \mathcal{H}_{half}(x_{0}, f^{D}x_{0})$. Hence, the sets in Display \ref{eqn:halfspaces2N} are indeed $2^{N}$ disjoint subsets of $X \setminus \mathcal{H}_{half}(x_{0}, f^{D}x_{0}$. One of them should avoid $A$ by the pigeonhole principle. Item (1) of the conclusion now follows.

Moreover, a similar logic shows that $X \setminus \mathcal{H}_{half}( x_{0}, f^{-D}x_{0})$ and $X \setminus \mathcal{H}_{half}(wx_{0}, wf^{-D} x_{0})$ are disjoint, as $Ax(f)$ and $wAx(f)$ have $K_{0}$-bounded projections onto each other. This leads to Item (2) of the conclusion.
\end{proof}

Proposition \ref{prop:supportingWPD} now follows from Proposition \ref{prop:supportingWPDWeak} and Lemma \ref{lem:pigeon}. Therefore acylindrically hyperbolic groups satisfy the assumption of Theorem \ref{thm:hutchcroft1plus2}.

\section{Branching set} \label{section:branching}

Recall the notions of barriers and roughly branching sets (Definition \ref{dfn:barrier}, \ref{dfn:branching}). Recall that $\mathcal{H}_{R}(x_{0}, y) := \{ z \in X : (z | y)_{x_{0}} > R\}$. Our aim is to show:

\begin{prop}\label{prop:barrierContracting}
Let $X$ be a $\delta$-hyperbolic space and let $G \le \Isom(X)$ be a non-virtually cyclic group with a unital, axial WPD element $f \in G$. Let $x_{0} \in Ax(f)$. Let $S$ be a finite generating set of $G$. Then there exists $r>0$ such that the following holds.

Let $R>0$ and $y \in X$. Then there exists an $r$-branching subset $B = B_{1} \sqcup \ldots \sqcup B_{R/r} \subseteq G$ such that, for every $g \in G$ such that $gx_{0} \in \mathcal{H}_{R}(x_{0}, y)$, every $d_{S}$-path connecting $id$ to $g$ passes through each of $B_{1}, \ldots, B_{R/r}$.
\end{prop}

\begin{proof}
Let $K_{0} = K > 1000\delta$ be as in Theorem \ref{thm:elemClos} for $G$ and $f$. Recall that $EC(f)$ is a virtually cyclic subgroup of $G$. Now let  \[
\mathscr{A} := \{ g Ax(f) : g \in G\}.
\]
Note that  $\diam_{\gamma}(\gamma') \le K_{0}$ for distinct axes $\gamma, \gamma' \in \mathscr{A}$.

By enlarging $K_{0}$, we can guarantee that $d_{X}(x_{0}, sx_{0}) < K_{0}$ for each $s \in S$. Since $G$ is not virtually cyclic, we can take $w \in S \setminus \{ g \in G : g^{2} \in EC(f)\}$. Then we have $d_{X}(x_{0}, wx_{0}) \le K_{0}$. By enlarging $K_{0}$ once again, we can guarantee the following: 

\begin{obs}\label{obs:2DivergingChoice}for every $x_{1}, x_{2} \in X$ either
\begin{enumerate}
\item $(x_{j} \, | \, p)_{x_{0}} < K_{0}$ for each $j \in \{1, 2\}$ and $p \in Ax(f)$; 
\item $(x_{j} \, | \,w p)_{x_{0}} < K_{0}$ for each $j \in \{1, 2\}$ and $p \in Ax(f)$, or
\item $(x_{j} \, | \,w^{-1} p)_{x_{0}} < K_{0}$ for each $j \in \{1, 2\}$ and $p \in Ax(f)$.
\end{enumerate}
\end{obs}

We now describe the roughly branching barrier. Let  \[
I_{i}:= \big[ 100K_{0}i - 25K_{0}, 100K_{0} i + 25K_{0} \big], \quad J_{i} := \big[ 100K_{0} i - K_{0}, 100K_{0} i + K_{0} \big],
\]\[\begin{aligned}
B_{i} &:= \left\{ g \in G :  \begin{array}{c} \big(gx_{0} \, \big| \, y)_{x_{0}}  \in I_{i},\\ 
\forall \gamma \in \mathscr{A} \big[  d_{\gamma}(x_{0}, y) \ge 5K_{0} \vee d_{\gamma}(x_{0}, g x_{0}) \le 100K_{0}\big] \end{array} \right\}.
\end{aligned}
\]
We claim that:\begin{claim}\label{claim:pathBlock}
Let $P = (g_{1}, g_{2}, \ldots, g_{N})$ be a $d_{S}$-path such that $(g_{1} x_{0}| y)_{x_{0}} \in I_{0}$ and $(g_{N}x_{0}|y)_{x_{0}} \in I_{2}$. Then there exists $i$ such that $g_{i} \in B_{1}$. 
\end{claim}

\begin{proof}[Proof of Claim \ref{claim:pathBlock}]

For this proof, let \[
\mathscr{A}' := \{ \gamma \in \mathscr{A} : d_{\gamma}(x_{0}, y) < 5K_{0}\}.
\]

Suppose to the contrary that $P$ does not pass through $B_{1}$. Recall that for each $z, z' \in X$, $(z | y)_{x_{0}}$ and $(z' | y)_{x_{0}}$ differ by at most $d_{X}(z, z')$. Hence, along the $d_{S}$-path $(g_{1}, g_{2}, \ldots, g_{N})$, the quantity $(g_{i} x_{0}|y)_{x_{0}}$ changes by at most $K_{0}$ at each step $i$. Since $(g_{i} x_{0}|y)_{x_{0}}$ changes from less than $20K_{0}$ to more than $180K_{0}$, there exists a step $i(1)$ for which $(g_{i(1)} x_{0}|y)_{x_{0}}$ lies in $J_{1}$.

Since we supposed that $P$ does not pass through $B_{1}$, $x_{0}$ and $g_{i(1)}x_{0}$ are $100K_{0}$-separated along some $\gamma \in \mathscr{A}'$. In particular, \[
\mathcal{C}_{0} := \big\{ \gamma \in \mathscr{A}' : d_{\gamma} (x_{0}, g_{i(1)} x_{0}) \ge 80K_{0}\big\}
\]
is non-empty. We pick $\gamma_{0} \in \mathcal{C}_{0}$ that is the closest to $x_{0}$.

At this moment, we observe: \begin{obs}\label{obs:chain}
Let $\gamma_{1}, \gamma_{2}, \ldots, \gamma_{n} \in \mathscr{A}'$ and $z \in X$ be such that: \begin{enumerate}
\item $d_{\gamma_{i-1}}(x_{0}, \gamma_{i}) \ge 50K_{0}$ for $1 \le i \le n$;
\item $d_{\gamma_{n}} (x_{0}, z) \ge 50 K_{0}$.
\end{enumerate}
Then $(z | y)_{x_{0}}$ and $(g_{i(1)} x_{0} | y)_{x_{0}}$ are $22K_{0}$-close.  In particular, $(z | y)_{x_{0}}$ lies in $I_{1}$ and not in $I_{2}$.
\end{obs}

To see this, suppose that $\gamma_{1}, \ldots, \gamma_{n} \in \mathscr{A}'$ and $z \in X$ satisfy the assumption. Then $d_{\gamma_{0}}(x_{0}, z) \ge 46K_{0}$ by Lemma \ref{lem:chainBBF}. Let $p \in \pi_{\gamma_{0}}(x_{0})$ and $q \in \pi_{\gamma_{0}}(z)$. Then $[x_{0}, z]$ is $0.01K_{0}$-close to $p$. Meanwhile, recall that $d_{\gamma_{0}}(x_{0}, y) <5K_{0}$, which implies $d_{\gamma_{0}}(y, z) \ge 40K_{0}$. By Corollary \ref{cor:projGrom}(2), $[y, z]$ is $0.01K_{0}$-close to $\pi_{\gamma_{0}}(y)$, which is $5K_{0}$-close to $p$. In conclusion, $[x_{0}, z]$ and $[y, z]$ are $0.01K_{0}$-close and $5.01K_{0}$-close to $p$, respectively. Hence, $(z|y)_{x_{0}}$ and $(p|y)_{x_{0}}$ differ by at most $10.1K_{0}$. 

Now recall that $d_{\gamma_{0}}(x_{0}, g_{i(1)}x_{0}) \ge 100K_{0}$. For the same reason, $(g_{i(1)}x_{0} | y)_{x_{0}}$ and $(p|y)_{x_{0}}$ differ by at most $10.1K_{0}$. Observation \ref{obs:chain} now follows.

Let us now go back to the proof of the claim. If $d_{\gamma_{0}} (x_{0}, g_{j} x_{0}) \ge80K_{0}$ for all $j \ge i(1)$, then Observation \ref{obs:chain} tells us that the $(g_{N}x_{0} | y)_{x_{0}}$ lies in $I_{1}$ and not in $I_{2}$, a contradiction. Hence, we can pick the earliest $i(2) > i(1)$ such that $d_{\gamma_{0}} (x_{0}, g_{i(2)} x_{0}) \le 80K_{0}$. By the coarse Lipschitzness of $\pi_{\gamma_{0}}(\cdot)$, we have $d_{\gamma_{0}}(x_{0}, g_{i(2)} x_{0}) \ge 78K_{0}$, and Observation \ref{obs:chain} still tells us that $(g_{i(2)}x_{0}|y)_{x_{0}}\big)\in I_{1}$. Since $g_{i(2)} \in P$ is assumed not to be in $B_{1}$, the collection \[
\mathcal{C}_{1} := \{ \gamma \in \mathscr{A}' : d_{\gamma}(x_{0}, g_{i(2)}x_{0}) > 100K_{0}\}
\]
is nonempty. We pick $\gamma_{1}  \in \mathcal{C}_{1}$ that is the closest to $x_{0}$. Clearly $\gamma_{1} \neq \gamma_{0}$.

Note that $[x_{0},  g_{i(2)}x_{0}]$ has large projections onto both $\gamma_{0}$ and $\gamma_{1}$. Let $\eta_{0}$ and $\eta_{1}$ be subsegments of $[x_{0}, g_{i(2)}x_{0}]$ that are $12\delta$-equivalent to the two projections, respectively. Then $\eta_{0}$ is at least $77K_{0}$-long and $\eta_{1}$ is at least $99K_{0}$-long. Moreover, recall that $\diam_{\gamma_{0}}(\gamma_{1}) < K_{0}$ as distinct axes in $\mathscr{A}$ have $K_{0}$-bounded projection. This implies that $\eta_{0}$ and $\eta_{1}$ overlap for length less than $2K_{0}$.

Suppose to the contrary that $d_{X}(x_{0}, \gamma_{1}) <d_{X}(x_{0}, \gamma_{0})$. This implies that  $\eta_{1}$ appears earlier than $\eta_{0}$ along $[x_{0}, g_{i(2)} x_{0}]$. Since they do not overlap much and since $\eta_{0}$ is long enough, we can take $p \in \eta_{0}$ such that \[\begin{aligned}
d_{X}(g_{i(2)}x_{0}, p) &\le d_{X}(g_{i(2)} x_{0}, \eta_{1}) - 75K_{0} \\
&\le d_{X}\big(g_{i(2)} x_{0}, \pi_{\gamma_{1}} ([x_{0}, g_{i(2)}x_{0}])\big) - 74K_{0} = d_{X}(g_{i(2)} x_{0}, \gamma_{1}) - 74K_{0}.
\end{aligned}
\]
By Lemma \ref{lem:projGrom} we have $d_{\gamma_{1}}(p, g_{i(2)} x_{0}) \le 12\delta$. Since $p$ is $12\delta$-close to $\gamma_{0}$ and since $\gamma_{0}$ has bounded projection onto $\gamma_{1}$ (as they are distinct!), we conclude that $d_{\gamma_{1}}(\gamma_{0}, g_{i(2)} x_{0}) \le 3K_{0}$. As a result, we have \[
d_{\gamma_{1}}(x_{0}, \gamma_{0}) \ge 97K_{0}.
\]
Let us observe $\mathcal{C}_{0}$ for the moment. Since $d_{\gamma_{0}} (x, g_{i(1)} x_{0}) \ge 100K_{0}$, there exists a point $p \in [x, g_{i(1)} x_{0}]$ that is $12\delta$-close to some $q \in \gamma_{0}$. Since $d_{\gamma_{1}}(x_{0}, \gamma_{0}) \ge 97K_{0}$ and  $\diam_{\gamma_{1}}(\gamma_{0})\le K_{0}$, we have $d_{\gamma_{1}}(x_{0}, q) \ge 96K_{0}$ and $d_{\gamma_{1}}( x_{0}, p) \ge 95K_{0}$. In particular, $d_{\gamma_{1}}([x_{0}, g_{i(1)} x_{0}]) \ge 95K_{0}$, which implies $d_{\gamma_{1}} (x_{0}, g_{i(1)}x_{0}) \ge 94K_{0}$ by Corollary \ref{cor:projGrom}(3). Thus, $\gamma_{1}$ belongs to $\mathcal{C}_{0}$. Since $d_{X}(x_{0}, \gamma_{1}) <d_{X}(x_{0}, \gamma_{0})$, this contradicts the minimality of $\gamma_{0}$.

We therefore conclude that $d_{X}(x_{0}, \gamma_{0}) \le d_{X}(x_{0}, \gamma_{1})$, and $\eta_{0}$ appears earlier than $\eta_{1}$. Then $d_{\eta_{0}}(\eta_{1}, g_{i_{2}}x_{0}) \le 2K_{0}$ and $d_{\gamma_{0}}(g_{i(2)}x_{0}, \gamma_{1}) \le 3K_{0}$. Hence, $d_{\gamma_{0}}(x_{0}, \gamma_{1}) \ge 87K_{0}$.

We keep this manner. If $d_{\gamma_{1}}(x_{0}, g_{j} x_{0}) \ge 80K_{0}$ for all $j > i(2)$, then $(g_{N} x_{0} |y)_{x_{0}}$ lies in $I_{1}$ and not in $I_{2}$ by Observation \ref{obs:chain}, a contradiction. Hence, there is the first moment $i(3) > i(2)$ at which $d_{\gamma_{1}}(x_{0}, g_{j} x_{0}) \le 80K_{0}$. Then $d_{\gamma_{1}}(x_{0}, g_{j}x_{0}) =_{2K_{0}} 80K_{0}$ and Observation \ref{obs:chain} again tells us that $(g_{i(3)}x_{0}, y)_{x_{0}} \in I_{1}$. Since $g_{i_{3}} \in P \notin B_{1}$, the collection \[
\mathcal{C}_{2} := \{ \gamma \in \mathscr{A}' : d_{\gamma}(x_{0}, g_{i(3)} x_{0}) > 100K_{0}\}
\]
is nonempty. We pick $\gamma_{2} \in \mathcal{C}_{2}$ that is the closest to $x_{0}$. If $\gamma_{2}$ is closer than $\gamma_{1}$ to $x_{0}$, then we can argue the same as before that $\gamma_{2} \in \mathcal{C}_{1}$, violating the minimality of $\gamma_{1}$. It follows that $\gamma_{1}$ is closer to $\gamma_{2}$, and $d_{\gamma_{1}}(x_{0}, \gamma_{2}) \ge 90K_{0}$. 

If this process does not halt, we obtain infinite sequence of step numbers $i(1) < i(2) < \ldots$ for the finite path $P$, a contradiction. Hence, the process must halt and the path $P$ should intersect $B_{1}$.
\end{proof}

Similarly, for $g \in G$ such that $gx_{0} \in \mathcal{H}_{R}(x_{0}, y)$, every  $d_{S}$-path from $id$ to $g$ must  pass through each of $B_{1}, B_{2}, \ldots, B_{R/100K_{0}}$.

It remains to show that $\sqcup_{i \ge 1} B_{i}$ is roughly branching. Since $EC(f)$ is a finite extension of a quasi-isometrically embedded subgroup $\langle f \rangle$, the set \[
EC(f) \cap \{ g \in G : d_{X}(x_{0}, gx_{0}) < 200K_{0}\}
\]
is finite. Hence,  it is contained in $\{g : \|g\|_{S} \le R'\}$ for some $R'$.  We claim that $\sqcup_{i \ge 1} B_{i}$ is $(R'+4+ 200K_{0} \|f\|_{S})$-roughly branching.

Let us take a subset $B'$ of $\sqcup_{i \ge 1} B_{i}$ that is maximally $(R'+2)$-separated (in terms of the word metric $d_{S}$). 
We will construct a map $F : B' \rightarrow F(B') \subseteq G$. Given $a \in B'$, Observation \ref{obs:2DivergingChoice} guarantees $\mathscr{W}(a) \in \{w^{-1}, id, w\}$ such that \[
\big(x_{0} \,\big| a\mathscr{W}(a) p\big)_{ax_{0}} <K_{0}, \quad \big(y \,\big|\, a\mathscr{W}(a) p\big)_{ax_{0}} <K_{0} \quad \big(\forall p \in Ax(f)\big).
\]
Furthermore, there exists $\mathscr{W} \in \{w^{-1}, id, w\}$ such that \[
\big(y \,\big|\, \mathscr{W}^{-1}p\big)_{x_{0}} < K_{0}. \quad\big( \forall p \in Ax(f)\big)
\]
Then we define \[
F(a) := a \mathscr{W}(a) f^{200K_{0}} \mathscr{W}.
\]

We now claim that: \begin{claim}\label{claim:chainFInj}
If $a_{1}, a_{2}, \ldots, a_{k}, b_{1}, \ldots, b_{k} \in B'$ are such that \[
F(a_{1}) F(a_{2}) \cdots F(a_{k}) = F(b_{1}) F(b_{2}) \cdots F(b_{k}),
\]
then $a_{1} = b_{1}$.
\end{claim}

\begin{proof}[Proof of Claim \ref{claim:chainFInj}]
Let $U= F(a_{1}) \cdots F(a_{k})$. Note that \[
\left(\begin{array}{c}  
[x_{0}, a_{1} x_{0}],\, a_{1} [x_{0}, \mathscr{W}(a_{1}) f^{200K_{0}} \mathscr{W} x_{0}] ,\, F(a_{1}) [x_{0}, a_{2}x_{0}], \, F(a_{1}) a_{2} [x_{0}, \mathscr{W}(a_{2}) f^{200K_{0}} \mathscr{W} x_{0}], \, \ldots
\end{array}\right)
\]
is a sequence of consecutive geodesics, each longer than $50K_{0}$. (Recall that $(a_{i}x_{0} | y)_{x_{0}} \in I_{1} \cup I_{2} \cup \ldots$ is at least $75K_{0}$.) Next, between each pair of consecutive geodesics the Gromov product is bounded by $2.1K_{0}$. This is because\begin{itemize}
\item $(x_{0} | F(a_{i})\mathscr{W}^{-1}x_{0})_{a_{i} x_{0}} < K_{0}$  and $d_{X}( F(a_{i})\mathscr{W}^{-1}x_{0}, F(a_{i})x_{0}) < K_{0}$. 
\item $\big( a_{i} \mathscr{W}(a_{i}) x_{0}\, \big|\, F(a_{i})y\big)_{F(a_{i})x_{0}} = \big(\mathscr{W}^{-1} f^{-200K_{0}} x_{0} \,\big|\, y\big)_{x_{0}} < K_{0}$ and $(y|a_{i+1}x_{0} )_{x_{0}} \ge 75K_{0}$, which imply $\big( a_{i} \mathscr{W}(a_{i}) x_{0}\, \big|\, F(a_{i})a_{i+1}x_{0}\big)_{F(a_{i})x_{0}} <K_{0} + 4\delta$.

Moreover, $d_{X}(a_{i} \mathscr{W}(a_{i})x_{0}, a_{i}x_{0} ) < K_{0}$.
\end{itemize}
By the stability lemma, there exist points $p_{1}, q_{1}, \ldots, p_{k}, q_{k}$ on $[x_{0}, Ux_{0}]$, in order from closest to farthest from $x_{0}$, such that $d_{X}(p_{1}, a_{1} x_{0}), d_{X}(q_{1}, F(a_{1})x_{0}), \ldots$ are all smaller than $2.2K_{0}$. Similarly, there exist points $p_{1}', q_{1}', \ldots, p_{k}', q_{k}'$ on $[x_{0}, Ux_{0}]$, in order, such that $d_{X}(p_{1}', b_{1}x_{0}), d_{X}(q_{1}', F(b_{1})x_{0}), \ldots \le 2.2K_{0}$.

Suppose to the contrary that $d_{X}(x_{0}, p_{1}) > d_{X}(x_{0}, p_{1}') + 130K_{0}$. Note that \[
d_{X}\big(p_{1}', b_{1} \mathscr{W}(b_{1}) x_{0}\big) \le d_{X}(p_{1}', b_{1}x_{0}) + d_{X}\big(x_{0}, \mathscr{W}(b_{1})x_{0}\big) \le 3.2K_{0},
\]
and similarly $q_{1}'$ and $b_{1} \mathscr{W}(b_{1}) \cdot f^{200K_{0}} x_{0}$ are $3.2K_{0}$-close. By Lemma \ref{lem:GromHausdorff}, $[p_{1}', q_{1}']$ is $3.3K_{0}$-equivalent to $b_{1} \mathscr{W}(b_{1}) [x_{0}, f^{200K_{0}} x_{0}]$.

Let $q$ be $q_{1}'$ or $p_{1}$, whichever coming earlier along $[x_{0}, Ux_{0}]$. Then, \begin{itemize}
\item $p_{1}'$ and $q$ are both closer to $x_{0}$ than $p_{1}$ is. Hence, $p_{1}', q \in [x_{0}, p_{1}]$.
\item $q \in [p_{1}', q_{1}']$ is $3.3K_{0}$-close to $b_{1} \mathscr{W}(b_{1})Ax(f)$, as well as $p_{1}'$. Meanwhile, we have $d_{X}(p_{1}', q_{1}') =_{6.6K_{0}} d_{X}(x_{0}, f^{200K_{0}}) = 200K_{0}$ and $d_{X}(p_{1}', p_{1}) > 130K_{0}$. 
Hence, $p_{1}'$ and $q$ are $130K_{0}$-distant points on $[x_{0}, p_{1}]$ that are $3.3K_{0}$-close to $b_{1}\mathscr{W}(b_{1})Ax(f)$.
\end{itemize}

Now observe that $[x_{0}, a_{1}x_{0}]$ and $[x_{0}, p_{1}]$ are $2.3K_{0}$-equivalent, as $a_{1} x_{0}$ and $p_{1}$ are $2.2K_{0}$-close. Hence,  $b_{1} \mathscr{W}(b_{1}) Ax (f)$ is $5.6K_{0}$-close to points on $[x_{0}, a_{1} x_{0}]$ that are at least $127K_{0}$-distant. This implies that $\diam_{b_{1} \mathscr{W}(b_{1})Ax(f)} ([x_{0}, a_{1}x_{0}]) > 115K_{0}$. Corollary \ref{cor:projGrom}(5) then tells us that $d_{b_{1} \mathscr{W}(b_{1})Ax(f)} (x_{0}, a_{1}x_{0})> 114K_{0}$. Recall our definition of $B_{i}$'s. We are led to $d_{b_{1} \mathscr{W}(b_{1})Ax(f)} (x_{0}, y) \ge 5K_{0}$. 

Meanwhile, our definition of $\mathscr{W}(b_{1})$ tells us that $(x_{0} | p)_{b_{1}\mathscr{W}(b_{1}) x_{0}} < 2K_{0}$ for every $p \in b_{1} \mathscr{W}(b_{1}) Ax(f)$. Lemma \ref{lem:projection} implies that the projection of $x_{0}$ onto $b_{1} \mathscr{W}(b_{1})Ax(f)$ is $(2K_{0} + 8\delta)$-close to $b_{1}\mathscr{W}(b_{1}) x_{0}$. Similarly, because  $(y| p)_{b_{1}\mathscr{W}(b_{1}) x_{0}} < 2K_{0}$ for every $p \in b_{1} \mathscr{W}(b_{1}) Ax(f)$, the projection $\pi_{b_{1} \mathscr{W}(b_{1}) Ax(f)}(y)$ should be $(2K_{0} + 8\delta)$-close to $b_{1}\mathscr{W}(b_{1}) x_{0}$. In conclusion, $d_{b_{1} \mathscr{W}(b_{1}) Ax(f)}(x_{0}, y) < 4.5K_{0}$. This is a contradiction.

A similar contradiction happens if $d_{X}(x_{0}, p_{1}') > d_{X}(x_{0}, p_{1}) + 130K_{0}$. Hence, $p_{1}$ and $p_{1}'$ are $130K_{0}$-close. Since $q_{1}$ ($q_{1}'$, resp.) appears later than $p_{1}$ ($p_{1}'$. resp.) by at least $195K_{0}$, we conclude that $[p_{1}, q_{1}]$ and $[p_{1}', q_{1}']$ overlap for length at least $65K_{0}$. 

Recall that $[p_{1}, q_{1}]$ ($[p_{1}' ,q_{1}']$, resp.) and $a_{1} \mathscr{W}(a_{1})[ x_{0}, f^{200K_{0} }x_{0}]$ ($b_{1} \mathscr{W}(b_{1})[ x_{0}, f^{200K_{0} }x_{0}]$, resp.) are $3.3K_{0}$-equivalent. By Corollary \ref{cor:projGrom}(1), (4), (5), we conclude \[
 \diam_{a_{1} \mathscr{W}(a_{1})Ax(f)} \big(b_{1} \mathscr{W}(b_{1}) Ax(f)\big) \ge 20K_{0}.
 \]
 This implies that $\mathscr{W}(a_{1})^{-1} a_{1}^{-1} b_{1} \mathscr{W}(b_{1})$ lies in $EC(f)$. Meanwhile, note that $d_{X}(a_{1} \mathscr{W}(a_{1})x_{0}, b_{1}\mathscr{W}(b_{1}) x_{0}) \le 140K_{0}$. Hence, $\mathscr{W}(a_{1})^{-1} a_{1}^{-1} b_{1} \mathscr{W}(b_{1})$ lies in $B_{S}(R')$, and $d_{S}(a_{1}, b_{1}) \le R'+2$. Since  $a_{1}, b_{1}$ are chosen from an $(R'+2)$-separated set $B'$, this forces $a_{1} = b_{1}$.
\end{proof}

Now an inductive argument leads to: \begin{claim}\label{claim:chainFInjInd}
If $a_{1}, a_{2}, \ldots, a_{k}, b_{1}, \ldots, b_{k} \in B'$ are such that \[
F(a_{1}) F(a_{2}) \cdots F(a_{k}) = F(b_{1}) F(b_{2}) \cdots F(b_{k}),
\]
then $a_{i} = b_{i}$ for each $i = 1, \ldots, k$.
\end{claim}

It remains to check that $\sqcup_{i} B_{i}$ is contained in a bounded neighborhood of $F(B')$. Given any $a \in \sqcup_{i} B_{i}$, it is $(R'+2)$-close to some $a' \in B'$, as $B'$ is a maximal $(R'+2)$-separated subset of $B$. Now, $F(a')$ and $a'$ are $(2+ 200K_{0} \|f\|_{S})$-close. In summary, $a$ is $(R+4+ 200K_{0} \|f\|_{S})$-close to $F(B')$ as desired.
\end{proof}

Combining Proposition \ref{prop:magicProper} and Proposition \ref{prop:barrierContracting}, we conclude that relatively hyperbolic groups satisfy the assumption of Theorem \ref{thm:hutchcroftIotaGen} with \[
\mathscr{H}_{D} := \big\{\{ g \in G : gx_{0} \in  \mathcal{H}_{100K_{0} D} (x_{0}, y) \} : y \in X\big\}
\]
and $\mathcal{G}_{D, E} := \emptyset$ for each $D, E$. Therefore, Cayley graphs of relatively hyperbolic groups satisfy the assumption of Theorem \ref{thm:hutchcroft}, and $p_{c} < p_{u}$ and $\nabla_{p_{c}} < +\infty$ hold for such graphs.

\section{Barriers in acylindrically hyperbolic group} \label{section:barrierAcyl}

Let $G$ be an acylindrical hyperbolic group with a finite generating set $S$. Then $G$ acts on a suitable $\delta$-hyperbolic space $(X, d_{X})$ with  a unital, axial WPD loxodromic element $f\in G$. Let $x_{0} \in Ax(f)$. We fix these choices throughout the section.

The following is immediate from the $\delta$-hyperbolicity.

\begin{fact}\label{fact:thin1000D}
Let $i<j<k<l$ be integers, let $x \in \mathcal{N}_{j-i - 2\delta}(f^{i}x_{0})$ and let $y \in \mathcal{N}_{l-k - 2\delta}(f^{l} x_{0})$. Then there exists a subsegment $[x', y']$ of $[x, y]$ such that $x' \in \mathcal{N}_{2\delta}(f^{j}x_{0})$ and $y' \in \mathcal{N}_{2\delta}(f^{k} x_{0})$.
\end{fact}

For $D, E \ge 0$ and $u \in G$, we consider two versions of anti-halfspaces:\[\begin{aligned}
 \mathfrak{A}_{D, E, f}^{\pm}(g) &:= \left\{ u \in G :\begin{array}{c} \textrm{$\exists$ geodesic  $\gamma :[0, \tau] \rightarrow X$, $\exists 0 \le \tau_{1} \le \tau_{2} \le \tau_{3} \le \tau$,  $\exists h \in G$}\\
\textrm{such that}\,\, \|h\|_{S} \le E,   \gamma(0) = gx_{0}, \gamma(\tau_{1}) \in \mathcal{N}_{1.1D}(x_{0}),\\
\gamma(\tau_{2}) \in \mathcal{N}_{20\delta}(hx_{0}), \gamma(\tau_{3}) \in \mathcal{N}_{20\delta}(hf^{\pm 200D}x_{0}), \gamma(\tau) \in \mathcal{N}_{5D}(ux_{0})
\end{array}\right\}, \\
\mathfrak{A}_{D, E, f}(g) &:=  \mathfrak{A}_{D, E, f}^{+}(g) \cup  \mathfrak{A}_{D, E, f}^{-}(g),
\end{aligned}
\]
\[\begin{aligned}
 \mathfrak{B}_{D, E, f}^{\pm}(g) &:= \left\{  u \in G :\begin{array}{c} \textrm{$\exists$ geodesic  $\gamma :[0, \tau] \rightarrow X$, $\exists 0 \le \tau_{1} \le \tau_{2} \le \tau_{3} \le \tau$,  $\exists h \in G$}\\
\textrm{such that}\,\, \|h\|_{S} \le 2E,   \gamma(0) = gx_{0}, \gamma(\tau_{1}) \in \mathcal{N}_{3D}(x_{0}),\\
\gamma(\tau_{2}) \in \mathcal{N}_{20\delta}(hx_{0}), \gamma(\tau_{3}) \in \mathcal{N}_{20\delta}(hf^{\pm 180D}x_{0}), \gamma(\tau) \in \mathcal{N}_{5D}(ux_{0})
\end{array}\right\}, \\
\mathfrak{B}_{D, E, f}(y) &:=  \mathfrak{B}_{D, E, f}^{+}(g) \cup  \mathfrak{B}_{D, E, f}^{-}(g).
\end{aligned}
\]

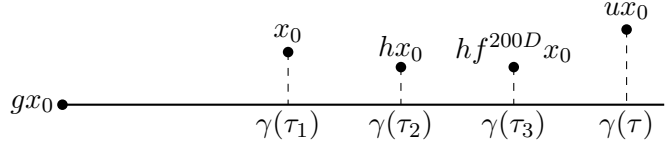
\begin{figure}
\begin{tikzpicture}
\draw[thick] (0, 0) -- (8, 0);
\fill (0, 0) circle (0.07);
\fill (3, 0.7) circle (0.07);
\fill (4.5, 0.5) circle (0.07);
\fill (6, 0.5) circle (0.07);
\fill (7.5, 1) circle (0.07);
\draw (3, -0.27) node {$\gamma(\tau_{1})$};
\draw (4.5, -0.27) node {$\gamma(\tau_{2})$};
\draw (6, -0.27) node {$\gamma(\tau_{3})$};
\draw (7.5, -0.27) node {$\gamma(\tau)$};
\draw (-0.4, 0) node {$gx_{0}$};

\draw[dashed] (3, 0.7) -- (3, 0);
\draw[dashed] (4.5, 0.5) -- (4.5, 0);
\draw[dashed] (6, 0.5) -- (6, 0);
\draw[dashed] (7.5, 1) -- (7.5, 0);

\draw (3, 0.95) node {$x_{0}$};
\draw (4.5, 0.75) node {$hx_{0}$};
\draw (6, 0.75) node {$hf^{200D} x_{0}$};
\draw (7.5, 1.25) node {$ux_{0}$};
\end{tikzpicture}
\caption{Schematics for $\mathfrak{A}_{D, E, f}^{+}(g)$.}
\label{fig:ADEfScheme}
\end{figure}

Some observations are in order. 
\begin{obs}\label{obs:acylObs0}
Let $D \ge 1000(\delta+1)$ and $E > 10D\|f\|_{S}$. Then for each $g \in G$ we have \[
\mathfrak{A}_{D, E, f}^{\pm}(g) \subseteq \mathfrak{B}_{D, E, f}^{\pm} (g). \quad (\forall g \in G)
\]\end{obs}
\begin{proof}
Let $u \in \mathfrak{A}_{D, E, f}^{\pm}(g)$. Let $\gamma : [0, \tau] \rightarrow X$, $0\le \tau_{1} \le \tau_{2} \le \tau_{3} \le \tau$ and $h \in G$ be the ingredients for the membership. In particular, $\gamma(\tau_{2})$ is $20\delta$-close to $hx_{0}$ and $\gamma(\tau_{3})$ is $20\delta$-close to $hf^{\pm 200D} x_{0}$. By Fact \ref{fact:thin1000D}, there exist $\tau_{2} \le \tau_{2}' \le \tau_{3}' \le \tau_{3}$ such that \[
d_{X}\big( hf^{\pm 10 D} x_{0}, \gamma(\tau_{2}') \big) \le 2\delta, \,\,
d_{X}\big( hf^{\pm 190 D} x_{0}, \gamma(\tau_{3}') \big) \le 2\delta.
\]
Furthermore, $\|hf^{\pm 10D}\|_{S} \le \|h\|_{S} + 10D \|f\|_{S} \le 2E$ holds. It is now clear that $u \in \mathfrak{B}_{D, E, f}^{\pm}(g)$.
\end{proof}

In the definition of $\mathfrak{B}_{D, E, f}^{\pm} (g)$ we have $d_{X}(\gamma(\tau_{2}), \gamma(\tau_{3})) \ge 180D - 2\cdot 20\delta$, whereas $d(x_{0}, \gamma(\tau_{1})), d_{X}(ux_{0}, \gamma(\tau)) \le 5D$. This leads to:

\begin{obs}\label{obs:acylObs1}
Let $D \ge 1000(\delta+1)$ and let $u \in \mathfrak{B}_{D, E, f}(g)$. Then $d_{X}(x_{0}, ux_{0})$ and $d_{X}(gx_{0}, ux_{0})$ are at least $100D$.  In particular, $id, g \notin \mathfrak{B}_{D, E, f}(g)$. Moreover, we have $d_{X}(gx_{0}, ux_{0}) \ge d_{X}(gx_{0}, x_{0}) + 100D$.
\end{obs}

Recall that $EC(f)$ is a finite extension of a cyclic subgroup $\langle f \rangle$, which is quasi-isometrically embedded in $X$. Hence,  the set \[
EC(f) \cap \{g : d_{X}(x_{0}, gx_{0}) \le 500D + 2E + 20\delta \}
\]
is finite. The following observation tells us that ``lineage is linear".

\begin{obs}\label{obs:acylObs2}
For each large enough $D>0$ and for each $E>0$ there exists $R>0$ such that the following holds. 

Let $u,v \in G$ and suppose that there exists $w\in u\mathfrak{A}_{D, E, f}(u^{-1}) \cap v \mathfrak{B}_{D, E, f}(v^{-1})$. Then one of the following holds.
\begin{enumerate}
\item $v \in u\mathfrak{A}_{D, E, f}(u^{-1})$ and $d_{X}(wx_{0}, vx_{0}) < d_{X}(wx_{0}, ux_{0})$;
\item $u \in v \mathfrak{B}_{D, E, f}(v^{-1})$ and $d_{X}(wx_{0},  ux_{0}) < d_{X}(wx_{0}, vx_{0})$, or
\item $d_{S}(u, v) \ge R$.
\end{enumerate}
\end{obs}

\begin{proof}
Let $K_{0} = K$ be the constant as in Theorem \ref{thm:elemClos}. Furthermore, let $D_{0} = D_{0}$ be as in Lemma \ref{lem:amplifyWPD}. We assume that $D > 1000(\delta + K_{0} + D_{0})$.

Furthermore, let $R' = R(200D, 300D\|f\|_{S} + E)$ be as in Lemma \ref{lem:amplifyWPD}, and let $R = R' + E$.

By the assumption, there exist a geodesic $\gamma : [0, \tau] \rightarrow X$, $0\le \tau_{1} \le \tau_{2} \le \tau_{3} \le \tau$, a sign $\epsilon \in \{+1, -1\}$ and an element $h \in G$ with $\|h\|_{S} \le E$ such that \[\begin{aligned}
\gamma(0) = x_{0}, \, d_{X}\big(\gamma(\tau_{1}), ux_{0}\big) < 1.1D, \,d_{X}\big(\gamma(\tau_{2}), uhx_{0}\big) < 20\delta, \\
d_{X}\big(\gamma(\tau_{3}), uhf^{200\epsilon D}x_{0}\big) < 20\delta, \, d_{X}\big(\gamma(\tau), wx_{0}\big) < 5D.
\end{aligned}
\]
Similarly, there exist a geodesic $\gamma' : [0, \sigma] \rightarrow X$, $0\le \sigma_{1} \le \sigma_{2} \le \sigma_{3} \le \sigma$,  $\epsilon' \in \{+1, -1\}$ and  $h' \in G$ with $\|h'\|_{S} \le 2E$ such that \[\begin{aligned}
\gamma'(0) = x_{0}, \, d_{X}\big(\gamma'(\sigma_{1}), vx_{0}\big) < 3D, \,d_{X}\big(\gamma'(\sigma_{2}), vh'x_{0}\big) < 20\delta, \\
d_{X}\big(\gamma'(\sigma_{3}), vh'f^{180\epsilon' D}x_{0}\big) < 20\delta, \, d_{X}\big(\gamma'(\sigma), wx_{0}\big) < 5D.
\end{aligned}
\]

We define $L_{min} := (\gamma(\tau) | \gamma'(\sigma))_{x_{0}}$. Then we have \[
L_{min} = \tau - (x_{0} | \gamma'(\sigma))_{\gamma(\tau)} \ge \tau - d_{X}(\gamma(\tau), \gamma'(\sigma)) \ge \tau - 10D.
\]
Similarly, $L_{min} \ge \tau - 10D$. Similarly $L_{min} \ge \sigma - 10D$. By Lemma \ref{lem:fellow}, $\gamma(t)$ and $\gamma'(t)$ are $4\delta$-close for $0\le t\le L_{min}$.

We claim that if $\tau_{2} \le \sigma_{2}$ then either (1) or (3) holds, and if $\tau_{2} \ge \sigma_{2}$ then either (2) or (3) holds. Since the latter case follows from a similar argument, we only explain the former one.

If $\tau_{3} \le \sigma_{1} + 1.5D$ in addition,  then we have \[\begin{aligned}
d_{X}\big(\gamma(\sigma_{1} +1.5D), vx_{0}\big) &\le d_{X}\big(\gamma(\sigma_{1}+ 1.5D), \gamma(\sigma_{1}) \big) + d_{X}(\gamma(\sigma_{1}), \gamma'(\sigma_{1})) + d_{X}\big(\gamma'(\sigma_{1}), vx_{0} \big) \\
&\le 1.5D +4\delta +3D  < 4.6D.
\end{aligned}
\]
Here, we  can feed the parameter $\sigma_{1} + 1.5D$ in $\gamma(\cdot)$ because $\sigma_{1}\le \sigma_{1}+1.5D \le \sigma - 150D \le L_{min} \le \tau$. Since we have the geodesic $\gamma$ with timing $\tau_{1} \le \tau_{2} \le \tau_{3} \le \sigma_{1}+1.5D$, we conclude that $vx_{0} \in u \mathfrak{A}_{D, E, f}(u^{-1})$. Furthermore, since $\tau_{3} \ge \tau_{2} + 199D$ we have \[\begin{aligned}
d_{X}(wx_{0}, vx_{0}) &=_{5D} \tau - (\sigma_{1} + 1.5D) \le \tau - \tau_{3} \\
&\le \tau - \tau_{2} + 199D \le \tau - \tau_{1} + 199D =_{D} d_{X}(wx_{0}, ux_{0}) + 199D.
\end{aligned}
\]

If $\tau_{3} \ge \sigma_{1} + 1.5D$, then we claim that the intersection $I$ of $[ \sigma_{1}, \sigma_{3}]$ and $[\tau_{2}, \tau_{3}]$ is large. Indeed, there are three cases: \begin{itemize}
\item First note that $[\sigma_{1}, \sigma_{3}]$ and $[\tau_{2}, \tau_{3}]$ are both $100D$-long. Hence, if one includes the other one, the intersection $I$ must be $100D$-long. 
\item If $\sigma_{1} \le \tau_{2}$ and $\sigma_{3} \le \tau_{3}$, then $I$ is $100D$-long as \[
\sigma_{3} \ge \sigma_{2} + 100D \ge \tau_{2} + 100D.
\]
\item If $\sigma_{1} \ge \tau_{2}$ and $\sigma_{3} \ge \tau_{3}$, then $I$ is at least $1.5D$-long by the assumption $\sigma_{1} \le \tau_{3} - 1.5D$.
\end{itemize} 
All in all, we have $\diam(I) \ge 1.5D$. In other words, the projections of $\gamma(\sigma_{1})$ and $\gamma(\sigma_{3})$ onto $\gamma([\tau_{2}, \tau_{3}])$ is at least $1.5D$-distant. Note that $uh[x_{0}, f^{200\epsilon D}x_{0}]$ and $\gamma([\tau_{2}, \tau_{3}])$ are $22\delta$-equivalent by Lemma \ref{lem:GromHausdorff}. Hence, Corollary \ref{cor:projGrom}(4) tells us that 
 $d_{uh[x_{0}, f^{200 \epsilon D}x_{0}]}(\gamma(\sigma_{1}), \gamma(\sigma_{3})) \ge 1.4D$. Moreover, note that \[
d_{X}(\gamma(\sigma_{1}), vx_{0}) \le d_{X}(\gamma(\sigma_{1}), \gamma'(\sigma_{1})) + d_{X}(\gamma'(\sigma_{1}), vx_{0}) \le 4\delta + 1.1D.
\]Similarly, $\gamma(\sigma_{3})$ and $vh'f^{180\epsilon'D}x_{0}$ are $24\delta$-close. By Corollary \ref{cor:projGrom}(1) we conclude\[
d_{uh[x_{0}, f^{200 \epsilon D}x_{0}]}(vx_{0}, vh'f^{180\epsilon' D}x_{0}) >0.1D> K_{0}.
\]
Note that $\|h'f^{180\epsilon' D}\|_{S} \le E + 180D \|f\|_{S}$. Our choice of constant $R'$ based on Lemma \ref{lem:amplifyWPD} guarantees $d_{S}(uh, v) \le R'$. Hence, $d_{S}(u, v) \le R' + E = R$. 
\end{proof}

We now need
\begin{obs}\label{obs:acylObs2.1}
Let $D$ be large enough and let $E > 10\|f\|_{S}D$. Let $u, v, w \in G$ such that $w \in \mathfrak{A}_{D, E, f}^{\pm}(u)$ and $(vx_{0} | wx_{0})_{x_{0}} < 2D$. Then $w \in \mathfrak{B}_{D, E, f}^{\pm}(v)$.
\end{obs}

\begin{proof}
Let $\gamma : [0, \tau] \rightarrow X$ be a geodesic starting at $ux_{0}$  and let $h \in G$ be the ones that realize the membership $w \in \mathfrak{A}_{D, E, f}^{\pm}(u)$. In particular, there are timing $\tau_{1} \le \tau_{2} \le \tau_{3} \le \tau$ such that \[\begin{aligned}
d_{X}(\gamma(\tau_{1}), x_{0}) < 1.1D, \,\, 
d_{X}(\gamma(\tau_{2}), hx_{0}) < 20\delta,\\
d_{X}\big(\gamma(\tau_{3}), hf^{\pm 200D} x_{0}\big) < 20\delta,\,\,
d_{X}(\gamma(\tau), wx_{0}) < 5D.
\end{aligned}
\]

Let us draw a geodesic $\eta : [0, L'] \rightarrow X$ that connects $vx_{0}$ to $wx_{0}$. Since we are assuming $(vx_{0} | wx_{0})_{x_{0}} < 2D$,  there exists $\tau^{\ast}_{1}$ such that $\eta(\tau^{\ast}_{1})$ and $x_{0}$ are $(2D+20\delta)$-close. Now $\eta([\tau^{\ast}_{1}, L'])$ and $\gamma([\tau_{1}, \tau_{3}])$ are $5.1D$-equivalent by Lemma \ref{lem:GromHausdorff}. Namely, there exist $\tau^{\ast}_{1} \le t_{2} \le t_{3} \le L'$ such that $\eta(t_{2})$ and $\eta(t_{3})$ are $5.5D$-close to $hx_{0}$ and $h f^{200\epsilon D} x_{0}$, respectively. 

Now Fact \ref{fact:thin1000D} gives timing $t_{2} \le \tau_{2}^{\ast} \le \tau_{3}^{\ast} \le t_{3}$ such that $\eta(\tau_{2}^{\ast})$ and $\eta(\tau_{3}^{\ast})$ are $2\delta$-close to $hf^{10\epsilon D} x_{0}$ and $hf^{190\epsilon D} x_{0}$, respectively. The geodesic $\eta$ together with $\tau_{1}^{\ast} \le \tau_{2}^{\ast} \le \tau_{3}^{\ast} \le L$ show that $v \in u\mathfrak{B}_{D, E, f}(g)$, as $\|hf^{10\epsilon D}\|_{S} \le E + 10\|f\|_{S} D \le 2E$.
\end{proof}

Now for each $u \in G$ we define \[\begin{aligned}
\mathcal{H}_{D, E, f}(u) 
&:=  \left\{g \in G :  \big(gx_{0}\big|ux_{0}\big)_{x_{0}} > D \,\,\textrm{or}\not \exists h \in G \Big[[\|h\|_{S}\le E] \wedge [d_{h Ax(f)} (x_{0}, gx_{0}) \ge 250D]\Big]\right\}.
\end{aligned}
\]
We call it an $f$-halfspace radius parameters $(D, E)$. This is related to anti-halfspaces $\mathfrak{A}_{D, E, f}$  because:
\begin{lem}\label{lem:fHalfNest}
For each $D>1000(\delta+1)$ and $E \ge 0$ there exists $F>E$ such that \[
\big[g \notin \mathcal{H}_{D, E, f} (u) \big]\Rightarrow  \big[ g \in \mathfrak{A}_{D, F, f}(u) \big] \quad(\forall g, u \in G).
\]
\end{lem}

\begin{proof}
Let $K_{0} := \max_{s \in S} d_{X}(x_{0}, sx_{0}) + \|f\|_{S}$. We claim that $F := E + (2K_{0}E + 2D + 1)K_{0}$ works.  To see this, let $g \notin \mathcal{H}_{D, E, f}(u)$. Then there exists $h \in G$ such that $\|h\|_{S} \le E$ and $d_{hAx(f)} (x_{0}, gx_{0}) \ge 250D$. Now let $\gamma : [0, L] \rightarrow X$ be the geodesic connecting $ux_{0}$ to $gx_{0}$. Since $(gx_{0} | ux_{0})_{x_{0}} \le D$, there exists $\tau_{1} \in [0, L]$ such that $\gamma(\tau_{1})$ is $1.1D$-close to $x_{0}$. Then by the coarse Lipschitzness of $\pi_{hAx(f)}(\cdot)$, we have \begin{equation}\label{eqn:248DIneq}
d_{hAx(f)} \big(\gamma(\tau_{1}), \gamma(L)\big) > d_{hAx(f)} (x_{0}, gx_{0}) - (1.1D+12\delta) >  248D.
\end{equation}

Let $i, j \in \Z$ be such that $\pi_{hAx(f)}(\gamma(\tau_{1}))$ intersects $[hf^{i} x_{0}, hf^{i+1} x_{0}]$ and $\pi_{hAx(f)}(\gamma(L))$ intersects $[hf^{j} x_{0}, hf^{j+1} x_{0}]$. Then either $j > i + 247D$ or $j < i-247D$ due to Inequality \ref{eqn:248DIneq}. We will focus on the former case; the latter case can be handled in a similar way. In this case, Corollary \ref{cor:projGrom} tells us that there exist $\tau_{1} \le \tau_{2} \le \tau_{3} \le L$ such that $\gamma(\tau_{2})$ is $12\delta$-close to $hf^{i+1} x_{0}$ and $\gamma(\tau_{3})$ is $12\delta$-close to $hf^{i + 200D+1}x_{0}$.

Recall that $d_{X}(\gamma(\tau_{1}), hAx(f)) \le d_{X}(\gamma(\tau_{1}), x_{0}) + d_{X}(x_{0}, hx_{0})$. Using this, we observe that  \[\begin{aligned}
d_{X}(hx_{0}, h f^{i+1}x_{0})& \le d_{X}( hx_{0}, \pi_{hAx(f)} (\gamma(\tau_{1})) + 1\\
&\le d_{X}\big(hx_{0}, \gamma(\tau_{1}) \big) + d_{X}\big( \gamma(\tau_{1}), hAx(f) \big) \\
& \le 2d_{X}(x_{0}, hx_{0}) + 2D + 1 \le 2K_{0} E + 2D +1.
\end{aligned}
\]
This means $|i+1| < 2K_{0}E + 2D + 1$ and $\|hf^{i+1}\|_{S} \le \|h\|_{S} + (2K_{0}E + 2D + 1) \|f\|_{S} \le F$. 

All in all, our choice of timing $\tau_{1} \le \tau_{2} \le \tau_{3} \le L$, together with $hf^{i+1}\in G$ with $\|hf^{i+1}\|_{S} \le F$, guarantees that $gx_{0} \in \mathfrak{A}_{D, F, f}(u)$ as desired.
\end{proof}

We can now state:

\begin{prop}\label{prop:magicAcyl}
Let $X$ be a $\delta$-hyperbolic space and let $G$ be a non-virtually cyclic group acting on $X$ with an axial, unital WPD loxodromic element $f$. Let $S$ be a finite generating set of $G$. Let $x_{0} \in Ax(f)$.

Then for each $\epsilon>0$ and for each large $D>0$ and $E, E' \ge 0$ there exists a constant $N = N(\epsilon, D, E, E')$ such that for every finite set $A \subseteq G$ there exists a subset $A' \subseteq A$ satisfying:

\begin{enumerate}
\item $\#A' \ge (1-\epsilon) \#A$;
\item For each $a \in A'$ there exist $f$-halfspaces $\mathcal{H}_{1}, \mathcal{H}_{2} \subseteq G$ with radius parameters $(D, E)$ such that  \[
\#\Big( A \setminus a\cdot \big(\mathcal{H}_{1} \cup \mathcal{H}_{2} \setminus \{g \in G : \|g\|_{S} \le E'\} \big) \Big) \le N.
\]
\end{enumerate}
\end{prop}

\begin{proof}
Note that $\{g \in G : \|g\|_{S} \le E'\}$ have elements at most $(2\#S)^{E'}$. Hence, the statement for general $E'$ will follow once we prove it for $E' = 0$. For this reason we set $E' = 0$. Let $D$ be large enough that Observation \ref{obs:acylObs1},  \ref{obs:acylObs2} and \ref{obs:acylObs2.1} apply, and let $E \ge 10\|f\|_{S} D$. 

Let $F =F(D, E)$ be as in Lemma \ref{lem:fHalfNest} and let $R_{0}=R$ be as in Observation \ref{obs:acylObs2} for $(D, F)$. We claim that \[
N = N(\epsilon, D, E) := \frac{2 \cdot (2\#S)^{R}}{\epsilon}
\]
works.

Let us begin the proof by collecting problematic elements, i.e., \[
\mathcal{A}_{1} := A \setminus A' = \left\{a \in A :  \begin{array}{c}\#\big( A  \setminus a(\mathcal{H}_{1} \cup \mathcal{H}_{2}) \big) \ge N\,\, \textrm{for every $f$-halfspaces }\\
\textrm{$\mathcal{H}_{1}, \mathcal{H}_{2}$ with radius parameters $(D, E)$}\end{array}\right\}.
\]
Let $\mathcal{A}_{2}$ be a maximally $R$-separated subset $\mathcal{A}_{2}$ of $\mathcal{A}_{1}$, i.e., we have \begin{enumerate}
\item $d_{S}(a, a') \ge R$ for each pair of distinct elements $a, a' \in \mathcal{A}_{2}$;
\item $\mathcal{A}_{2}$ is a maximal subset of $\mathcal{A}_{1}$ satisfying this property.
\end{enumerate}

Then $\bigcup_{a \in \mathcal{A}_{2}} a \cdot \{g \in G : \|g\|_{S} \le R\}$ covers entire $\mathcal{A}_{1}$. Hence, we have \[
\#\mathcal{A}_{2} \ge \frac{1}{(2\#S)^{R}} \cdot \#\mathcal{A}_{1}.
\]

As before, we first prepare empty collections $\mathcal{B}= \mathcal{U} = \mathcal{G} = \emptyset$. Enumerate $\mathcal{A}_{2}$ by the distance from $x_{0}$, i.e., let $\mathcal{A}_{2} = \{a_{1}, a_{2}, \ldots, a_{\#\mathcal{A}_{2}}\}$ be such that $d_{X}(x_{0}, a_{i}) \le d_{X}(x_{0}, a_{i+1})$ for each $i$. 

We now describe a procedure that takes place throughout $\#\mathcal{A}_{2}$ steps. At step $i$, we first declare $\mathfrak{A}_{i}:= a_{i}\mathfrak{A}_{D, F, f}(a_{i}^{-1})$. \begin{enumerate}
\item If $\mathcal{A}_{2} \cap \mathfrak{A}_{i}$ has no element, then we declare that $a_{i} \in \mathcal{G}$ and $b_{i} := id$.
\item If not, pick $b_{i} \in \mathcal{A}_{2} \cap \mathfrak{A}_{i}$ whose orbit point $b_{i}x_{0}$ is the \emph{closest} to $x_{0}$. We then declare $\mathfrak{A}_{i}' := a_{i}\mathfrak{A}_{D, F, f}( a_{i}^{-1}b_{i})$. 
\begin{enumerate}
\item If $\mathcal{A}_{2} \cap \mathfrak{A}_{i} \cap \mathfrak{A}_{i}'$ has no element, then we declare that $a_{i} \in \mathcal{G}$. 
\item If not, we pick $c_{i} \in \mathcal{A}_{2} \cap \mathfrak{A}_{i} \cap \mathfrak{A}_{i}'$ whose orbit point $c_{i}x_{0}$ is the \emph{closest} to $x_{0}$. We then declare $a_{i} \in \mathcal{B}$ and $b_{i}, c_{i} \in \mathcal{U}$.
\end{enumerate}
\end{enumerate}
(If an element in $\mathcal{U}$ is declared good or bad, it is not undecided anymore; we remove it from $\mathcal{U}$.)

Till step $i$, $\mathcal{G} \cup \mathcal{B}$ comprises of elements from $\{a_{1}, \ldots, a_{i}\}$; they do not contain any of $a_{i+1}, a_{i+2}, \ldots$. ($\ast$) Let us now observe what happens at step $i$.

In case (1), $\mathcal{G}$ gains one more element that might be from  $\mathcal{U}$ or not. $\mathcal{B}$ does not change. Overall, $\#\mathcal{B}$ stays the same and $\#\mathcal{U} + \#\mathcal{G}$ does not decrease. Similar situation happens in Case (2-a).

In case (2-b), $\mathcal{B}$ gains one element $a_{i}$, which might be from $\mathcal{U}$. In exchange, $\mathcal{U}$ gains elements $b_{i}$ and $c_{i}$. Here Observation \ref{obs:acylObs1} guarantees that $d_{X}(x_{0}, b_{i}x_{0}), d_{X}(x_{0}, c_{i}x_{0}) > d_{X}(x_{0}, a_{i}x_{0})$. Since $\mathcal{A}_{2}$ was labelled with respect to the distance from $x_{0}$, we conclude that $b_{i}, c_{i} \in \{a_{i+1}, a_{i+2}, \ldots\}$; in other words, neither $b_{i}$ nor $c_{i}$ come from $\mathcal{G} \cup \mathcal{B}$. We thus confirm that elements are never re-classified once they are put in $\mathcal{G} \cup \mathcal{B}$. 

Furthermore, Observation \ref{obs:acylObs1} guarantees that $d_{X}(b_{i}x_{0}, c_{i}x_{0}) > 100D$. Hence $b_{i}$ and $c_{i}$ are distinct elements. If $b_{i}, c_{i}$ are genuinely new addition to $\mathcal{U}$ and are not re-used from $\mathcal{U}$ at step $i-1$, then we can conclude that $\#\mathcal{U}$ increases at least by 1 in Case (2-b). It remains to show\begin{claim}\label{claim:unredunAcyl}
For $i< j$ such that $a_{i}, a_{j} \in \mathcal{B}$, we have $\{b_{i}, c_{i}\} \cap \{b_{j}, c_{j}\} =\emptyset$.
\end{claim} 

\begin{proof}[Proof of Claim \ref{claim:unredunAcyl}]
Suppose to the contrary that $b_{i} \in \{b_{j}, c_{j}\}$. That means \[
b_{i} \in a_{i} \mathfrak{A}_{D, F, f}(a_{i}^{-1}) \cap a_{j} \mathfrak{A}_{D, F, f}(a_{j}^{-1}).
\]
Here, recall that $d_{X}(x_{0}, a_{i}x_{0}) \le d_{X}(x_{0}, a_{j}x_{0})$ and $d_{S}(a_{i}, a_{j}) > R$. Observation \ref{obs:acylObs2} tells us that $a_{j} \in a_{i} \mathfrak{A}_{D, F, f}(a_{i}^{-1})$. (In Observation \ref{obs:acylObs2}, Case 2 cannot happen because of Observation \ref{obs:acylObs1}, and Case 3 cannot happen for $d_{S}(a_{i}, a_{j})>R$.) Here, note that $d_{X}(x_{0}, a_{j}x_{0}) < d_{X}(x_{0}, b_{j}x_{0})$ because of Observation \ref{obs:acylObs1}. This contradicts the minimality of $b_{i}$.

Next, suppose to the contrary that $c_{i} \in \{b_{j}, c_{j}\}$. That means \[
c_{i} \in a_{i} \mathfrak{A}_{D, F, f}(a_{i}^{-1}) \cap a_{i} \mathfrak{A}_{D, F, f}(a_{i}^{-1}b_{i}) \cap  a_{j} \mathfrak{A}_{D, F, f}(a_{j}^{-1}).
\]
For the same reason as above, we have $a_{j} \in a_{i} \mathfrak{A}_{D, F, f}(a_{i}^{-1})$ and $d_{X}(c_{i} x_{0}, a_{j}x_{0}) < d_{X}(c_{i} x_{0}, a_{i}x_{0})$.

We then have \[\begin{aligned}
d_{X}(x_{0}, a_{j}x_{0}) &\ge d_{X}(x_{0}, a_{i} x_{0}) + 100D, & (\because \textrm{Observation \ref{obs:acylObs1}})\\
d_{X}(b_{i} x_{0}, a_{j} x_{0}) &\ge d_{X}(b_{i}x_{0}, c_{i}x_{0}) - d_{X}(a_{j} x_{0}, c_{i}x_{0}) \\
&\ge d_{X}(b_{i} x_{0}, c_{i}x_{0}) - d_{X}(c_{i}x_{0}, a_{i}x_{0}) \\
&=_{2D} d_{X}(b_{i}x_{0}, a_{i}x_{0}) & \big(\because c_{i} \in a_{i} \mathfrak{A}_{D, F, f}(a_{i}^{-1} b_{i})\big).
\end{aligned}
\]
This implies that $(x_{0} | b_{i}x_{0})_{a_{j}x_{0}} \ge 98D$. Meanwhile, $(x_{0} | c_{i}x_{0})_{a_{j}x_{0}} \le D$ as $c_{i} \in a_{j} \mathfrak{A}_{D, F, f}(a_{j}^{-1})$. By Lemma \ref{lem:GromIneq}, we have $(b_{i}x_{0} | c_{i}x_{0})_{a_{j}x_{0}} < 2D$. Combining this with $c_{i} \in a_{j} \mathfrak{A}_{D, F, f}(a_{j}^{-1})$, we can apply Observation \ref{obs:acylObs2.1} to conclude that $c_{i} \in a_{j}\mathfrak{B}_{D, F, f}(a_{j}^{-1} b_{i})$.

We thus have $c_{i} \in a_{i} \mathfrak{A}_{D, E, f}(a_{i}^{-1}, b_{i}) \cap a_{j} \mathfrak{A}_{D, E, f}(a_{j}^{-1} b_{i})$. By Observation \ref{obs:acylObs2}, either: \begin{enumerate}
\item $a_{j} \in a_{i} \mathfrak{A}_{D, F, f}(a_{i}^{-1} b_{i})$ and $d_{X}(c_{i}x_{0}, a_{j} x_{0}) < d_{X}(c_{i} x_{0}, a_{i} x_{0})$, or
\item $a_{i} \in a_{j} \mathfrak{B}_{D, F, f}(a_{j}^{-1} b_{i})$ and  $d_{X}(c_{i}x_{0}, a_{j} x_{0}) > d_{X}(c_{i} x_{0}, a_{i} x_{0})$.
\end{enumerate}
(Again, $d_{S}(a_{i}, a_{j}) < R$ is ruled out.) Since we already know $d_{X}(c_{i}x_{0}, a_{j} x_{0}) < d_{X}(c_{i} x_{0}, a_{i} x_{0})$, the former case happens.

Hence $a_{j} \in a_{i} \mathfrak{A}_{D, F, f}(a_{i}^{-1}) \cap a_{i} \mathfrak{A}_{D, F, f}(a_{i}^{-1}b_{i})$ with $d_{X}(x_{0}, a_{j}x_{0}) < d_{X}(x_{0}, c_{i}x_{0}) - 100D$, as $c_{i} \in  a_{j} \mathfrak{A}_{D, F, f}(a_{j}^{-1})$. This contradicts the minimality of $c_{i}$.
\end{proof}

Thanks to the claim, we conclude that $\#\mathcal{B} \le \#\mathcal{U} + \#\mathcal{G}$ at each step. But recall also that $a_{i} \in \mathcal{A}_{2}$ is declared good or bad at step $i$ and is not affected thereafter. Hence, after the last step, there is no element of $\mathcal{U}$ left. Hence, we have $\#\mathcal{B} \le \#\mathcal{G}$, and $\mathcal{G}$ takes up at least half of $\mathcal{A}_{2}$.

Now, with the final $\mathcal{G}$ in hand, for each $i \in \{1, \ldots, \#\mathcal{A}_{2}\}$ such that $a_{i} \in \mathcal{G}$, we define \[
K_{i} := A \setminus a_{i} \big(\mathcal{H}_{D, E, f}(a_{i}^{-1}) \cup \mathcal{H}_{D, E, f}(a_{i}^{-1} b_{i})\big).
\]
Since $a_{i} \in \mathcal{G} \subseteq \mathcal{A}_{2}$, we have $\#K_{i} \ge N$. The remaining claim is: 

\begin{claim}\label{claim:magicDisjtWPD}
For every pair of distinct elements $a_{i}, a_{j} \in \mathcal{G}$, $K_{i}$ and $K_{j}$ do not intersect.
\end{claim}

To check this claim, suppose to the contrary that $K_{i}$ and $K_{j}$ have a common element $w$ for some $i < j$ such that $a_{i}, a_{j} \in \mathcal{G}$. By Lemma \ref{lem:fHalfNest}, we have \[
w \in a_{i}\mathfrak{A}_{D, F, f}(a_{i}^{-1}) \cap a_{i}\mathfrak{A}_{D, F, f}(a_{i}^{-1}b_{i}) \cap a_{j} \mathfrak{A}_{D, F, f}(a_{j}^{-1}) \cap a_{j} \mathfrak{A}_{D, F, f}(a_{j}^{-1}b_{j}).
\]

Depending on whether $d_{X}(wx_{0}, a_{j} x_{0}) \le d_{X} (wx_{0}, a_{i} x_{0})$ or not, we have $a_{j} \in a_{i} \mathfrak{A}_{D, F, f}(a_{i}^{-1}) \cap a_{i} \mathfrak{A}_{D, F, f}(a_{i}^{-1}b_{i})$ or  $a_{i} \in a_{j} (\mathfrak{A}_{D, F, f}(a_{j}^{-1}) \cap  a_{j}\mathfrak{A}_{D, F, f}(a_{j}^{-1}b_{j})$ by Observation \ref{obs:acylObs1}. This contradicts the goodness of $a_{i}$ or $a_{j}$. Hence, such a common element $w$ cannot exist and $K_{i}$ and $K_{j}$ are disjoint.

With Claim \ref{claim:magicDisjt} in hand, we have \[\begin{aligned}
\# A &\ge \sum_{i : a_{i} \in \mathcal{G}} \# (K_{i} \cap A) \ge N \cdot \#\mathcal{G} \ge N \cdot \frac{\#\mathcal{A}_{2}}{2} \\
&\ge N \cdot \frac{\#\mathcal{A}_{1}}{2 \cdot (2\#S)^{R} } \ge \frac{1}{\epsilon}(\#A - \#A'). \qedhere
\end{aligned}
\]
\end{proof}

We now have to check the branching property. Recall that $\mathscr{A}$ is the collection of all translates of $Ax(f)$. Let  \[\begin{aligned}
\mathcal{NF}_{D} &:= \big\{ g \in G : \forall \gamma \in \mathscr{A} [ d_{\gamma}(x_{0}, gx_{0}) < D] \big\}, \\
\mathcal{NF}_{D}^{\ge i}& := \mathcal{NF}_{D} \cap \{ g : d_{S}(id, g) \ge i \}.
\end{aligned}
\]
We then observe that: \begin{prop}\label{prop:NFRough}
For each $D$, $\mathcal{NF}_{D}$ is $r$-roughly branching for some $r$.
\end{prop}
\begin{proof}
Since $G$ is non-virtually cyclic, there exists $w \in S$ such that $Ax(f)$ and $w Ax(f)$ have $K_{0}$-bounded projection onto each other. This guarantees a constant $K_{1}$ such that the following holds. For each $g \in \mathcal{NF}_{D}$, there exists $\mathscr{W}(g) \in \{id, w\}$ such that \[
\diam_{g^{-1} Ax (f)} \big(\mathscr{W}(g) \cdot Ax (f)\big ) \le K_{1}.
\]
By increasing $K_{1}$ if necessary, we can also guarantee that $d_{X}(x_{0}, wx_{0}) < K_{1}$. 

We will prove the proposition for $D > 10^{4}(\delta+K_{1}+1)$. For each $g \in G$ we define $F(g) := g \mathscr{W}(g) f^{50D}$.

Recall that the set \[
EC(f) \cap \{ g \in G : d_{X}(x_{0}, gx_{0}) < 100D\}
\]
is a finite set. Namely, it is contained in $\{g \in G : d_{S}(id, g) < R'\}$ for some $R'$. Let $R = R'+2$.

We now consider a  subset $A$ of $\mathcal{NF}_{D}$ that is maximally $R$-separated in the word metric $d_{S}$. Let $a_{1}, \ldots, a_{k}, b_{1}, \ldots, b_{k} \in A $ be such that \[
F(a_{1}) F(a_{2}) \cdots F(a_{k}) = F(b_{1}) F(b_{2}) \cdots F(b_{k}).
\]
We then claim $a_{1} = b_{1}$. To see this, let us define \[\begin{aligned}
p_{i} &:= F(a_{1}) \cdots F(a_{i}) x_{0} & (i=0, \ldots, k),\\
q_{i} &:= F(a_{1}) \cdots F(a_{i-1}) \cdot a_{i} \mathscr{W}(a_{i})x_{0} & (i=1, \ldots, k).
\end{aligned}
\]
We claim that: \begin{enumerate}
\item $d_{X}(q_{i-1}, p_{i}) = 50D$, $q_{i-1}$ is $1.1D$-close to $[p_{i-1}, p_{i}]$ and $d_{X}(p_{i-1}, p_{i}) > 48D$ for $i=1, \ldots, k$;
\item $(q_{i-1} | p_{i+1})_{p_{i}}, (p_{i-1} |p_{i+1})_{p_{i}} \le 2.2D$ for $i=1, \ldots, k-1$.
\end{enumerate}
Recall that $a_{i} \in \mathcal{NF}_{D}$. Hence, we have \[
d_{a_{i} \mathscr{W}(a_{i})Ax(f)} (x_{0}, a_{i}x_{0}) \le D. 
\] 
By the coarse Lipschitzness of the projection (Corollary \ref{cor:projGrom}(1)), we also have \[
d_{a_{i} \mathscr{W}(a_{i})Ax(f)} \big(a_{i}x_{0}, a_{i}\mathscr{W}(a_{i})x_{0}\big) \le 0.001D. 
\]
In summary, we have $d_{a_{i} \mathscr{W}(a_{i})Ax(f)} (x_{0}, a_{i}\mathscr{W}(a_{i})x_{0}) \le 1.001D$. By Corollary \ref{cor:projGrom}(5), we then have \[
d_{a_{i} \mathscr{W}(a_{i}) [x_{0}, f^{50D}x_{0}]}(x_{0}, a_{i}\mathscr{W}(a_{i})x_{0}) \le 1.01D.
\]
By Lemma \ref{lem:projection}, we conclude $(x_{0}| a_{i}\mathscr{W}(a_{i}) f^{50D}x_{0})_{a_{i} \mathscr{W}(a_{i}) x_{0}} \le 1.015D$. Now Lemma \ref{lem:fellow} tells us that $a_{i}\mathscr{W}(a_{i}) x_{0}$ is $1.1D$-close to $[x_{0}, a_{i}\mathscr{W}(a_{i})f^{50D} x_{0}]$. This also implies \[
d_{X}\big(x_{0}, a_{i} \mathscr{W}(a_{i})f^{50D} x_{0}\big) \ge d_{X}\big(a_{i}\mathscr{W}(a_{i})x_{0}, a_{i} \mathscr{W}(a_{i})f^{50D} x_{0}\big) - 1.1D \ge 50D  - 1.1D \ge 48D.
\]
Hence, we conclude Item (1).

We now observe that \[\begin{aligned}
d_{Ax(f)}(x_{0}, a_{i+1} x_{0}) &\le D, \\
d_{Ax(f)} (a_{i+1} x_{0},  a_{i+1} \mathscr{W}(a_{i+1})x_{0}) &\le 0.001D, \\
 d_{Ax(f)} \big(a_{i+1} \mathscr{W}(a_{i+1})x_{0}, a_{i+1} \mathscr{W}(a_{i+1})f^{50D}x_{0}\big) &\le 0.001D.
 \end{aligned}
\]
The first inequality is due to the membership $a_{i+1} \in \mathcal{NF}_{D}$. The second inequality is by Corollary \ref{cor:projGrom}(1). The third inequality is the requirement for $\mathscr{W}(a_{i+1})$. Combined with Corollary \ref{cor:projGrom}(5), these imply \[
d_{Ax(f)} \big(x_{0}, a_{i+1} \mathscr{W}(a_{i+1}) f^{50D} x_{0}\big) \le 1.002D,\,\,d_{[f^{-50 D} x_{0}, x_{0}]}  \big(x_{0}, a_{i+1} \mathscr{W}(a_{i+1}) f^{50D} x_{0}\big) \le 1.003D.
\]
All in all, we have \[
\big(f^{-50D} x_{0} \, \big| \, a_{i+1} \mathscr{W}(a_{i+1})f^{50D}x_{0}\big)_{x_{0}} \le 1.01D,
\]
i.e. $(q_{i-1}| p_{i+1})_{q_{i}} \le 2D$. Meanwhile, by Item (1) we have \[\begin{aligned}
\big( f^{-50D} x_{0} \, \big| \, f^{-50D} \mathscr{W}(a_{i})^{-1} a_{i}^{-1} x_{0} \big)_{x_{0}} &= \big(a_{i} \mathscr{W}(a_{i}) x_{0} \, \big| \, x_{0} \big)_{F(a_{i})x_{0}} \\
&=_{4\delta} d_{X}(x_{0}, f^{-50D}x_{0}) - d_{X} (a_{i} \mathscr{W}\big(a_{i})x_{0}, [x_{0}, F(a_{i})x_{0}] \big)\\
&\ge 50D - 1.1D  \ge 48D.
\end{aligned}
\]
Now Gromov's 4-point inequality implies that \[
 \big( f^{-50D} \mathscr{W}(a_{i})^{-1} a_{i}^{-1} x_{0}\, \big| \, a_{i+1} \mathscr{W}(a_{i+1})f^{50D}x_{0} \big)_{x_{0}} \le 1.1D,
\]
i.e. $(p_{i-1} | p_{i+1})_{p_{i}} \le 2D$. This leads to Item (2).

We can now apply Lemma \ref{lem:stability} to the sequence \[
\big(x_{0},\, q_{1}, \, p_{1}, \, p_{2}, \, \ldots, \, p_{k}\big).
\]
Let $\gamma : [0, L] \rightarrow X$ be the geodesic connecting $x_{0}$ to $F(a_{1}) \cdots F(a_{k})x_{0}$. By Lemma \ref{lem:stability} there exist $\tau \le \tau'$ such that  $d_{X}(\gamma(\tau), q_{1}), d_{X}(\gamma(\tau'), p_{1}) \le 2.25D$. Note that $\tau' \ge \tau + 40D$.

For the exactly same reason, there exist $\sigma \le \sigma'$ such that $\gamma(\sigma)$, $\gamma(\sigma')$ are $2.25D$-close to $b_{1} \mathscr{W}(b_{1}) x_{0}$ and $F(b_{1})x_{0}$, respectively. 

Suppose without loss of generality that $\tau \le \sigma$. There are two cases. \begin{enumerate}
\item If $\sigma \ge \tau + 25D$, then we have \[
d_{\gamma([\tau, \tau'])} (x_{0}, \gamma(\sigma) ) \ge 25D.
\]
Recall that $\gamma([\tau, \tau'])$ and $[q_{1}, p_{1}] = a_{1}\mathscr{W}(a_{1}) [x_{0}, f^{50D}x_{0}]$ are $2.3D$-equivalent by Lemma \ref{lem:GromHausdorff}. Moreover, $\gamma(\sigma)$ and $b_{1} x_{0}$ are $2.3D$-close. These facts and Corollary \ref{cor:projGrom}(5) imply  \[
d_{a_{1}\mathscr{W}(a_{1})[x_{0}, f^{50D} x_{0}]} (x_{0}, b_{1}x_{0}) \ge 12D, \quad 
d_{a_{1}\mathscr{W}(a_{1}) Ax(f)} (x_{0}, b_{1}x_{0}) \ge 10D.
\]
This contradicts the requirement that $b_{1} \in \mathcal{NF}_{D}$.
\item If $\sigma \in [\tau, \tau+25D]$, then $\gamma([\tau, \tau'])$ and $\gamma([\sigma, \sigma'])$ overlap for length at least $15D$. Since $\gamma([\tau, \tau'])$ ($\gamma([\sigma, \sigma'])$, resp.) and $a_{1}\mathscr{W}(a_{1}) [x_{0}, f^{50D}x_{0}]$ ($b_{1}\mathscr{W}(b_{1}) [x_{0}, f^{50D}x_{0}]$, resp.) are $2.3D$-equivalent. By Corollary \ref{cor:projGrom}(1), (4) and (5), we have  \[
\diam_{a_{1}\mathscr{W}(a_{1}) [x_{0}, f^{50D}x_{0}]} \big(b_{1}\mathscr{W}(b_{1}) Ax(f) \big) \ge 8D, \,\,\diam_{a_{1}\mathscr{W}(a_{1}) Ax(f)} \big(b_{1}\mathscr{W}(b_{1}) Ax(f) \big) \ge 7D.
\]
This implies that $\mathscr{W}(a_{1})^{-1} a_{1}^{-1} b_{1} \mathscr{W}(b_{1}) \in EC(f)$. Meanwhile, note that $a_{1}\mathscr{W}(a_{1}) x_{0}$ and $b_{1}\mathscr{W}(b_{1}) x_{0}$ are $30D$-close. Hence, we have \[
\big\|\mathscr{W}(a_{1})^{-1} a_{1}^{-1} b_{1} \mathscr{W}(b_{1})\big\|_{S} \le R',\,\,  \|a_{1}^{-1} b_{1} \|_{S} \le R.
\]
Since $a_{1}, b_{1}$ are chosen from an $R$-separated set $A$, we conclude $a_{1} = b_{1}$ as desired.
\end{enumerate}

By induction, we conclude that $a_{i} = b_{i}$ for each $i$. 

It remains to show that $\mathcal{NF}_{D}$ is contained in a bounded neighborhood of $F(A)$ in the word metric. It is clear that $\mathcal{NF}_{D}$ is contained in the $R$-neighborhood of $A$, and $A$ is contained in the $(50D \|f\|_{S} + 1)$-neighborhood of $F(A)$. This ends the proof.
\end{proof}

We now observe that $\mathcal{NF}_{D}^{>i}$ indeed serves as a ``barrier". First, we record the following corollary of Lemma \ref{lem:chainBBF}.

\begin{lem}\label{lem:nestedBBF}
There exists $K>0$ such that the following holds. Let $\gamma_{0}, \gamma_{1}, \ldots, \gamma_{N} \in \mathscr{A} := \{g Ax (f) : g \in G\}$ and let $z \in X$ be such that: \begin{enumerate}
\item $d_{\gamma_{i-1}}(x_{0}, \gamma_{i}) \ge K$ for $1 \le i \le N$, and
\item $d_{\gamma_{N}}(x_{0}, z) \ge K$.
\end{enumerate}
Then $d_{\gamma_{1}}(x_{0}, z)\ge d_{\gamma_{0}}(x_{0}, \gamma_{1}) - K$.
\end{lem}

\begin{prop}\label{prop:DEThenEPrime}
For each large enough $D, E>0$, there exists $E'>0$ such that the following holds.
Let $g \in G$ be such that  \begin{enumerate}
\item there does not exist $h \in G$ such that \[
d_{S}(id, h) \le E'\,\,\textrm{and}\,\,d_{h Ax(f)}(x_{0}, gx_{0}) \ge 250D.
\]
\item $d_{S}(id, g) \ge E$.
\end{enumerate}
Let $(id = g_{0}, g_{1}, \ldots, g_{N} := g)$ be a $d_{S}$-path between $id$ and $g$. Then there exists $t$ such that $g_{t} \in \mathcal{NF}_{300D}^{\ge E}$.
\end{prop}

\begin{proof}
Let $D_{0}$ be as in Lemma \ref{lem:amplifyWPD} and let $K_{0} = K$ be as in Lemma \ref{lem:nestedBBF}. We assume that $D >1000(\delta + \max_{s \in S} d_{X}(x_{0}, sx_{0}) + D_{0} + K_{0})$. Furthermore, let $E_{amp}= R(600D, E)$ be as in Lemma \ref{lem:amplifyWPD} and let $E' = E_{amp} + 600\|f\|_{S}$.

Suppose to the contrary that a $d_{S}$-path $P=(g_{1}, \ldots, g_{N})$ never intersects $\mathcal{NF}_{300D}^{\ge E}$. We will deduce contradiction. Let \[
i(0) := \max \{ i : d_{S}(id, g_{i}) \le E\}.
\]
If $g_{i(0)} \in \mathcal{NF}_{300D}$, then it is in $\mathcal{NF}_{300D}^{\ge E}$. Due to our standing assumption, this is not the case.

Thus, there exists $h \in G$ such that $d_{h Ax(f)} (x_{0}, g_{i(0)} x_{0}) \ge 300D$. By replacing $h$ with an element of $\{h f^{i} \}_{i \in \Z}$, we may suppose that $\pi_{hAx(f)} (x_{0})$ intersects $[hx_{0}, hfx_{0}]$. Now consider a $d_{S}$-geodesic $Q = (id, g_{1}', \ldots, g_{E}' =: g_{i(0)})$ connecting $id$ and $g_{i(0)}$. By the coarse Lipschitzness of $\pi_{hAx(f)}(\cdot)$, there exist $1 \le j \le E$ such that $d_{X}(hx_{0}, \pi_{hAx(f)}(g_{j}'x_{0})) =_{D} 30D$.

Then  $\pi_{hAx(f)} (x_{0})$ and $\pi_{hAx(f)}(g_{j}'x_{0})$ are  contained in $[hf^{-300D}x_{0}, hf^{300D}x_{0}]$, so we have \[
d_{[hf^{-300D}x_{0}, hf^{300D}x_{0}]}(x_{0}, g_{j}'x_{0})=
d_{hAx(f)}(x_{0}, g_{j}'x_{0}) \ge 25D.
\] 
Note here that $\|g_{j}'\|_{S} \le \|g_{i(0)}\|_{S} =E$. Lemma \ref{lem:amplifyWPD} then tells us that $\|h\|_{S} \le E_{amp}$ and $\|h\|_{S} \le E'$.

In summary,  the collection \[
\mathcal{C}_{0}:= \{ \gamma \in \mathscr{A} : d_{\gamma}(x_{0}, g_{i(1)} x_{0})  \ge 300D\}
\]
is a nonempty collection, and each element of $\mathcal{C}_{0}$ is realized as $hAx(f)$ for some $\|h\|_{S} \le E'$. Let us take $\gamma_{0} = h_{0}Ax(f) \in \mathcal{C}_{0}$ whose axis is the closest to $x_{0}$, i.e., the one  with the smallest $d_{X}(x_{0}, h_{0} Ax(f))$.

If $d_{h_{0} Ax(f)} (x_{0}, g_{i}x_{0}) \ge 300D$ for all $i \ge i(0)$, including $i = N$, then it contradicts the condition on $g$. Hence, $d_{h_{0}Ax(f)} (x_{0}, g_{i} x_{0}) < 290D$ for some $i$. Let us take the smallest such $i$ and name it $i(1)$. By the coarse Lipschitzness of $\pi_{h_{0}Ax(f)}(\cdot)$, we have $d_{h_{0} Ax(f)} (x_{0}, g_{i(1)} x_{0}) \ge 289D$. 

Meanwhile, by our standing assumption, the collection \[
\mathcal{C}_{1} := \{ \gamma \in \mathscr{A} : d_{\gamma}(x_{0}, g_{i(1)} x_{0}) > 300D \}
\]
is nonempty. We pick $\gamma_{1} \in \mathcal{C}_{1}$ that is the closest to $x_{0}$. Clearly $\gamma_{1} \neq \gamma_{0}$. 

Now, as in the proof of Claim \ref{claim:pathBlock}, we can prove that $\gamma_{1}$ appears later than $\gamma_{0}$ along $[x_{0}, g_{i(2)}x_{0}]$; otherwise it will contradict the minimality of $\gamma_{0}$ in $\mathcal{C}_{0}$. We deduce that $d_{\gamma_{0}}(x_{0}, \gamma_{1}) > 289D$.

The proof goes on. If $d_{\gamma_{1}}(x_{0}, g_{i}x_{0}) \ge 300D$ for all $i \ge i(1)$, including $i=N$, then Lemma \ref{lem:nestedBBF} implies that $d_{\gamma_{0}}(x_{0}, g_{N}x_{0}) > 280D$. This contradicts the condition on $g$.

Hence, $d_{\gamma_{1}}(x_{0}, g_{i}x_{0}) < 290D$ for some $i$, and we take the smallest such $i$ as $i(2)$. We have $d_{\gamma_{1}}(x_{0}, g_{t_{2}}x_{0}) \ge 289D$. By the standing assumption,  \[
\mathcal{C}_{2} := \{ \gamma \in \mathscr{A} : d_{\gamma}(x_{0}, g_{u(2)} x_{0}) > 300D \}
\]
is nonempty. We pick $\gamma_{2} \in \mathcal{C}_{2}$ that is the closest to $x_{0}$. We then observe that $\gamma_{2}$ appears later than $\gamma_{1}$; otherwise it violates the minimality of $\gamma_{1}$ in $\mathcal{C}_{1}$. We deduce that $d_{\gamma_{1}}(x_{0}, \gamma_{2}) > 289D$.

If $d_{\gamma_{2}}(x_{0}, g_{i} x_{0}) \ge 300D$ for all $i \ge i(2)$, then  Lemma \ref{lem:nestedBBF} again implies that $d_{\gamma_{0}}(x_{0}, g_{N}x_{0}) > 280D$, a contradiction. Thus, $d_{\gamma_{2}}(x_{0}, g_{i}x_{0}) < 290D$ for some $i$, and we take the smallest such $i$ as $i(3)$. We have $d_{\gamma_{2}}(x_{0}, g_{i(3)}x_{0}) \ge 289D$. 

If this process persists, it means we get an infinite sequence $i(1) < i(2) < \ldots$ in a finite sequence $0 \le 1 \le \ldots \le N$. This is a contradiction.
\end{proof}

\begin{cor}\label{cor:hutchcroftIotaAcyl}
Let $\Gamma$ be the Cayley graph of an acylindrically hyperbolic group $G$. Then Equation \ref{eqn:hutchcroftGamma2} holds.
\end{cor}

\begin{proof}
Without loss of generality, we can fix an action of $G$ on a $\delta$-hyperbolic space $X$ with a unital, axial WPD element $f \in G$. Let $S$ be a finite generating set for $G$ that gives rise to $\Gamma = Cay(G, S)$. Let $x_{0} \in Ax(f)$.

Let $r>0$ be as in Proposition \ref{prop:barrierContracting}. Given $D, E > 0$, we define \[\begin{aligned}
\mathscr{H}_{D} &:= \big\{\{ g\in G : gx_{0} \in \mathcal{H}_{r \cdot D}(x_{0}, ux_{0})\} : u \in G \big\},\\
S_{D} &:= \mathcal{NF}_{300D} = \sqcup_{i=1}^{\infty} \mathcal{NF}_{300D}^{i}.
\end{aligned}
\]
Also, let $E' = E'(D, E)$ be as in Proposition \ref{prop:DEThenEPrime} for $D$ and $E$. We then define \[
\mathcal{G}_{D, E} := \big\{ g \in G : \|g\|_{S} \ge E, \, \not\exists h \in G \big[ \|h\|_{S} \le E' \,\, \textrm{and}\,\,d_{hAx(f)} (x_{0}, gx_{0}) \ge 250D\big]\big\}.
\]
Now given $\epsilon > 0$ as well, we let $N = N(\epsilon, rD, E'(rD, E), E)$ be as in Proposition \ref{prop:magicAcyl}.

Then $S_{D}$ is roughly branching for each $D$. Moreover, for each $\mathcal{H} \in \mathscr{H}_{D}$ there exists an $r$-branching subset $B = B_{1} \sqcup \ldots \sqcup B_{D}$ that is a disjoint union of $D$ $d_{S}$-barriers $B_{1}, \ldots, B_{D}$ between $id$ and $\mathcal{H}$, by Proposition \ref{prop:barrierContracting}.

Finally, for each $D, E> 0$ we observe that $\mathcal{NF}_{300D}^{\ge E}$ is a $d_{S}$-barrier for $id$ and $\mathcal{G}_{D, E}$ by Proposition \ref{prop:DEThenEPrime}. Lastly, Proposition \ref{prop:magicAcyl} guarantees that for each finite $A \subseteq G$ there exist $A' \subseteq A$ with $\#A' \ge (1-\epsilon)\#A$ such that for each $a \in A'$, there exist $\mathcal{H}_{1}, \mathcal{H}_{2} \in \mathscr{H}_{D}$ such that \[\# \big( A \setminus a(\mathcal{H}_{1} \cup \mathcal{H}_{2} \cup \mathcal{G}_{D, E}) \big) \le N.
\]

We can thus apply Theorem \ref{thm:hutchcroftIotaGen} and conclude Equation \ref{eqn:hutchcroftGamma2}.
\end{proof}

\appendix
\section{Proof of Theorem \ref{thm:hutchcroft1plus2}}\label{appendix:hutchcroft1}

We summarize Hutchcroft's proof of Theorem \ref{thm:hutchcroft1plus2}. Let $G, S, R$ and $\mathscr{H}=\{H(g) : g \in G\}$ be as in the assumption. Our goal is to show that there exists $\epsilon > 0$ such that $\left(\frac{d}{dp}\right)_{+}\chi_{p} \ge \epsilon \chi_{p}^{2}$ holds for all $p_{c}/2 \le p \le p_{c}$. Then by integration we will have $\chi_{p} \le \epsilon^{-1} (p_{c} - p)^{-1}$ for each $p_{c}/2 < p < p_{c}$, which leads to Equation \ref{eqn:hutchcroftGamma1}.

We first define a $\{0, 1\}$-valued function $I = I(g, A) \subseteq \Omega$ for inputs $a \in G$ and $A \subseteq G$: \[
I(a, A) := 1_{\{\exists g, h \in G [ \|g\|_{S}, \|h\|_{S} \le R \,\,\textrm{and}\,\, A \subseteq a H(g) \,\, \textrm{and}\,\, H(g) \cap h H(g) = \emptyset]\}}.
\]
By our assumption, we have \[
\sum_{a \in A} I(a, A) \ge \frac{1}{2} \#A
\]
for each $A \subseteq G$. We now define $F : \Omega \times G \rightarrow \mathbb{R}$: \[
F(\omega, a) := I\big(a, C_{\omega}(id)\big) 1_{a \in C_{\omega}(id)}.
\]
We will now fix $p_{c} /2 < p < p_{c}$. We have \[
\sum_{a \in G} \E_{p} F(\omega, a) = \E_{p} \sum_{a \in C_{\omega}(id)} I\big(a, C_{\omega}(id)\big) \ge \E_{p} \left( \frac{1}{2} \# C_{\omega}(id) \right) = \frac{1}{2} \chi_{p}.
\]
Meanwhile, we have \[
\begin{aligned}
\sum_{a \in G} \E_{p} F(\omega, a) &= 
\sum_{a \in G} \E_{p} F(\omega, a^{-1}) = \E_{p} \sum_{a : a^{-1} \in C_{\omega}(id)} I\big(a^{-1}, C_{\omega}(id)\big) \\
&= \E_{p} \sum_{a : id \in C_{\omega}(a)} I\big(id, C_{\omega}(a)\big) \\
&= \E_{p} \sum_{a \in C_{\omega}(id)} I\big(id, C_{\omega}(id)\big) = \E_{p}\big[ \#C_{id} \cdot I\big(id, C_{\omega}(id)\big) \big].
\end{aligned}
\]

Recall also that there are at most $N$ choices $g, h \in G$ such that $\|g\|_{S}, \|h\|_{S}< R$, where $N = (2\#S)^{2R}$. Among those finitely many candidates, there exist a concrete, non-random $g, h$ such that $H(g)$ and $hH(g)$ are disjoint and such that \[
\E_{p} \big[ \#C(id) \cdot 1_{\{C(id) \subseteq H(g)\}}\big] \ge \frac{1}{2N} \chi_{p}.
\]
By applying the action of $h$, we also have  \[
\E_{p} \big[ \#C(h) \cdot 1_{\{C(h) \subseteq hH(g)\}}\big] \ge \frac{1}{2N} \chi_{p}.
\]
Now, note that the event $\{C(id) \subseteq H(g)\}$ is determined solely by edges in $H(g)$, and $\{C(h) \subseteq hH(g)\}$ is determined solely by edges in $hH(g)$. Since the two sets are disjoint, the two events are independent. We conclude that \[
\E_{p} \big[ \big(\#C(id) \big) \cdot \big( \#C(h) \big) \cdot 1_{\{C(id) \subseteq H(g)\,\,\textrm{and}\,\, C(h) \subseteq h H(g)\}}\big] \ge \frac{1}{4N^{2}} \chi_{p}^{2}.
\]

We now pick a $d_{S}$-geodesic $\gamma = (g_{1}, g_{2}, \ldots, g_{\|h\|_{S}})$ connecting $id$ to $h$. Let $e = \overrightarrow{g_{n} g_{n+1}}$ be the first (oriented) edge of $\gamma$ that connects $H(g)$ to $\Gamma \setminus H(g)$. As described in Hutchcroft's proof, a standard conversion of finitely many states guarantees a constant $c_{p}$, which is a linear combination of finitely many products and ratios of $p$ and $1-p$, hence bounded on compact subsets of $(0, 1)$, such that \[
\E_{p}\big[ \big(\#C(g_{n}) \big) \cdot \big( \# C(g_{n+1})\big) \cdot 1_{\{ g_{n} \not\leftrightarrow g_{n+1} \}} \big] \ge c_{p} 
\E_{p} \big[ \big(\#C(id) \big) \cdot \big( \#C(h) \big) \cdot 1_{\{C(id) \subseteq H(g), C(h) \subseteq h H(g)\}}\big].
\]
By applying the action by $g_{n}^{-1}$, we conclude that \[
\E_{p} \left[  \big(\#C(id) \big) \cdot \big( \# C(s)\big)  \cdot 1_{id \not\leftrightarrow s} \right] \ge c \chi_{p}^{2}
\]
for some $c$ uniform on $0.5p_{c} < p < p_{c}$ and for some $s = s(p)$ in the generating set $S$.

We now recall Russo's formula. For each $g \in G$, a closed edge $e = \overrightarrow{e^{-} e^{+}} \in \mathcal{E}^{\rightarrow}$ is pivotal for the event $\{id \leftrightarrow g\}$ if and only if $id \leftrightarrow e^{-}$, $e^{+} \leftrightarrow g$ and $e^{-} \not\leftrightarrow e^{+}$. Hence, we have \[
\left(\frac{d}{dp}\right)_{+} \tau_{p}(g) \ge \frac{1}{1-p} \sum_{e \in \mathcal{E}^{\rightarrow}} \Prob_{p} (\{ id \leftrightarrow e^{-} \} \cap \{ e^{-} \not\leftrightarrow e^{+}\} \cap \{ e^{+} \leftrightarrow g\}).
\]
Since $\tau_{p}(g)$ are monotonic for each $g \in G$, summing over finitely many $g$'s and taking limits imply  \[\begin{aligned}
\left(\frac{d}{dp} \right)_{+} \chi_{p} &\ge \frac{1}{1-p} \sum_{g \in G} \sum_{e \in \mathcal{E}^{\rightarrow}} \Prob_{p} (\{ id \leftrightarrow e^{-} \} \cap \{ e^{-} \not\leftrightarrow e^{+}\} \cap \{ e^{+} \leftrightarrow g\}) \\
&= \frac{1}{1-p}  \sum_{s \in S}\sum_{h \in G} \sum_{g\in G} \Prob_{p} (\{ id \leftrightarrow h\}  \cap \{ h \not\leftrightarrow hs\} \cap \{ hs \leftrightarrow g\} )\\
&= \frac{1}{1-p}   \sum_{s \in S}\sum_{h \in G} \sum_{g\in G} \Prob_{p} (\{ id \leftrightarrow h\}  \cap \{ h \not\leftrightarrow hs\} \cap \{ hs \leftrightarrow hsg\} )\\
&= \frac{1}{1-p} \sum_{g, h \in G, s \in S} \Prob_{p} (\{ h^{-1} \leftrightarrow id\}  \cap \{ id \not\leftrightarrow s\} \cap \{ s \leftrightarrow g \} ).
\end{aligned}
\]
The last summation is bounded from below by $\frac{1}{1-p}\E_{p} \left[  \big(\#C(id) \big) \cdot \big( \# C(s)\big)  \cdot 1_{id \not\leftrightarrow s} \right]$ for our choice $s=s(p)$. Hence, we conclude that $(d/dp)_{+}\chi_{p}$ is uniformly coarsely bounded from below by $\chi_{p}^{2}$ for $p \in (p_{c}/2, p_{c})$. This ends the proof.

\section{Proof of Corollary \ref{cor:projGrom}}\label{appendix:proj}

We sketch the proof of Corollary \ref{cor:projGrom}.
\begin{enumerate}
\item It is a direct consequence of Lemma \ref{lem:projGrom}.
\item We have \[
d_{X}(x, p) + d_{X}(p, y) - 12\delta \le d_{X}(x, p) + d_{X}(p, q) + d_{X}(q, y) - 12 \delta \le  d_{X}(x, y).
\]
Hence, $(x | y)_{p} \le 6\delta$. Pick $x' \in [x, y]$ such that $d_{X}(x, x') = (y | p)_{x}$. Then by Lemma \ref{lem:fellow}, $x'$ is $10\delta$-close to $p$. Similarly, the point $y' \in [x, y]$ such that $d_{X}(x, y') = (y | q)_{x}$ is $10\delta$-close to $p$.
Note that \[
d_{X}(x, p) \le d_{X} (x, y) - d_{X}(q, y)  - d_{X}(p, q) + 12\delta < d_{X}(x, q) - 12\delta + 12\delta = d_{X}(x, q).
\]
For a similar reason, we have $d_{X}(y, q) < d_{X}(y, p)$. This implies that 
\[
d_{X}(x, x') = \frac{1}{2} [d_{X}(x, y) + d_{X}(x, p) - d_{X}(y, p)] \le  \frac{1}{2} [d_{X}(x, y) + d_{X}(x, q) - d_{X}(y, q)] = d_{X}(x, y').
\]
Hence, $x'$ comes earlier than $y'$ along $[x, y]$. By Lemma \ref{lem:GromHausdorff}, $[x', y']$ and $[p, q]$ are $12\delta$-equivalent.      
\item 
Suppose to the contrary that there exist $p \in \pi_{\gamma}(x)$, $q \in \pi_{\gamma}(y)$ and $r \in \pi_{\gamma}(z)$ such that $r \notin \mathcal{N}_{12\delta}([p, q])$. This means that $d_{X}(p, r) + d_{X}(r, q) > 24\delta + d_{X}(p, q)$.

Now Lemma \ref{lem:projGrom} tells us that \[
\begin{aligned}
d_{X}(x, z) \ge d_{X}(x, p) + d_{X}(r, z) + d_{X}(p, r) - 12\delta,\\
d_{X}(z, y) \ge d_{X}(y, q) + d_{X}(r, z) + d_{X}(q, r) - 12\delta.
\end{aligned}
\]
This implies that \[
d_{X}(x, y) \ge d_{X}(x, p) + d_{X}(y, q) + \big( d_{X}(p, r) + d_{X}(r, q) \big) - 24\delta > d_{X}(x, p) + d_{X}(y, q) + d_{X}(p, q).
\]
This is a contradiction.
   
\item Let $\gamma = [y, z]$ and $\gamma' = [y', z']$. Let $a \in [y ,z]$ be such that $d_{X}(y, a) = (x|z)_{y}$, let $b \in [y, x]$ be such that $d_{X}(y, b) = (x|z)_{y}$, let $c \in [y, x]$ be such that $d_{X}(y, c) = (x|z')_{y}$, let $d \in [x, z']$ be such that $d_{X} (z', d) = (x | y)_{z'}$, let $e \in [x, z']$ be such that $d_{X} (z', e) = (x | y')_{z'}$, and let $f \in [y', z']$ be such that $d_{X}(z', f) = (x | y')_{z'}$. 

Then by Lemma \ref{lem:projection}, $a$ and $e$ are $8\delta$-equivalent to $\pi_{\gamma}(x)$ and $\pi_{\gamma'}(x)$, respectively. Moreover, Lemma \ref{lem:fellow} tells us that $d_{X}(a, b), d_{X} (c, d), d_{X}(e, f) \le 4\delta$, and the triangle inequality tells us that $d_{X}(b, c) \le d_{X}(z, z') \le D$, $d_{X}(d, e) \le d_{X}(y, y') \le D$. In conclusion, $\pi_{\gamma}(x)$ and $\pi_{\gamma'}(x)$ are $(2D+28\delta)$-equivalent as desired.
\item Suppose that there exist $p \in \pi_{\gamma}(x)$ and $q \in \pi_{\gamma}(y)$ that realizes $d_{X}(p, q)  = d_{\gamma'}(x, y) > 12\delta$. By Corollary \ref{cor:projGrom}(2), there exists $x', y' \in [x, y]$ that are $10\delta$-close to $p$ and $q$, respectively. In particular, $x'$ and $y'$ are $16\delta$-close to $\gamma' \subseteq \gamma$. This implies \[\begin{aligned}
\pi_{\gamma}(x') \subseteq \mathcal{N}_{10\delta}(x') \subseteq \mathcal{N}_{20\delta}(p), \quad \pi_{\gamma}(y') \subseteq \mathcal{N}_{10\delta}(y') \subseteq \mathcal{N}_{20\delta}(q).
\end{aligned}
\]
This implies that $d_{\gamma}(x', y') \ge d_{X}( p, q) - 40\delta$. Meanwhile, Corollary \ref{cor:projGrom}(3) tells us that $d_{\gamma}(x, y) \ge d_{\gamma}(x', y') - 24\delta$. Combining these two, we conclude that  \[
\diam_{\gamma}(x, y) \ge d_{\gamma}(x', y') - 24\delta \ge d_{X}(p, q) - 64\delta. \qedhere
\]
\end{enumerate}

%
%
%

\medskip
\bibliographystyle{alpha}
\bibliography{perc}

\begin{thebibliography}{BLPS99b}

\bibitem[AKN87]{aizenman1987uniqueness}
M.~Aizenman, H.~Kesten, and C.~M. Newman.
\newblock Uniqueness of the infinite cluster and continuity of connectivity
  functions for short and long range percolation.
\newblock {\em Comm. Math. Phys.}, 111(4):505--531, 1987.

\bibitem[AN84]{aizenman1984tree}
Michael Aizenman and Charles~M. Newman.
\newblock Tree graph inequalities and critical behavior in percolation models.
\newblock {\em J. Statist. Phys.}, 36(1-2):107--143, 1984.

\bibitem[BBF15]{bestvina2015quasi}
Mladen Bestvina, Ken Bromberg, and Koji Fujiwara.
\newblock Constructing group actions on quasi-trees and applications to mapping
  class groups.
\newblock {\em Publ. Math. Inst. Hautes \'{E}tudes Sci.}, 122:1--64, 2015.

\bibitem[BBFS19]{MR4057354}
Mladen Bestvina, Ken Bromberg, Koji Fujiwara, and Alessandro Sisto.
\newblock Acylindrical actions on projection complexes.
\newblock {\em Enseign. Math.}, 65(1-2):1--32, 2019.

\bibitem[Bes23]{bestvina0groups}
Mladen Bestvina.
\newblock Groups acting on hyperbolic spaces---a survey.
\newblock In {\em I{CM}---{I}nternational {C}ongress of {M}athematicians.
  {V}ol. 2. {P}lenary lectures}, pages 678--711. EMS Press, Berlin, [2023]
  \copyright 2023.

\bibitem[BF02]{bestvina2002bounded}
Mladen Bestvina and Koji Fujiwara.
\newblock Bounded cohomology of subgroups of mapping class groups.
\newblock {\em Geom. Topol.}, 6:69--89, 2002.

\bibitem[BF09]{bestvina2009higher}
Mladen Bestvina and Koji Fujiwara.
\newblock A characterization of higher rank symmetric spaces via bounded
  cohomology.
\newblock {\em Geom. Funct. Anal.}, 19(1):11--40, 2009.

\bibitem[BF14]{bestvina2014hyperbolicity}
Mladen Bestvina and Mark Feighn.
\newblock Hyperbolicity of the complex of free factors.
\newblock {\em Adv. Math.}, 256:104--155, 2014.

\bibitem[BH99]{bridson1999metric}
Martin~R. Bridson and Andr\'e Haefliger.
\newblock {\em Metric spaces of non-positive curvature}, volume 319 of {\em
  Grundlehren der mathematischen Wissenschaften [Fundamental Principles of
  Mathematical Sciences]}.
\newblock Springer-Verlag, Berlin, 1999.

\bibitem[BK89]{burton1989density}
R.~M. Burton and M.~Keane.
\newblock Density and uniqueness in percolation.
\newblock {\em Comm. Math. Phys.}, 121(3):501--505, 1989.

\bibitem[BLPS99a]{benjamini1999group-invariant}
I.~Benjamini, R.~Lyons, Y.~Peres, and O.~Schramm.
\newblock Group-invariant percolation on graphs.
\newblock {\em Geom. Funct. Anal.}, 9(1):29--66, 1999.

\bibitem[BLPS99b]{benjamini1999critical}
Itai Benjamini, Russell Lyons, Yuval Peres, and Oded Schramm.
\newblock Critical percolation on any nonamenable group has no infinite
  clusters.
\newblock {\em Ann. Probab.}, 27(3):1347--1356, 1999.

\bibitem[Bon96]{bonk1996quasi-geodesic}
Mario Bonk.
\newblock Quasi-geodesic segments and {G}romov hyperbolic spaces.
\newblock {\em Geom. Dedicata}, 62(3):281--298, 1996.

\bibitem[BS96]{benjamini1996percolation}
Itai Benjamini and Oded Schramm.
\newblock Percolation beyond {$\bold Z^d$}, many questions and a few answers.
\newblock {\em Electron. Comm. Probab.}, 1:no. 8, 71--82, 1996.

\bibitem[BS00]{bonk2000embeddings}
M.~Bonk and O.~Schramm.
\newblock Embeddings of {G}romov hyperbolic spaces.
\newblock {\em Geom. Funct. Anal.}, 10(2):266--306, 2000.

\bibitem[BS01]{benjamini2001percolation}
Itai Benjamini and Oded Schramm.
\newblock Percolation in the hyperbolic plane.
\newblock {\em J. Amer. Math. Soc.}, 14(2):487--507, 2001.

\bibitem[CDP90]{coornaert1990geometrie}
M.~Coornaert, T.~Delzant, and A.~Papadopoulos.
\newblock {\em G\'{e}om\'{e}trie et th\'{e}orie des groupes}, volume 1441 of
  {\em Lecture Notes in Mathematics}.
\newblock Springer-Verlag, Berlin, 1990.
\newblock Les groupes hyperboliques de Gromov. [Gromov hyperbolic groups], With
  an English summary.

\bibitem[Cho25]{choi2025acylindrically}
Inhyeok Choi.
\newblock Acylindrically hyperbolic groups and counting problems.
\newblock {\em arXiv preprint arXiv:2504.20985}, 2025.

\bibitem[CS11]{caprace2011rank}
Pierre-Emmanuel Caprace and Michah Sageev.
\newblock Rank rigidity for {CAT}(0) cube complexes.
\newblock {\em Geom. Funct. Anal.}, 21(4):851--891, 2011.

\bibitem[Cza24]{czajkowski2024non-uniqueness}
Jan Czajkowski.
\newblock Non-uniqueness phase of percolation on reflection groups in {$\Bbb
  H^3$}.
\newblock {\em J. Theoret. Probab.}, 37(3):2534--2575, 2024.

\bibitem[DGO17]{dahmani2017hyperbolically}
F.~Dahmani, V.~Guirardel, and D.~Osin.
\newblock Hyperbolically embedded subgroups and rotating families in groups
  acting on hyperbolic spaces.
\newblock {\em Mem. Amer. Math. Soc.}, 245(1156):v+152, 2017.

\bibitem[FKG71]{fortuin1971correlation}
C.~M. Fortuin, P.~W. Kasteleyn, and J.~Ginibre.
\newblock Correlation inequalities on some partially ordered sets.
\newblock {\em Comm. Math. Phys.}, 22:89--103, 1971.

\bibitem[Gab05]{MR2221157}
D.~Gaboriau.
\newblock Invariant percolation and harmonic {D}irichlet functions.
\newblock {\em Geom. Funct. Anal.}, 15(5):1004--1051, 2005.

\bibitem[Gen19]{genevois2019negative}
Anthony Genevois.
\newblock Negative curvature in automorphism groups of one-ended hyperbolic
  groups.
\newblock {\em J. Comb. Algebra}, 3(3):305--329, 2019.

\bibitem[GH21a]{genevois2021acylindrical}
Anthony Genevois and Camille Horbez.
\newblock Acylindrical hyperbolicity of automorphism groups of infinitely ended
  groups.
\newblock {\em J. Topol.}, 14(3):963--991, 2021.

\bibitem[GH21b]{MR4503954}
Anthony Genevois and Camille Horbez.
\newblock Acylindrical hyperbolicity of automorphism groups of infinitely ended
  groups.
\newblock {\em J. Topol.}, 14(3):963--991, 2021.

\bibitem[Gho23]{MR4541452}
Pritam Ghosh.
\newblock Relative hyperbolicity of free-by-cyclic extensions.
\newblock {\em Compos. Math.}, 159(1):153--183, 2023.

\bibitem[GKN92]{gandolfi1992uniqueness}
A.~Gandolfi, M.~S. Keane, and C.~M. Newman.
\newblock Uniqueness of the infinite component in a random graph with
  applications to percolation and spin glasses.
\newblock {\em Probab. Theory Related Fields}, 92(4):511--527, 1992.

\bibitem[GN90]{grimmett1990percolation}
G.~R. Grimmett and C.~M. Newman.
\newblock Percolation in {$\infty+1$} dimensions.
\newblock In {\em Disorder in physical systems}, Oxford Sci. Publ., pages
  167--190. Oxford Univ. Press, New York, 1990.

\bibitem[Gri89]{grimmett1989percolation}
Geoffrey Grimmett.
\newblock {\em Percolation}.
\newblock Springer-Verlag, New York, 1989.

\bibitem[Gro87]{gromov1987hyperbolic}
M.~Gromov.
\newblock Hyperbolic groups.
\newblock In {\em Essays in group theory}, volume~8 of {\em Math. Sci. Res.
  Inst. Publ.}, pages 75--263. Springer, New York, 1987.

\bibitem[GS21]{goldsborough2021markov}
Antoine Goldsborough and Alessandro Sisto.
\newblock Markov chains on hyperbolic-like groups and quasi-isometries.
\newblock {\em arXiv preprint arXiv:2111.09837}, 2021.

\bibitem[Ham08]{hamenstadt2008bounded}
Ursula Hamenst\"{a}dt.
\newblock Bounded cohomology and isometry groups of hyperbolic spaces.
\newblock {\em J. Eur. Math. Soc. (JEMS)}, 10(2):315--349, 2008.

\bibitem[Har60]{harris1960a-lower}
T.~E. Harris.
\newblock A lower bound for the critical probability in a certain percolation
  process.
\newblock {\em Proc. Cambridge Philos. Soc.}, 56:13--20, 1960.

\bibitem[HJ06]{haggstrom2006uniqueness}
Olle H\"{a}ggstr\"{o}m and Johan Jonasson.
\newblock Uniqueness and non-uniqueness in percolation theory.
\newblock {\em Probab. Surv.}, 3:289--344, 2006.

\bibitem[HMS24]{hagen2024extra-large}
Mark Hagen, Alexandre Martin, and Alessandro Sisto.
\newblock Extra-large type {A}rtin groups are hierarchically hyperbolic.
\newblock {\em Math. Ann.}, 388(1):867--938, 2024.

\bibitem[HP99]{haggstrom1999monotonicity}
Olle H\"{a}ggstr\"{o}m and Yuval Peres.
\newblock Monotonicity of uniqueness for percolation on {C}ayley graphs: all
  infinite clusters are born simultaneously.
\newblock {\em Probab. Theory Related Fields}, 113(2):273--285, 1999.

\bibitem[HP24]{hutchcroft2024percolation}
Tom Hutchcroft and Minghao Pan.
\newblock Percolation at the uniqueness threshold via subgroup relativization.
\newblock {\em arXiv preprint arXiv:2409.12283}, 2024.

\bibitem[HPS99]{haggstrom1999percolation}
Olle H\"{a}ggstr\"{o}m, Yuval Peres, and Roberto~H. Schonmann.
\newblock Percolation on transitive graphs as a coalescent process: relentless
  merging followed by simultaneous uniqueness.
\newblock In {\em Perplexing problems in probability}, volume~44 of {\em Progr.
  Probab.}, pages 69--90. Birkh\"{a}user Boston, Boston, MA, 1999.

\bibitem[Hut16]{hutchcroft2016critical}
Tom Hutchcroft.
\newblock Critical percolation on any quasi-transitive graph of exponential
  growth has no infinite clusters.
\newblock {\em C. R. Math. Acad. Sci. Paris}, 354(9):944--947, 2016.

\bibitem[Hut19]{hutchcroft2019percolation}
Tom Hutchcroft.
\newblock Percolation on hyperbolic graphs.
\newblock {\em Geom. Funct. Anal.}, 29(3):766--810, 2019.

\bibitem[Hut20a]{hutchcroft2020the-l2-boundedness}
Tom Hutchcroft.
\newblock The {$L^2$} boundedness condition in nonamenable percolation.
\newblock {\em Electron. J. Probab.}, 25:Paper No. 127, 27, 2020.

\bibitem[Hut20b]{hutchcroft2020nonuniqueness}
Tom Hutchcroft.
\newblock Nonuniqueness and mean-field criticality for percolation on
  nonunimodular transitive graphs.
\newblock {\em J. Amer. Math. Soc.}, 33(4):1101--1165, 2020.

\bibitem[Hut22a]{hutchcroft2022on-the-derivation}
Tom Hutchcroft.
\newblock On the derivation of mean-field percolation critical exponents from
  the triangle condition.
\newblock {\em J. Stat. Phys.}, 189(1):Paper No. 6, 33, 2022.

\bibitem[Hut22b]{hutchcroft2022slightly}
Tom Hutchcroft.
\newblock Slightly supercritical percolation on non-amenable graphs {I}: {T}he
  distribution of finite clusters.
\newblock {\em Proc. Lond. Math. Soc. (3)}, 125(4):968--1013, 2022.

\bibitem[Hut24]{hutchcroft2024slightly}
Tom Hutchcroft.
\newblock Slightly supercritical percolation on nonamenable graphs {II}: growth
  and isoperimetry of infinite clusters.
\newblock {\em Probab. Theory Related Fields}, 188(1-2):549--582, 2024.

\bibitem[Lal98]{MR1614583}
Steven~P. Lalley.
\newblock Percolation on {F}uchsian groups.
\newblock {\em Ann. Inst. H. Poincar\'{e} Probab. Statist.}, 34(2):151--177,
  1998.

\bibitem[Lyo00]{MR1757952}
Russell Lyons.
\newblock Phase transitions on nonamenable graphs.
\newblock volume~41, pages 1099--1126. 2000.
\newblock Probabilistic techniques in equilibrium and nonequilibrium
  statistical physics.

\bibitem[Lyo13]{MR3009109}
Russell Lyons.
\newblock Fixed price of groups and percolation.
\newblock {\em Ergodic Theory Dynam. Systems}, 33(1):183--185, 2013.

\bibitem[MO15]{minasyan2015acylindrical}
Ashot Minasyan and Denis Osin.
\newblock Acylindrical hyperbolicity of groups acting on trees.
\newblock {\em Math. Ann.}, 362(3-4):1055--1105, 2015.

\bibitem[MS20]{mathieu2020deviation}
Pierre Mathieu and Alessandro Sisto.
\newblock Deviation inequalities for random walks.
\newblock {\em Duke Math. J.}, 169(5):961--1036, 2020.

\bibitem[MT18]{maher2018random}
Joseph Maher and Giulio Tiozzo.
\newblock Random walks on weakly hyperbolic groups.
\newblock {\em J. Reine Angew. Math.}, 742:187--239, 2018.

\bibitem[NP12]{nachmias2012non-amenable}
Asaf Nachmias and Yuval Peres.
\newblock Non-amenable {C}ayley graphs of high girth have {$p_c<p_u$} and
  mean-field exponents.
\newblock {\em Electron. Commun. Probab.}, 17:no. 57, 8, 2012.

\bibitem[NS81]{newman1981infinite}
C.~M. Newman and L.~S. Schulman.
\newblock Infinite clusters in percolation models.
\newblock {\em J. Statist. Phys.}, 26(3):613--628, 1981.

\bibitem[Osi16]{osin2016acylindrically}
D.~Osin.
\newblock Acylindrically hyperbolic groups.
\newblock {\em Trans. Amer. Math. Soc.}, 368(2):851--888, 2016.

\bibitem[PSN00]{pak2000on-non-uniqueness}
Igor Pak and Tatiana Smirnova-Nagnibeda.
\newblock On non-uniqueness of percolation on nonamenable {C}ayley graphs.
\newblock {\em C. R. Acad. Sci. Paris S\'{e}r. I Math.}, 330(6):495--500, 2000.

\bibitem[Rus81]{russo1981on-the-critical}
Lucio Russo.
\newblock On the critical percolation probabilities.
\newblock {\em Z. Wahrsch. Verw. Gebiete}, 56(2):229--237, 1981.

\bibitem[Sch99]{schonmann1999stability}
Roberto~H. Schonmann.
\newblock Stability of infinite clusters in supercritical percolation.
\newblock {\em Probab. Theory Related Fields}, 113(2):287--300, 1999.

\bibitem[Sis16]{sisto2016quasi-convexity}
Alessandro Sisto.
\newblock Quasi-convexity of hyperbolically embedded subgroups.
\newblock {\em Math. Z.}, 283(3-4):649--658, 2016.

\bibitem[Sis18]{sisto2018contracting}
Alessandro Sisto.
\newblock Contracting elements and random walks.
\newblock {\em J. Reine Angew. Math.}, 742:79--114, 2018.

\bibitem[vdBK85]{berg1985inequalities}
J.~van~den Berg and H.~Kesten.
\newblock Inequalities with applications to percolation and reliability.
\newblock {\em J. Appl. Probab.}, 22(3):556--569, 1985.

\end{thebibliography}

\end{document}